\documentclass{article}
\usepackage{arxiv}
\usepackage[utf8]{inputenc}
\usepackage[T1]{fontenc}
\usepackage{hyperref}       
\usepackage{url}
\usepackage{booktabs}
\usepackage{amsmath,amssymb,amsfonts,amsthm}
\usepackage{nicefrac}
\usepackage{microtype}
\usepackage{lipsum}
\usepackage{graphicx}
\usepackage{natbib}
\usepackage{doi}
\usepackage{color,soul}
\usepackage{kotex}

\newtheorem{theorem}{Theorem}
\newtheorem{lemma}[theorem]{Lemma}

\newtheorem{corollary}[theorem]{Corollary}

\newenvironment{definition}[1][Definition]{\begin{trivlist}
\item[\hskip \labelsep {\bfseries #1}]}{\end{trivlist}}

\DeclareMathOperator*{\argmax}{arg\,max}

\title{Statistical Inference for Random Unknowns 
\\ via Modifications of Extended Likelihood}

\date{}

\author{ 
\href{https://orcid.org/0000-0002-3447-4306}
{\includegraphics[scale=0.06]{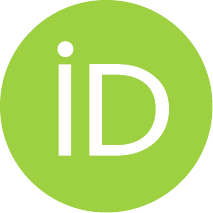}\hspace{1mm}
Hangbin~Lee}\\
Department of Information and Statistics\\
Chungnam National University\\
Daejeon, 34134 \\
\texttt{hangbin221@gmail.com} \\
\And
\href{https://orcid.org/0000-0001-9820-6434}
{\includegraphics[scale=0.06]{orcid.pdf}\hspace{1mm}
Youngjo~Lee} \\
Department of Statistics\\
Seoul National University\\
Seoul, 08826 \\
\texttt{youngjo@snu.ac.kr} \\
}

\renewcommand{\shorttitle}{Statistical Inference for Random Unknowns}

\begin{document}
\maketitle

\begin{abstract}
Fisher's likelihood is widely used for statistical inference for fixed unknowns. This paper aims to extend two important likelihood-based methods, namely the maximum likelihood procedure for point estimation and the confidence procedure for interval estimation, to embrace a broader class of statistical models with additional random unknowns. We propose the new h-likelihood and the h-confidence by modifying extended likelihoods. Maximization of the h-likelihood yields both maximum likelihood estimators of fixed unknowns and asymptotically optimal predictors for random unknowns, achieving the generalized Cram{\'e}r-Rao lower bound. The h-likelihood further offers advantages in scalability for large datasets and complex models. The h-confidence could yield a valid interval estimation and prediction by maintaining the coverage probability for both fixed and random unknowns in small samples. We study approximate methods for the h-likelihood and h-confidence, which can be applied to a general class of models with additional random unknowns.
\end{abstract}


\section{Introduction}
\label{sec:introduction}
Prediction of random unknowns is important in subject-specific inferences
\citep{zeger88} and predictions of future events \citep{hinkley79, lee09}.
However, the likelihood for additional random unknowns has not been
well-established yet. \citet{fisher22} introduced the likelihood
for statistical models with fixed unknowns, which has remained 
the cornerstone of statistical modeling and inference. The maximum likelihood estimators
(MLEs) are asymptotically the best among the consistent estimators,
achieving the Cram{\'{e}}r-Rao lower bound (CRLB). 
The information matrix and associated delta-method 
provide a necessary variance estimator for any function of parameters of interest.
Furthermore, the likelihood principle (LP) indicates
that the likelihood contains all the evidence regarding fixed unknowns
in the data \citep{birnbaum62}. 
As the most significant contribution in the 20th century, \citet{efron98} noted:
\begin{quote}
Fisher's great accomplishment was to provide an optimality standard for
statistical estimation--a yardstick of the best it's possible to do in any
given estimation problem. Moreover, he provided a practical method, maximum
likelihood, that quite reliably produces estimators coming close to the
ideal optimum even in small samples.
\end{quote}
The widespread popularity with extensive usage of the ML procedure
necessitates its extension to a more general class of statistical models for a
practical analysis of large-scaled data. 
Although the LP does not tell us how to proceed,
likelihood procedures should respect the LP to exploit the full data evidence.

The introduction of unobserved random variables (unobservables) allows
flexible modeling at the individual level. In statistical literature,
unobservables naturally appear in various names: random effects, latent  
variables, hidden states, auxiliary variables, factors, unobserved future
observations, potential outcomes, missing data, etc.
A broad class of models with unobservables have been proposed, such as  
generalized linear mixed models (GLMMs; \citeauthor{breslow93}, \citeyear{breslow93}), 
hierarchical generalized linear models (HGLMs; \citeauthor{lee96}, \citeyear{lee96}),
generalized linear latent and mixed models \citep{rabe04, skrondal04, skrondal07}
frailty models for survival data \citep{therneau00}, 
structural equation modeling \citep{joreskog70, joreskog74},
Gaussian process models \citep{williams06,hamelijnck21,zhang23},
item response models \citep{baker04},
models for missing data \citep{little19}, prediction \citep{bjornstad90}, 
and potential outcomes in causality \citep{rubin06}, etc. 
Recently, deep neural networks with unobservables 
are proposed to enhance the prediction accuracy 
\citep{tran20, mandel23, simchoni21, simchoni23, lee23}.
Unobservables can appear in various components of statistical models. 
Random effects in the mean component can account for within-subject correlation in
longitudinal studies \citep{diggle02}, 
smooth spatial data \citep{besag99},
spline-type function fitting \citep{eilers96}, 
and factor analysis \citep{bartholomew87}, etc., 
while random effects in the variance component can account for 
heteroscedasticity \citep{lee17} 
and heavy-tailed distributions for robust modeling \citep{noh07}. 
Modeling of unobservables is the key to these models,  
so there exists a keen interest 
in developing a suitable extension of the likelihood procedures
to general classes of statistical models with additional random unknowns. 
Despite numerous attempts, achieving such a goal has remained unresolved.

For the prediction of unobservables, 
\citet{berger88} formulated an extended likelihood.
However, \citet{bayarri88} demonstrated the difficulties in extending ML procedures 
by showing that maximization of the extended likelihoods 
cannot yield sensible estimators for both fixed and random parameters. 
\citet{lee96} proposed the use of h-likelihood,
which is defined as an extended likelihood
on a particular scale of random parameters.
However, \citet{firth06} noted the ambiguity in the original h-likelihood
and \citet{meng09} discussed the challenges of the extended likelihood approach
by demonstrating the differences between 
estimating fixed unknowns and predicting random unknowns;
see \citet{meng11} and accompanying discussions for more details.
This paper establishes a new h-likelihood and h-confidence 
for prediction of random unknowns by modifying the extended likelihood.
The new h-likelihood and the h-confidence extend
the ML procedure for point estimation of fixed unknowns
and the confidence distribution (CD) for interval estimation of fixed unknowns
to accommodate additional random unknowns.

In Section \ref{sec:review}, 
we underline inferential difficulties of existing likelihood approaches
and provide a rationale for introducing 
the new h-likelihood as a modified extended likelihood.
In Section \ref{sec:hlik}, 
we propose the new h-likelihood at the Bartlizable scale with a modification term
to provide both MLEs for fixed unknowns 
and asymptotic best unbiased predictors (BUPs) for random unknowns.
LP and ELP imply that the proposed h-likelihood
exploits all the evidence about fixed and random parameters in the data.
A numerical study demonstrates the advantage
of the proposed h-likelihood approach in scalability
over the marginal likelihood approach.
In Section \ref{sec:asymptotic}, 
we derive the generalized CRLB (GCRLB)
to accommodate additional random parameters.
Then, we show that the new h-likelihood gives asymptotically optimal 
maximum h-likelihood estimators (MHLEs)
by achieving the GCRLB, along with their consistent variance estimators 
based on the Hessian matrix.
In Section \ref{sec:small}, we investigate scenarios
where consistency may not hold for either fixed or random unknowns. 
Using two illustrative examples, we show that the MHLEs could still
retain asymptotic optimality by achieving the GCRLB.
However, in small samples,
asymptotic optimalities of point estimation and prediction are generally not feasible
and the interval estimators and predictors based on asymptotic normality may not be applicable.
In Section \ref{sec:finite}, 
we propose the h-confidence, which is another modified extended likelihood.
The h-confidence extends the confidence theory for fixed unknowns
to accommodate additional random unknowns, in a frequentist perspective.
We illustrate that the h-confidence can provide 
the predictive distribution (PD) for interval predictions of random unknowns,
which can exactly maintain the coverage probability, even with a single observation.
In Section \ref{sec:approximate}, we address approximate methods
to the h-confidence and the h-likelihood
when they do not have explicit expressions.
Consequently, the proposed methods can be applied 
to a broad class of models with additional random unknowns,
followed by concluding remarks in Section \ref{sec:conclusion}.
All the proofs and technical details are in the Appendix.

\section{Backgrounds}
\label{sec:review}

\subsection{Marginal likelihood approach}
We start with a statistical model
$f_{\boldsymbol{\theta}}(\mathbf{y},\mathbf{u})$
that includes both fixed unknowns 
$\boldsymbol{\theta}=(\theta_{1},\cdots,\theta_{p})^{\intercal} \in \boldsymbol{\Theta}$
and random unknowns
$\mathbf{u}=(u_{1},...,u_{n})^{\intercal} \in \Omega_{\mathbf{u}}$,
where $\mathbf{y}=(y_{1},...,y_{N})^{\intercal} \in \Omega_{\mathbf{y}}$
denotes the observed data.
The MLEs of fixed unknowns $\boldsymbol{\theta}$ can be obtained by
maximizing the marginal likelihood,
\begin{equation} \label{eq:marginal_lik}
L(\boldsymbol{\theta};\mathbf{y})
\equiv f_{\boldsymbol{\theta}}(\mathbf{y})
=\int_{\Omega_{\mathbf{u}}} f_{\boldsymbol{\theta}}(\mathbf{y},\mathbf{u}) \ d\mathbf{u},
\end{equation}
but the marginal likelihood often involves intractable integration.
To obtain the marginal likelihood and the corresponding MLEs, 
various approximate methods have been proposed.
Well-known methods include Gauss-Hermite quadrature and Monte Carlo EM algorithm. 
However, they become computationally demanding
as the dimension of integration $d$ increases \citep{Hedeker06}.
The Laplace approximation (LA) offers an alternative 
and has been widely utilized through various R packages \citep{rue09,Kristensen16,dhglm}. 
However, it may suffer from severe bias, particularly with binary data. 
Additionally, the first-order LA is applicable only when $d=o(N^{1/3})$ \citep{shun95,ogden21},
Numerical studies by \citet{noh07reml} demonstrated that 
the second-order LA can reduce the bias
in the analysis of binary data,
but it is computationally demanding for correlated random effects with $d=O(N)$.
\citet{jin24} highlighted the difficulty in estimating variance.
The EM algorithm \citep{dempster77} gives the MLEs of fixed parameters
without explicitly implementing the intractable integration.
However, the marginal likelihood and EM algorithm 
do not provide inference for random unknowns.
For prediction of random unknowns,
they require an additional procedure,
such as empirical Bayesian or full Bayesian approach,
to compute the posterior mean (BUP) of $\mathbf{u}$,
$$
\textrm{E}(\mathbf{u}|\mathbf{y}) = \int_{\Omega_{\mathbf{u}}} \mathbf{u} \cdot f_{\boldsymbol{\theta}}(\mathbf{u}|\mathbf{y}) \ d\mathbf{u}.
$$
These approaches also suffer from a scalability issue 
in large-scale data with complex correlation structures.
When random unknowns are correlated,
the marginal log-likelihood of the full data $\ell(\boldsymbol{\theta};\mathbf{y})=\log L(\boldsymbol{\theta};\mathbf{y})$ 
is not equivalent to the sum of the marginal likelihood of each observation $\ell(\boldsymbol{\theta};y_{i})$ in general.
This causes a difficulty in the application of stochastic optimization methods 
to the marginal likelihood for neural networks with correlated random effects
\citep{simchoni23}.
Thus, with the growing complexity of models and data in modern applications, 
new methodologies are required to effectively address these challenges.
A numerical study in Section \ref{sec:stochastic} 
demonstrates that this challenge can be overcome 
by the proposed h-likelihood.

\subsection{Extended likelihood approaches}
\label{sec:extended}

For prediction of random unknowns $\mathbf{u}$, 
\citet{berger88} proposed the extended likelihood, 
based on the joint density $f_{\boldsymbol{\theta}}(\mathbf{y},\mathbf{u})$, 
\begin{equation} \label{eq:extended}
L_{e}(\boldsymbol{\theta},\mathbf{u};\mathbf{y})
\equiv L(\boldsymbol{\theta};\mathbf{y}) \ 
L_{p}(\mathbf{u};\boldsymbol{\theta},\mathbf{y}),  
\end{equation}
where $L_{p}(\mathbf{u};\boldsymbol{\theta},\mathbf{y})
\equiv f_{\boldsymbol{\theta}}(\mathbf{u}|\mathbf{y})$ represents
the predictive likelihood \citep{hinkley79, lee17}.
The joint density characterizes the data generation process 
of a fully specified hierarchical model as follows: 
\begin{equation*}
f_{\boldsymbol{\theta}}(\mathbf{y},\mathbf{u})
=f_{\boldsymbol{\theta}}(\mathbf{y}|\mathbf{u})f_{\boldsymbol{\theta}}(\mathbf{u}).
\end{equation*}
This extended likelihood accords with suggestions made by various authors
\citep{henderson59,kaminsky85,butler86,bjornstad96}. 
\citet{bjornstad96} established the extended likelihood principle (ELP), 
stating that all the evidence about $(\boldsymbol{\theta},\mathbf{u})$ 
contained in the data $\mathbf{y}$ 
is captured in the extended likelihood \eqref{eq:extended}.
The use of joint density $f_{\boldsymbol{\theta}}(\mathbf{y},\mathbf{u})$ 
as an extended likelihood $L_{e}(\boldsymbol{\theta},\mathbf{u};\mathbf{y})$ 
requires care due to the Jacobian term associated with random parameters. 
For discrete random parameters, \citet{lee13} showed that 
all the extended likelihoods are h-likelihoods \citep{lee17}, 
whose MHLE gives the optimal prediction of discrete latent status,
because discrete random parameters do not involve the Jacobian term. 
Using the h-likelihood for discrete latent variables,
\citet{lee13} proposed large-scale multiple testing 
and \citet{chee21} proposed non-parametric MLEs for finite mixture frailty models. 
Thus, in this paper, we only consider continuous random parameters.

Estimators from the joint maximization of an extended likelihood 
are sensitive to the parameterization (scale) of random parameters. 
\citet{lee96} noted the importance of the scale of random parameters 
in the extended likelihood for their prediction.
They introduced HGLMs, having the linear predictor 
\begin{equation*}
\boldsymbol{\eta}=\mathbf{X}\boldsymbol{\beta}+\mathbf{Z}\mathbf{v}, 
\end{equation*}
and proposed the use of weak canonical scale $\mathbf{v}$, 
which is additive to fixed effects in the linear predictor 
to form their (original) h-likelihood, 
\begin{equation*}
L_{e}(\boldsymbol{\theta},\mathbf{v};\mathbf{y})
=L(\boldsymbol{\theta};\mathbf{y})
\ L_{p}(\mathbf{v};\boldsymbol{\theta},\mathbf{y})
=f_{\boldsymbol{\theta}}(\mathbf{y},\mathbf{v}).  
\end{equation*}
This covers the GLMM by assuming normal random effects $\mathbf{v}$.
In this formulation, the original h-likelihood 
$L_{e}(\boldsymbol{\theta},\mathbf{v};\mathbf{y})$ 
for statistical inference is the same with the joint density for data generation
$f_{\boldsymbol{\theta}}(\mathbf{y|v})f_{\boldsymbol{\theta}}(\mathbf{v})$ 
from the hierarchical model, 
as the classical likelihood $L(\boldsymbol{\theta};\mathbf{y})$ 
is the same with the marginal density $f_{\boldsymbol{\theta}}(\mathbf{y})$. 
A number of authors \citep{henderson59,gilmour85,harville84,schall91,breslow93,wolfinger93} 
have extended joint maximization algorithms using different justifications,
but they fail to give optimal estimators and predictors in general.

There has been a long-standing disagreement
regarding the general definition of likelihood 
with additional random unknowns \citep{bjornstad96}. 
Various extended likelihoods have been proposed,
either to satisfy the LP or to enable ML procedures
\citep{royall76, butler86, bjornstad90, lee96, pawitan01}.
\citet{bayarri88} showed that all contemporary extended likelihoods 
have difficulty in yielding sensible estimators through the joint maximization.

\subsection{Motivation of modified extended likelihoods}
\label{sec:motivation}

The complete-data likelihood in the EM algorithm is an extended likelihood,
and it includes a modification (E-step)
to produce score equations for the MLEs of fixed parameters (M-step).
The original h-likelihood is also an extended likelihood,
and the application of LA \citep{lee02} can be
interpreted as a modified extended likelihood.
Though it fails to provide optimal estimation and prediction in general, 
it still offers valuable insights on the use of extended likelihood.
To achieve desired statistical properties,
the proposed new h-likelihood and h-confidence are defined as
modified extended likelihoods of the form,
$$
m(\boldsymbol{\theta};\mathbf{y}) L_{e}(\boldsymbol{\theta},\mathbf{u};\mathbf{y}),
$$
where $m(\boldsymbol{\theta};\mathbf{y})$ is a data-dependent modification term.
For LMMs, \citet{lee23}
derived the so-called canonical scale $\mathbf{v}^{c}$
that allows the joint maximization of $L_{e}(\boldsymbol{\theta},\mathbf{v}^{c};\mathbf{y})$ 
to provide the MLE of all the fixed parameters $\boldsymbol{\theta}$.
Their study provides insight that it can be generalized by ensuring that
$L_{p}(\widetilde{\mathbf{v}}^{c};\boldsymbol{\theta},\mathbf{y})
= \max_{\mathbf{v}^c} L_{p}(\mathbf{v}^{c};\boldsymbol{\theta},\mathbf{y})$ 
is independent of $\boldsymbol{\theta}$.

\medskip
\noindent Example 1 (LMM)
Consider an LMM with a linear predictor,
\begin{equation*}
\eta_{ij}=\mu_{ij}
=\textrm{E}(y_{ij}|v_{i})
=\mathbf{x}_{ij}^\intercal\boldsymbol{\beta} +v_{i},
\end{equation*}
where $v_{i} \sim N(0, \lambda)$ and $y_{ij}|v_{i} \sim N(\mu_{ij},\sigma^{2})$
for $i=1,...,n$ and $j=1,...,m$.
For simplicity, 
suppose that the variance components $\sigma^{2}=\lambda=1$ are known.
Then, $\mathbf{v}$ is the weak canonical scale of \citet{lee96}
and their h-likelihood for $(\boldsymbol{\beta},\mathbf{v})$ 
becomes the \citeauthor{henderson59}'s \citeyearpar{henderson59}
joint density of $(\mathbf{y},\mathbf{v})$, 
\begin{equation*}
\ell_{e}(\boldsymbol{\beta},\mathbf{v})
=\log f_{\boldsymbol{\beta}}(\mathbf{y},\mathbf{v})
=-\frac{1}{2}\sum_{i,j}(y_{ij}-\mathbf{x}_{ij}^\intercal\boldsymbol{\beta}-v_{i})^{2}
-\frac{1}{2}\sum_{i}v_{i}^{2}
-\frac{1}{2}\{n(m+1)\log 2\pi\}.
\end{equation*}
Here, joint maximization of $\ell_{e}(\boldsymbol{\beta},\mathbf{v})$
leads to the MLE $\widehat{\boldsymbol{\beta}}$.
For example, in a one-way LMM with 
$\mathbf{x}_{ij}^{\intercal}\boldsymbol{\beta} = \mu_{0}$,
the MLE is $\widehat{\mu}_{0}=\bar{y}=\sum_{i,j}y_{ij}/N$,
where $N=mn$ is the sample size.
However, if the random parameters $\mathbf{v}$ is reparameterized
in terms of log-normal $u_{i}=\exp (v_{i})$, 
the extended likelihood for $(\boldsymbol{\beta},\mathbf{u})$
based on the joint density of $(\mathbf{y},\mathbf{u})$ becomes
\begin{equation*}
\ell_{e}(\boldsymbol{\beta},\mathbf{u})
=\log f_{\boldsymbol{\beta}}(\mathbf{y},\mathbf{u})
=\log f_{\boldsymbol{\beta}}(\mathbf{y},\mathbf{v}) 
+ \sum_{i} \log \left| \frac{dv_{i}}{du_{i}} \right|
=\ell_{e}(\boldsymbol{\beta},\mathbf{v})-\sum_{i}\log u_{i},
\end{equation*}
where $\sum_i \log u_{i}$ is log-determinant of the Jacobian matrix
with respect to the transformation of $\mathbf{v}$.
The two models in terms of $\mathbf{v}$ and $\mathbf{u}$ are equivalent
but maximizing these extended likelihoods leads to different estimators; 
in one-way random effects model,
the joint maximization of $\ell_{e}(\boldsymbol{\beta},\mathbf{u})$ 
leads to $\widehat{\mu}_{0}=\bar{y}+1$.
Therefore, in the absence of a general principle to address this issue, 
the use of joint density as an h-likelihood would be controversial.
When $\lambda$ and $\sigma^2$ are unknown,
joint maximization of $\ell_{e}(\boldsymbol{\theta}, \mathbf{v})$
cannot yield the MLE of variance components $\lambda$ and $\sigma^2$,
where $\boldsymbol{\theta}$ denotes all the fixed parameters.
Recently, \citet{lee23} found that the canonical scale 
$\mathbf{v}^{c}=(v_{1}^{c},...,v_{n}^{c})^{\intercal}$ is
\begin{equation*}
v_{i}^{c}=\sqrt{\frac{\sigma^{2}+m\lambda}{\lambda\sigma^{2}}}\cdot v_{i}.
\end{equation*}
Since
$L_{p}(\widetilde{\mathbf{v}}^{c};\boldsymbol{\theta},\mathbf{y})
=\max f_{\boldsymbol{\theta}}(\mathbf{v}^{c}|\mathbf{y})
=(2\pi)^{-n/2}$
is independent of $\boldsymbol{\theta}$,
this scale leads to the MLEs of all the parameters in $\boldsymbol{\theta}$, 
including variance components $\sigma^{2}$ and $\lambda$.
They extended LMMs 
to deep neural networks and utilized 
$L_{e}(\boldsymbol{\theta},\mathbf{v}^{c};\mathbf{y})$ 
as a single objective function to develop efficient learning algorithms.
\hfill \qed

\section{New h-likelihood}
\label{sec:hlik}

This section proposes a reformulation of the h-likelihood for general cases.
For a moment, let $\ell_{e}(\boldsymbol{\theta},\mathbf{v};\mathbf{y})
=\log L_{e}(\boldsymbol{\theta},\mathbf{v};\mathbf{y})$ 
represent \citeauthor{lee96}'s \citeyearpar{lee96} original h-likelihood.
Here, we define the new h-likelihood
by modifying the original h-likelihood 
with a modification term 
$m(\boldsymbol{\theta};\mathbf{y})=\exp\{a(\boldsymbol{\theta};\mathbf{y})\}$, 
\begin{equation}\label{eq:h-lik}
h(\boldsymbol{\theta},\mathbf{v})
\equiv \ell_{e}(\boldsymbol{\theta},\mathbf{v};\mathbf{y})
+a(\boldsymbol{\theta};\mathbf{y}),
\end{equation}
such that the maximization of $h(\boldsymbol{\theta},\mathbf{v})$ 
yields MLEs of all fixed parameters. 
It is worth noting that such a modification 
$a(\boldsymbol{\theta};\mathbf{y})$ always exists. 
For example, let $a(\boldsymbol{\theta};\mathbf{y})
=-\ell_{p}(\widetilde{\mathbf{v}};\boldsymbol{\theta},\mathbf{y})$, where 
\begin{equation*}
\widetilde{\mathbf{v}}
= \widetilde{\mathbf{v}}(\boldsymbol{\theta})
=\argmax_{\mathbf{v}}h(\boldsymbol{\theta},\mathbf{v};\mathbf{y})
=\argmax_{\mathbf{v}}\ell_{p}(\mathbf{v};\boldsymbol{\theta},\mathbf{y}),
\end{equation*}
then the MHLEs of $\boldsymbol{\theta}$ coincide with the MLEs, because 
\begin{equation}
\label{eq:htilde}
h(\boldsymbol{\theta},\widetilde{\mathbf{v}})
=\ell (\boldsymbol{\theta};\mathbf{y})
=\log f_{\boldsymbol{\theta}}(\mathbf{y}). 
\end{equation}
Note here that, while the MLEs remain invariant 
with respect to any transformation of $\boldsymbol{\theta}$, 
\citet{lee05} demonstrated that the MHLEs of $\mathbf{v}$ 
can only be invariant under a linear transformation.
We consider $\mathbf{v}^{\ast}$ 
as a data-dependent linear transformation of $\mathbf{v}$
for given $\mathbf{y}$: 
\begin{equation*}
\mathbf{v}^{\ast}=\exp \{a(\boldsymbol{\theta};\mathbf{y})\}\cdot \mathbf{v}. 
\end{equation*}
This transformation produces the Jacobian term 
$|\partial \mathbf{v}^{\ast}/\partial \mathbf{v}|
=\exp \{a(\boldsymbol{\theta};\mathbf{y})\}.$ 
Thus, the new h-likelihood can be interpreted 
as an extended likelihood on the data-dependent scale $\mathbf{v}^{\ast}$, 
\begin{equation}\label{eq:h-lik-scale}
h(\boldsymbol{\theta},\mathbf{v})
=\ell (\boldsymbol{\theta};\mathbf{y})
+\ell_{p}(\mathbf{v};\boldsymbol{\theta},\mathbf{y})+a(\boldsymbol{\theta};\mathbf{y})
=\ell(\boldsymbol{\theta};\mathbf{y})+\ell_{p}
(\mathbf{v}^{\ast};\boldsymbol{\theta},\mathbf{y})
=\ell_{e}(\boldsymbol{\theta},\mathbf{v}^{\ast};\mathbf{y}).
\end{equation}
Note that joint maximization of
$\ell_{e}(\boldsymbol{\theta},\mathbf{v}^{\ast};\mathbf{y})$ and 
that of $\ell_{e}(\boldsymbol{\theta},\mathbf{v};\mathbf{y})$ 
provides different estimators for $\boldsymbol{\theta}$ in general. 
If $a(\boldsymbol{\theta};\mathbf{y})
=-\ell_{p}(\widetilde{\mathbf{v}};\boldsymbol{\theta},\mathbf{y})$, 
then $\mathbf{v}^{\ast}=\mathbf{v}^{c}$ is the canonical scale 
in LMMs and $h(\boldsymbol{\theta},\mathbf{v})
=\ell_{e}(\boldsymbol{\theta},\mathbf{v}^{\ast};\mathbf{y})$ 
becomes identical to the h-likelihood of \citet{lee23}
for deep neural network models.

In \citet{lee96}, the joint density for data generation and
their h-likelihood for inference were of an identical form. However, in the
new h-likelihood, the joint density 
\begin{equation*}
f_{\boldsymbol{\theta}}\left( \mathbf{v}^{\ast}\right) 
f_{\boldsymbol{\theta}}\left( \mathbf{y|v}^{\ast}\right) 
\end{equation*}
cannot be used for data generation because both 
$f_{\boldsymbol{\theta}}\left( \mathbf{v}^{\ast}\right) $ 
and $f_{\boldsymbol{\theta}}\left( \mathbf{y|v}^{\ast}\right) $ 
are often intractable. 
For data generation, we use the joint density for the original h-likelihood 
\begin{equation*}
f_{\boldsymbol{\theta}}\left( \mathbf{v}\right) 
f_{\boldsymbol{\theta}}\left( \mathbf{y|v}\right) 
=\ell_{e}(\boldsymbol{\theta},\mathbf{v};\mathbf{y}). 
\end{equation*}
As the canonical scale $\mathbf{v}^{\ast}$ 
for the new h-likelihood can be data-dependent 
(for an example, see Poisson-gamma HGLM in Example 2 of Section \ref{sec:bartlizable}), 
the new h-likelihood for inference is not necessarily
identical to the joint density for data generation: 
\begin{equation*}
h(\boldsymbol{\theta},\mathbf{v})
=\ell_{e}(\boldsymbol{\theta},\mathbf{v}^{\ast};\mathbf{y})
\neq \log f_{\boldsymbol{\theta}}\left( \mathbf{y,v}\right) 
=\ell_{e}(\boldsymbol{\theta},\mathbf{v};\mathbf{y)}. 
\end{equation*}

\medskip
\noindent Example 1 (continued)
In LMMs, the joint maximization of
\citeauthor{henderson59}'s \citeyearpar{henderson59} joint density
$\ell_{e}\left( \boldsymbol{\theta},\mathbf{v}\right)$
cannot provide MLEs for the variance components $\sigma^{2}$ and $\lambda$.
The canonical scale
$v_{i}^{\ast}=v_{i} \exp \{a(\boldsymbol{\theta};\mathbf{y})\}$
with
\begin{equation*}
a(\boldsymbol{\theta};\mathbf{y})
=-\frac{n}{2}\log \left( \frac{\sigma^{2}+m\lambda}{\lambda\sigma^{2}}\right)
\end{equation*}
leads to the new h-likelihood \eqref{eq:h-lik},
\begin{equation*}
h(\boldsymbol{\theta},\mathbf{v})
=\ell_{e}(\boldsymbol{\theta},\mathbf{v})+a(\boldsymbol{\theta};\mathbf{y})
=\ell_{e}(\boldsymbol{\theta},\mathbf{v})
-\frac{n}{2}\log \left( \frac{\sigma^{2}+m\lambda}{\lambda\sigma
^{2}}\right), 
\end{equation*}
whose MHLEs become the MLEs of $\boldsymbol{\theta}$ and the BUPs of $v_{i}$.
Given $\boldsymbol{\theta}$, the MHLE of $v_{i}$ is
\begin{equation*}
\widetilde{v}_{i}
=\frac{\lambda}{\sigma^{2}+m\lambda}
\ \sum_{j=1}^{m} (y_{ij}-\mathbf{x}_{ij}^{\intercal} \boldsymbol{\beta})
=\textrm{E}(v_{i}|\mathbf{y}), 
\end{equation*}
which is called the best linear unbiased predictor (BLUP) \citep{henderson59}. 
\hfill \qed

Suppose that $\widetilde{u}_{i}=$E$(u_{i}|\mathbf{y})$ is the BUP of $u_{i}$. 
For an arbitrary scale $v_{i}=v(u_{i})$, 
the plug-in method of MHLE cannot give the BUP of $v_{i}$,
because $\textrm{E}({v}_{i}|\mathbf{y}) \neq v(\widetilde{u}_{i})$ 
unless $v(\cdot)$ is a linear function. 
The BUP property can be invariant only 
with respect to a data-dependent linear transformation, 
not a non-linear transformation of the random parameter. 
Thus, for an arbitrary scale of random parameters,
MHLEs can achieve the BUP property only asymptotically.

\subsection{Bartlizable scale of random parameters}
\label{sec:bartlizable}

\citet{firth06} noted that the definition of weak canonical scale can be
ambiguous if there is no fixed effect in the linear predictor. 
In this section, we focus on generally defining the scale of random parameters 
to form the h-likelihood.
For notational convenience, we define the h-score and the h-information as
\begin{equation*}
S(\boldsymbol{\theta},\mathbf{v})
=\nabla_{\boldsymbol{\theta},\mathbf{v}} h(\boldsymbol{\theta},\mathbf{v}),
\quad \textrm{and} \quad
I(\boldsymbol{\theta},\mathbf{v})
=-\nabla^2_{\boldsymbol{\theta},\mathbf{v}} h(\boldsymbol{\theta},\mathbf{v})
=\begin{bmatrix}
I_{\boldsymbol{\theta}\boldsymbol{\theta}} 
& I_{\boldsymbol{\theta}\mathbf{v}} \\ 
I_{\mathbf{v}\boldsymbol{\theta}} & I_{\mathbf{v}\mathbf{v}}
\end{bmatrix},
\end{equation*}
and denote the observed and expected h-information by 
$\widehat{I}=I(\widehat{\boldsymbol{\theta}},\widehat{\mathbf{v}})$ and
$\mathcal{I}_{\boldsymbol{\theta}}
=\textrm{E}\left\{ I(\boldsymbol{\theta},\mathbf{v})\right\}$,
respectively.
\citet{meng09} showed that \citeauthor{lee96}'s \citeyearpar{lee96} 
h-likelihood $\ell_{e}(\boldsymbol{\theta},\mathbf{v})$ 
with the weak canonical scale is `Bartlizable', if it exists,
satisfying the first and second Bartlett identities: 
\begin{equation*}
\textrm{E}\left[ \nabla_{\boldsymbol{\theta},\mathbf{v}} \ell_{e}(\boldsymbol{\theta},\mathbf{v}) \right] = 0
\ \ \textrm{and}\ \ 
\textrm{E}\left[ 
\left\{\nabla_{\boldsymbol{\theta},\mathbf{v}} \ell_{e}(\boldsymbol{\theta},\mathbf{v})\right\}
\left\{\nabla_{\boldsymbol{\theta},\mathbf{v}} \ell_{e}(\boldsymbol{\theta},\mathbf{v})\right\}^{\intercal}
+\nabla^2_{\boldsymbol{\theta},\mathbf{v}} 
\ell_{e}(\boldsymbol{\theta},\mathbf{v}) \right] =0.
\end{equation*}
Identifying a Bartlizable scale is important for variance estimation;
otherwise, it could fail to be non-negative definite.
\citet{meng09} introduced a sufficient condition 
for Bartlizability that $f_{\boldsymbol{\theta}}(\mathbf{v})=0$ 
and $\nabla_{\mathbf{v}} f_{\boldsymbol{\theta}}(\mathbf{v})=0$ 
at the boundary $\partial \Omega_{\mathbf{v}}$ 
of the support $\Omega_{\mathbf{v}}$ of $\mathbf{v}$. 
In this paper, 
we adopt the concept of Bartlizable for 
$h\left( \boldsymbol{\theta},\mathbf{v}\right)$ 
and define it as the new h-likelihood
when the scale $\mathbf{v}$ satisfies this property.

\begin{definition}
$h(\boldsymbol{\theta},\mathbf{v})$ in the form of \eqref{eq:h-lik} 
is called the h-likelihood if it is Bartlizable, 
satisfying the first and second Bartlett identities as follows: 
\begin{equation*}
\textrm{E}\left[ S(\boldsymbol{\theta},\mathbf{v})\right] =0
\quad \textrm{and} \quad 
\textrm{E}\left[ S(\boldsymbol{\theta},\mathbf{v})
S(\boldsymbol{\theta},\mathbf{v})^{\intercal}
-I(\boldsymbol{\theta},\mathbf{v})\right] =0.
\end{equation*}
\end{definition}

The Bartlizable scale clearly resolves the ambiguity 
in defining the scale of $\mathbf{v}$ for the h-likelihood.
Since it does not require a linear predictor, 
it can be defined in the broader class of models 
than the weak-canonical scale.
The following lemma ensures the existence of a Bartlizable scale
and provides guidance for identification.

\begin{lemma}
\label{lem:bartlett} (a) For any continuous random parameters 
$(u_{1},...,u_{n})^{\intercal}$, 
there exists a one-to-one transformation to the Bartlizable scale 
$v_{i}=g_{v}(u_{i})$.
(b) The scale $\mathbf{v}=(v_{1},...,v_{n})^{\intercal}$ is Bartlizable if 
\begin{equation*}
f_{\boldsymbol{\theta}}(\mathbf{v}|\mathbf{y})=0
\quad \textrm{and}\quad 
\nabla_{\mathbf{v}} f_{\boldsymbol{\theta}}(\mathbf{v}|\mathbf{y})
=\mathbf{0}\quad \textrm{for all }\mathbf{v}\in \partial \Omega_{\mathbf{v}}.
\end{equation*}
(c) The scale $\mathbf{v}$ is Bartlizable if 
$f_{\boldsymbol{\theta}}(\mathbf{v})$ is differentiable 
for any $\mathbf{v}\in \Omega_{\mathbf{v}}=\mathbb{R}^{n}$.
\end{lemma}

\subsection{Geometric Insights into h-Likelihood under LP and ELP}
In accordance with LP \citep{birnbaum62}, 
the marginal likelihood $L(\boldsymbol{\theta};\mathbf{y})$ 
exploits all the evidence about $\boldsymbol{\theta}$ in the data.
Similarly, in accordance with ELP \citep{bjornstad96},
the predictive likelihood $L_{p}(\mathbf{u};\boldsymbol{\theta},\mathbf{y})$ 
exploits all the evidence about $\mathbf{u}$ in the data,
implying that the prediction of random parameters 
should be based solely on 
$L_{p}(\mathbf{u};\boldsymbol{\theta},\mathbf{y})
=f_{\boldsymbol{\theta}}(\mathbf{u}|\mathbf{y})$.
This section demonstrates that the proposed h-likelihood fully exploits
all the evidence about fixed and random unknowns
by illustrating the relationship among 
the marginal likelihood, predictive likelihood, and h-likelihood.

Figure \ref{fig:hlik3D} illustrates the new h-likelihood  
$H(\boldsymbol{\theta},\mathbf{v})=\exp\{h(\boldsymbol{\theta},\mathbf{v})\}$
for joint inferences of $\boldsymbol{\theta}$ and $\mathbf{v}$.
The profile h-likelihood of $\boldsymbol{\theta}$
becomes the marginal likelihood,
\begin{equation*} \label{eq:hlik-maxv}
H(\boldsymbol{\theta},\widetilde{\mathbf{v}}(\boldsymbol{\theta}))
=L(\boldsymbol{\theta};\mathbf{y}).
\end{equation*}
Thus, LP implies that the h-likelihood 
exploits all the evidence about fixed unknowns
to give the MLE 
$\widehat{\boldsymbol{\theta}} = \argmax_{\boldsymbol{\theta}}L(\boldsymbol{\theta};\mathbf{y})
= \argmax_{\boldsymbol{\theta}} H(\boldsymbol{\theta},\widetilde{\mathbf{v}}(\boldsymbol{\theta}))$.
Furthermore,
\begin{equation*} \label{eq:hlik-maxtheta}
H(\widehat{\boldsymbol{\theta}},\mathbf{v})
\propto L_p(\mathbf{v}; \widehat{\boldsymbol{\theta}}, \mathbf{y})
\propto L_e(\widehat{\boldsymbol{\theta}}, \mathbf{v}),
\end{equation*}
where $L_p(\mathbf{v}; \widehat{\boldsymbol{\theta}}, \mathbf{y})$
is a profile predictive likelihood
and $L_e(\widehat{\boldsymbol{\theta}}, \mathbf{v})$
is an extended likelihood for $\mathbf{v}$ of \citet{bjornstad96}.
Thus, the ELP implies that the h-likelihood exploits 
all the evidence about random unknowns.
The Fisher likelihood serves as the upper bound 
when projecting the h-likelihood onto the parameter space $\boldsymbol{\Theta}$, while the profile predictive likelihood is the upper bound 
when projecting the h-likelihood onto the support of latent variables $\Omega_{\mathbf{v}}$.

\begin{figure}[tbp]
\centering
\includegraphics[width=0.6\linewidth]{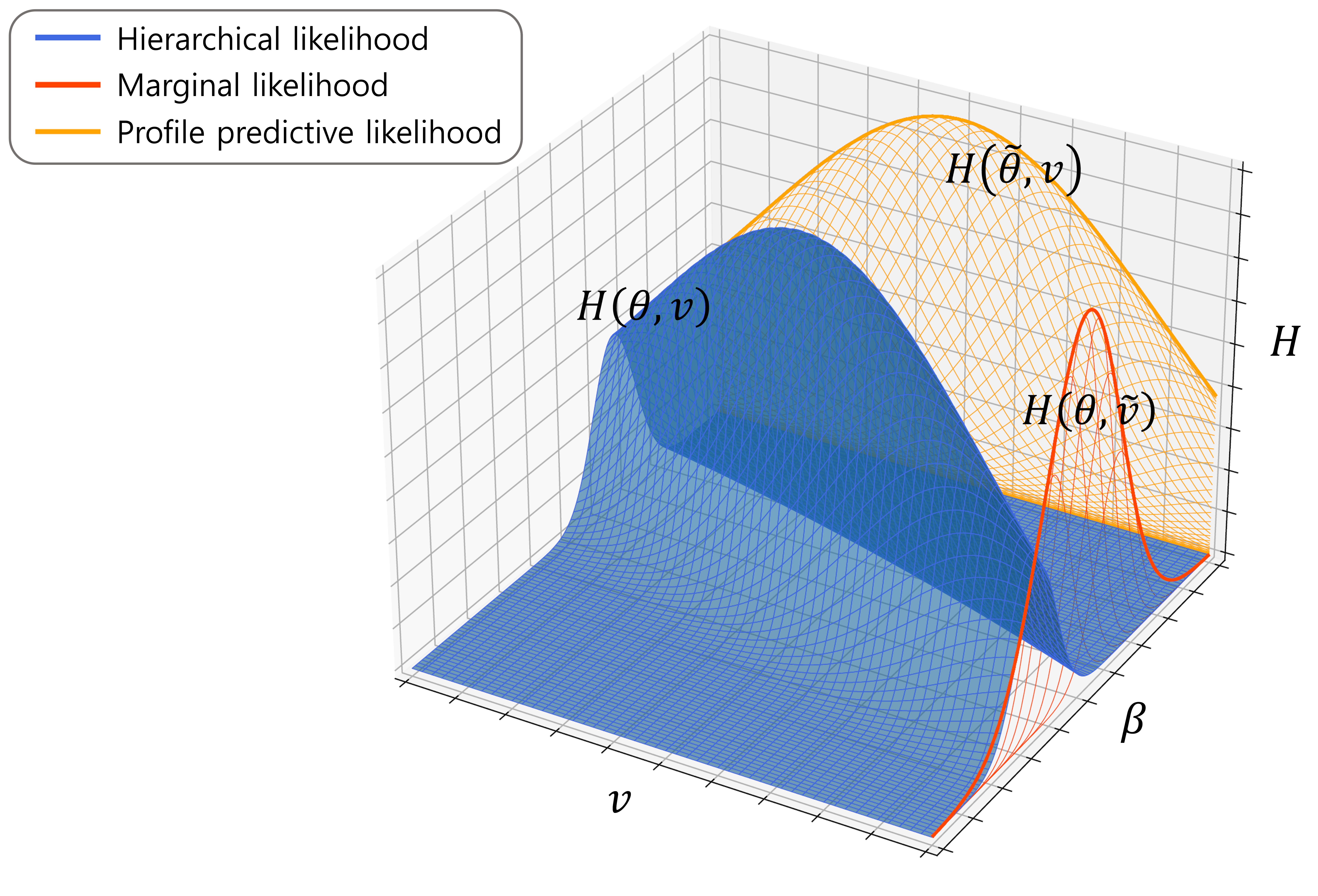}
\caption{The h-likelihood (blue), the marginal likelihood (red), and the
profile predictive likelihood (orange) when $y_{i}\sim N(x_{i} \theta +v,1)$
and $v\sim N(0,1)$ for $i=1,..., 2000$.}
\label{fig:hlik3D}
\end{figure}

\medskip
\noindent Example 1 (continued)
In LMMs, Lemma \ref{lem:bartlett}
implies that $\mathbf{v}$ is Bartlizable
and we showed that the h-likelihood $h(\boldsymbol{\theta},\mathbf{v})$ gives 
the MLE of $\boldsymbol{\theta}$ and BLUP of $v_{i}$.
Note that $\mathbf{u}$ is not weak canonical,
but it is Bartlizable to give an h-likelihood, 
\begin{equation*}
h(\boldsymbol{\theta},\mathbf{u})
=h\left( \boldsymbol{\theta},\mathbf{v}\right) -\sum_{i}\log u_{i}, 
\end{equation*}
which now gives MLEs of $\boldsymbol{\theta}$, 
whereas the extended likelihood 
$\ell_{e}(\boldsymbol{\theta},\mathbf{u})$ cannot give the MLEs.
This implies that the Bartlizable scale is 
more general than the weak canonical scale.
Here $h(\boldsymbol{\theta},\mathbf{u})$ gives the MHLE of $v_{i}$,
\begin{equation*}
\widetilde{v}_{i}^{\prime}
=\frac{\lambda}{\sigma^{2}+m\lambda}
\ \sum_{j=1}^{m} (y_{ij}-\mathbf{x}_{ij}^{\intercal} \boldsymbol{\beta})
- \frac{\sigma^{2}\lambda}{\sigma^{2}+m\lambda},
\end{equation*}
which is not the BUP of $v_{i}$, 
but it gives the BUP of $w_{i}=1/u_{i}^{2}$, 
\begin{equation*}
\widetilde{w}_{i}^{\prime}
=\exp \left[ \frac{-2\lambda
\sum_{j=1}^{m} (y_{ij}-\mathbf{x}_{ij}^{\intercal} \boldsymbol{\beta})
+2\sigma^{2}\lambda}
{\sigma^{2}+m\lambda}\right] 
=\textrm{E}(w_{i}|\mathbf{y}). 
\end{equation*}
Thus, $h(\boldsymbol{\theta},\mathbf{u})$ would be desirable for prediction
of $1/u_{i}^{2}$. This implies that there exist multiple h-likelihoods,
yielding BUPs for different scales of random unknowns in finite samples.
\hfill \qed

\medskip
\noindent Example 2 (Poisson-gamma HGLM)
Consider a Poisson-gamma HGLM with a log link,
\begin{equation*}
\eta_{ij}
=\log \mu_{ij}
=\log \{ \textrm{E}(y_{ij}|u_{i}) \}
=\mathbf{x}_{ij}^{\intercal}\boldsymbol{\beta}+\log u_{i},
\end{equation*}
where $y_{ij}|v_{i}\sim \textrm{Poisson}(\mu_{ij})$
and $u_{i}\sim \textrm{Gamma}(\alpha,\alpha)$
with $\textrm{E}(u_{i})=1$ and $\textrm{Var}(u_{i})=\lambda=1/\alpha$. 
The vector of fixed parameters $\boldsymbol{\theta}$ contains
$\boldsymbol{\beta}=(\beta_{0},\beta_{1},...,\beta_{p})^{\intercal}$ 
and the variance component $\alpha$. 
Here, $v_{i} = \log u_{i}$ is Bartlizable,
leading to the new h-likelihood,
$$
h(\boldsymbol{\theta},\mathbf{v})
= \ell_{e}(\boldsymbol{\theta},\mathbf{v}) + \sum_{i=1}^{n} a_i(\alpha;\mathbf{y})
=\sum_{i,j} (y_{ij} \log \mu_{ij} - \mu_{ij})
+ \sum_{i} \alpha (v_{i} - e^{v_{i}})
+ n\{\alpha \log \alpha - \log \Gamma(\alpha)\}
+ \sum_{i=1}^{n} a_i(\alpha;\mathbf{y}),
$$
where $
a_{i}(\alpha ;\mathbf{y)}
=(y_{i+}+\alpha)+\log \Gamma(y_{i+}+\alpha)-(y_{i+}+\alpha)\log(y_{i+}+\alpha)
$ and $y_{i+}=\sum_{j}y_{ij}$.
It provides the MLEs of all the parameters in $\boldsymbol{\theta}$
and the BLUP of $u_{i}$ \citep{lee96}.
The h-likelihood is equivalent to the use of
data-dependent canonical scale of \citet{lee24count},
\begin{equation*}
v_{i}^{c}=\frac{\Gamma(y_{i+}+\alpha)
\cdot \exp (y_{i+}+\alpha)}{(y_{i+}+\alpha)^{y_{i+}+\alpha}}
\cdot v_{i},
\end{equation*}
where
$L_{p}(\widetilde{\mathbf{v}}^{c};\boldsymbol{\theta},\mathbf{y})
=\max f_{\boldsymbol{\theta}}(\mathbf{v}^{c}|\mathbf{y})=0$
is independent of $\boldsymbol{\theta}$.
In the Poisson-gamma HGLM,
$\mathbf{u}$-scale is not Bartlizable
when $\alpha \leq 1$, since
\begin{equation*}
\textrm{E}\left[ \nabla_{u_i} h(\boldsymbol{\theta},\mathbf{u})\right] 
=\textrm{E}\left[ \textrm{E}\left[ u_i^{-1}(y_{i+}+\alpha -1)-(\mu_{i+}+\alpha)
|\mathbf{u}\right] \right] 
=(\alpha -1)\textrm{E}(u_{i}^{-1})-\alpha =\infty.
\end{equation*}
This gives a predictor $\widetilde{u}_{i}=0$ when $y_{i+}<1-\alpha$, 
which is not the BUP.
\hfill \qed

\begin{figure}[tbp]
\centering
\includegraphics[width=\linewidth]{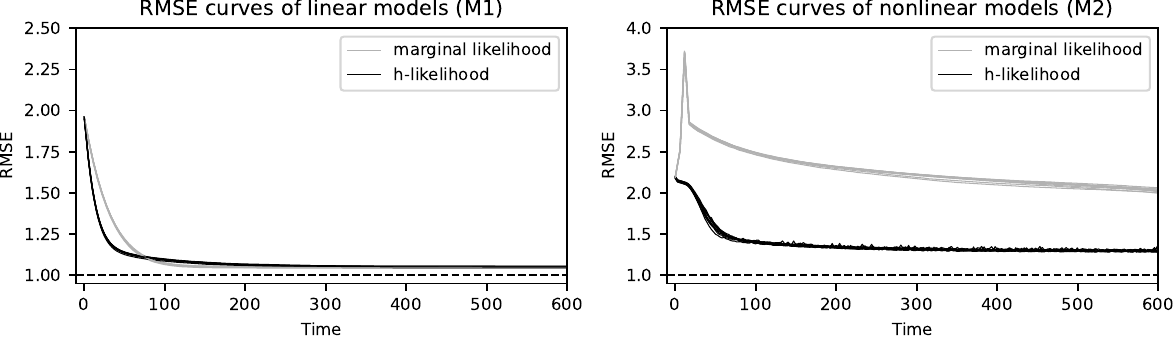}
\caption{RMSE curves of marginal likelihood (gray) 
and h-likelihood (black) approaches from 10 repetitions.
\label{fig:stochastic}}
\end{figure}

\subsection{Numerical studies for scalability}
\label{sec:stochastic}

In LMMs, both the marginal likelihood and the h-likelihood
yield MLEs for fixed parameters, producing identical inferences.
However, in neural networks with random effects, 
the estimators from the two likelihoods can be different
since they may not achieve the global optimum.
When random effects are correlated, 
stochastic optimization for large-scale data
may not be applicable to the marginal likelihood \citep{simchoni21, simchoni23},
thereby posing scalability challenges for current computational approaches for GLMMs \citep{lee24count},
whereas it can be straightforwardly applied to the h-likelihood.
To compare the performance of the two methods,
this section presents a numerical study based on the following two scenarios.

\begin{itemize}
\item M1 (linear): 
$y_{ij}$ is generated from $N(\mu_{ij},1)$ for $i=1,...,5000$ and $j=1,...,20$, where
\begin{align*}
\mu_{ij} = 1 + \frac{1}{5} \sum_{p=1}^{50} x_{ijp} + v_{i}.
\end{align*}
The following LMM is trained 
using $80\%$ of a dataset of size $N=100K$,
\begin{align*}
\text{M1: }
\mu_{ij} = \beta_{0} + \mathbf{x}_{ij}^\intercal \boldsymbol{\beta} + v_{i},
\end{align*}
where $\boldsymbol{\beta}$ denotes the regression coefficients.

\item M2 (nonlinear): 
$y_{ij}$ is generated from $N(\mu_{ij},1)$ for $i=1,...,5000$ and $j=1,...,20$, where
\begin{align*}
\mu_{ij} = \sum_{p=1}^{10} \frac{1}{1+x_{ijp}^2} + \sum_{p=11}^{20} \cos x_{ijp}
+ \left( \sum_{p=21}^{30} \cos x_{ijp} \right) \left( \sum_{p=31}^{40} \cos x_{ijp} \right)
+ v_{i}.
\end{align*}
The following neural network is trained 
using $80\%$ of a dataset of size $N=100K$,
\begin{align*}
\text{M2: }
\mu_{ij} = \text{NN}(\mathbf{x}_{ij}; \mathbf{w}) + v_{i},
\end{align*}
where $\mathbf{w}$ denotes a vector of all the weights in the network,
consisting of 4 hidden layers 
with (100, 50, 25, 12) number of nodes and ReLU activation function.
\end{itemize}

\noindent 
For both scenarios, $\mathbf{v} = (v_1,...,v_n)^\intercal$ 
is generated from a multivariate normal distribution  
with zero mean and an AR(1) covariance ($\rho=0.7$),
and $x_{ijp}$'s are generated from $\text{Uniform}(-\pi/2, \pi/2)$.
Each model is trained using either the marginal likelihood or the proposed h-likelihood,
optimized with the Adam optimizer (learning rate $0.001$ and batch size $1024$)
and implemented in Python with TensorFlow \citep{tensorflow}.

Figure \ref{fig:stochastic} shows the root mean squared error (RMSE) curves of the two methods over 10 repetitions. In scenario M1 (linear), RMSEs from both approaches converge to the same global optimum.
However, in scenario M2 (nonlinear), RMSEs from the two approaches converge to different optima, and those from the h-likelihood (black) decrease significantly faster than those of the marginal likelihood approach (gray). This highlights the advantages of h-likelihood in scalability for complex models with large datasets.

\section{Asymptotic properties of MHLEs}
\label{sec:asymptotic}

\subsection{Generalized Cram{\'e}r-Rao lower bound}
\label{sec:GCRLB}

Suppose that $\boldsymbol{\zeta}(\boldsymbol{\theta},\mathbf{v})$ 
is an arbitrary function of $(\boldsymbol{\theta},\mathbf{v})$ 
and $\widehat{\boldsymbol{\zeta}}(\mathbf{y})$ is an unbiased estimator of 
$\boldsymbol{\zeta}(\boldsymbol{\theta},\mathbf{v})$ such that 
\begin{equation*}
\textrm{E}\left[ \widehat{\boldsymbol{\zeta}}(\mathbf{y})
-\boldsymbol{\zeta}(\boldsymbol{\theta},\mathbf{v})\right] =0. 
\end{equation*}
Then, the following theorem presents the GCRLB, 
which generalizes the CRLB for fixed unknowns 
and the Bayesian CRLB \citep{van68} for random unknowns.

\begin{theorem}
\label{thm:hcrb} Let $
\mathcal{Z}_{\boldsymbol{\theta}}
=\mathrm{{E}\left[ \nabla_{\boldsymbol{\theta},\mathbf{v}} \boldsymbol{\zeta}(\boldsymbol{\theta},\mathbf{v})\right]}.
$
Under regularity conditions R1-R6 in Section \ref{app:proof} of the Appendix, 
\begin{equation}
\textrm{Var}\left[ \widehat{\boldsymbol{\zeta}}(\mathbf{y})
-\boldsymbol{\zeta}(\boldsymbol{\theta},\mathbf{v})\right] 
\geq \mathcal{Z}_{\boldsymbol{\theta}}
\mathcal{I}_{\boldsymbol{\theta}}^{-1}
\mathcal{Z}_{\boldsymbol{\theta}}^{\intercal},
\label{eq:hcrb}
\end{equation}
where the matrix inequality $A\geq B$ means that $A-B$ 
is positive semi-definite.
\end{theorem}

When random parameters $\mathbf{v}$ are known,
the GCRLB \eqref{eq:hcrb} becomes the CRLB.
When fixed parameters $\boldsymbol{\theta}$ are known, 
the GCRLB \eqref{eq:hcrb}
becomes the Bayesian CRLB, since 
\begin{equation*}
\nabla_{\mathbf{v}}^{2}h(\boldsymbol{\theta},\mathbf{v})
=\nabla_{\mathbf{v}}^{2}\ell_{e}(\boldsymbol{\theta},\mathbf{v})
=\nabla_{\mathbf{v}}^{2}\ell_{p}(\mathbf{\mathbf{v};\boldsymbol{\theta},\mathbf{y}}).
\end{equation*}
Note that the Bayesian optimality can be achieved with the BUP, 
$\textrm{E}[\boldsymbol{\zeta}(\boldsymbol{\theta},\mathbf{v})|\mathbf{y}]$,
even in small samples; 
however, obtaining it in an explicit form can often be intractable.
The h-likelihood procedure does not require an explicit form of
$\textrm{E}[\boldsymbol{\zeta}(\boldsymbol{\theta},\mathbf{v})|\mathbf{y}]$,
since the MHLE is not the conditional mean but the mode.

\subsection{Asymptotic normality of MHLEs}
\label{sec:normality}
This section shows that the new h-likelihood
provides an optimality standard for both fixed and random unknowns
for establishing a proper extension of the ML procedure.
Due to the wide scope of statistical models with additional random unknowns,
asymptotic theory should be developed on a model-by-model basis.
For simplicity of discussion,
this section focuses on the HGLMs with linear predictors of the form 
$\boldsymbol{\eta}=\mathbf{X}\boldsymbol{\beta}+\mathbf{Z}\mathbf{v}$, 
including GLMMs \citep{breslow93} with random intercepts or random slopes.
Besides the sample size $N$, we should consider the number of random unknowns $n$
and the minimum number of observed data associated with each random unknown $m$.
Note that $m$ and $n$ can be represented differently depending on the model. 
For example, $m=0$ in the missing data problem in Section \ref{sec:small}.
In Example 1, 
the number of random unknowns is represented by
$n=\mathrm{{rank}(\mathbf{Z})}$
and the minimum number of observed data associated with each random unknown 
is represented by
$m=\mathrm{min}\{\mathrm{{diag}(\mathbf{Z}^{\intercal}\mathbf{Z})\}}$.
In this section, we establish asymptotic properties of MHLEs
when both $n\to\infty$ and $m\to\infty$.
Asymptotic properties of the MHLE for $\boldsymbol{\theta}$ are well established
because $\widehat{\boldsymbol{\theta}}$ is the MLE in the absence of $\mathbf{v}$.
Separately, asymptotic properties of the MHLE for $\mathbf{v}$ are also well established
because $\widehat{\mathbf{v}}$ is the maximum a posteriori (MAP) in the absence of $\boldsymbol{\theta}$.
The following theorem shows the asymptotic normality of the MHLEs
for both fixed and random unknowns, achieving the GCRLB.

\begin{theorem}
\label{thm:h-asymptotic} Under regularity conditions R1-R6,
MHLEs $\widehat{\boldsymbol{\theta}}$ and $\widehat{\mathbf{v}}$
are consistent and asymptotically normal,
\begin{equation*}
(\widehat{\boldsymbol{\theta}}-\boldsymbol{\theta}, 
\ \widehat{\mathbf{v}}-\mathbf{v})^\intercal
\overset{d}{\to }N\left( \mathbf{0},\ \mathcal{I}_{\boldsymbol{\theta}}^{-1}\right),
\end{equation*}
and the variance can be consistently estimated by 
$\widehat{I}^{-1}=I(\widehat{\boldsymbol{\theta}},\widehat{\mathbf{v}})^{-1}$
as $m\to\infty$ and $n\to\infty$.
\end{theorem}

Theorem \ref{thm:hcrb} and \ref{thm:h-asymptotic} 
show the asymptotic optimality of the MHLEs
from joint maximization of the h-likelihood.
In HGLMs with the linear predictor, 
\begin{equation*}
\eta_{ij}=g(\mu_{ij})=\mathbf{x}_{ij}^{\intercal}\boldsymbol{\beta}+v_{i},
\end{equation*}
we want to find a predictor $\widetilde{v}_{i}$ 
that gives the BUP of conditional mean $\mu_{ij}$ in finite samples, 
\begin{equation*}
\widetilde{\mu}_{ij}
=g^{-1}(\mathbf{x}_{ij}^{\intercal}\boldsymbol{\beta}
+\widetilde{v}_{i})=\textrm{E}(\mu_{ij}|\mathbf{y}).  
\end{equation*}
In Section \ref{app:BUP} of the Appendix,
we show that such $\widetilde{v}_{i}$ may not always exist,
for example, in binomial HGLMs with the logit link.
However, Theorem \ref{thm:h-asymptotic} implies that
the MHLEs provide asymptotic BUP for the conditional mean $\mu_{ij}$.
Proof of Theorem \ref{thm:h-asymptotic} is in
Section \ref{proof:h-asymptotic} of the Appendix.

\subsection{Numerical studies for convergence}

We conduct numerical studies to investigate
the properties of MHLEs for the LMM in Example 1
with the linear predictor
$
\eta_{ij} = \beta_0 + \beta_1 x_{ij} + v_{i},
$
where $x_{ij}$ is generated from Uniform$[-0.5,0.5]$
and all the true parameters are set to be 0.5.
To see the convergence, 
we let $m$ or $n$ increase in $\{5, 20, 80\}$.
Figure \ref{fig:consistency_nn} shows the box-plots 
for the estimation errors of each parameter from 500 replications.
The MHLE of $\beta_{1}$ for the within-cluster covariate $x_{ij}$
and the MHLE of the within-cluster variance $\sigma^{2}$
are consistent when either $m$ or $n$ increases.
However, the MHLE of an intercept $\beta_{0}$
and the MHLE of the between-cluster variance $\lambda$
are consistent only when $n$ increases.
Furthermore, the MHLE of a random parameter $v_1$
is consistent when both $m$ and $n$ increase:
$\widetilde{v}_{1}-v_{1}\overset{p}{\to}0$ when $m\to\infty$,
whereas $\widehat{v}_{1}-\widetilde{v}_{1}=\widetilde{v}_{1}(\widehat{\boldsymbol{\theta}})-\widetilde{v}_{1}(\boldsymbol{\theta})\overset{p}{\to}0$
when $\widehat{\boldsymbol{\theta}}\overset{p}{\to}\boldsymbol{\theta}$ as $n\to\infty$.
Thus, $\widehat{v}_{1}-v_{1}\to 0$ when both $m$ and $n$ increase.
Similar numerical results for the Poisson-gamma HGLM
are presented in Section \ref{app:PG-HGLM} of the Appendix.
Consequently, consistency of $\widehat{\boldsymbol{\theta}}$ requires $n\to\infty$,
whereas that of $\widehat{\mathbf{v}}$ requires both $m\to\infty$ and $n\to\infty$.

\begin{figure}[tbp]
\centering
\includegraphics[width=\linewidth]{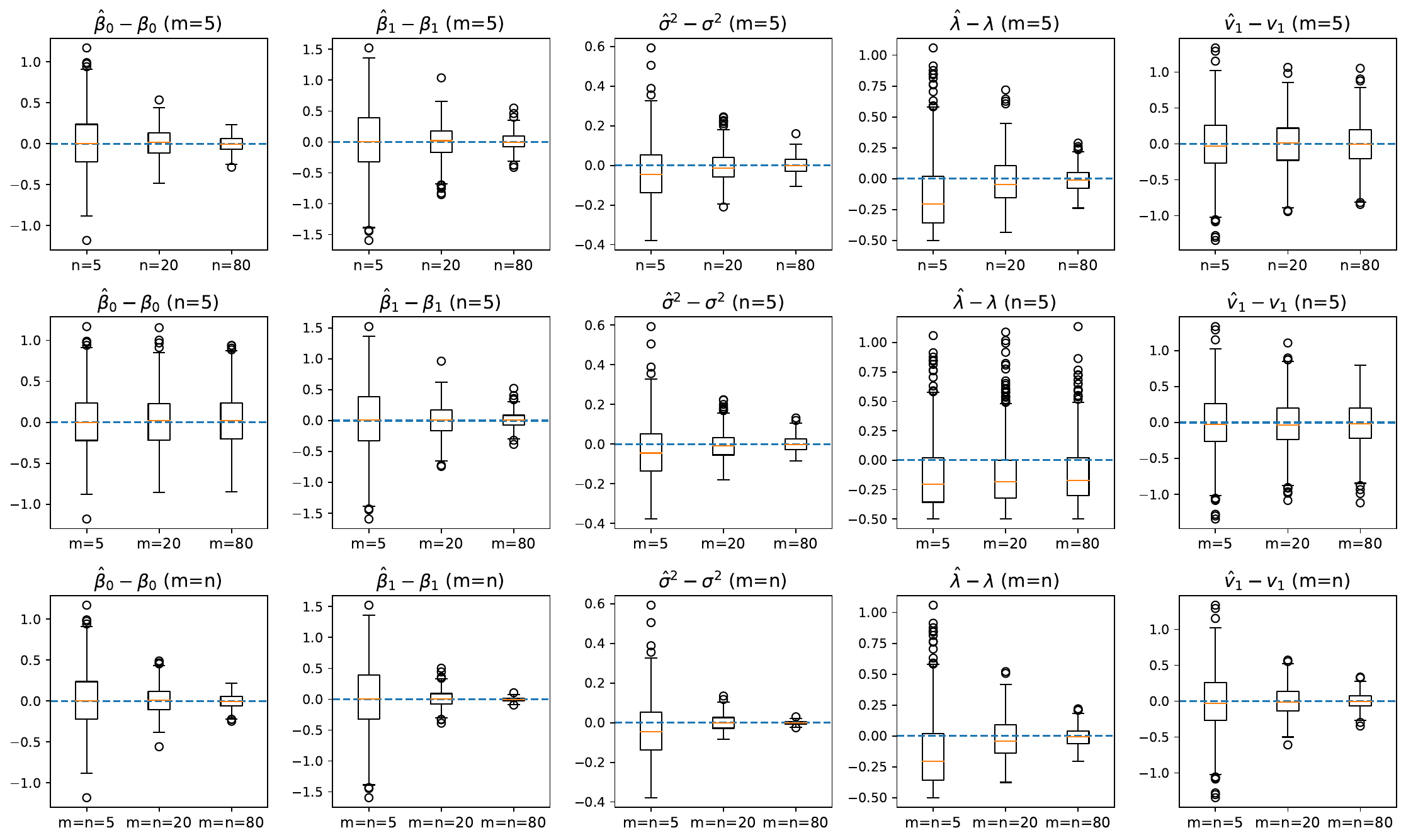}
\caption{Box-plots for the estimation errors of fixed and random unknowns in LMM. \label{fig:consistency_nn}}
\end{figure}

\section{Small sample properties of MHLEs} 
\label{sec:small}

Small sample properties of MLEs for fixed unknowns 
are fascinating but still controversial,
as the MLE loses less information than competing asymptotically
efficient estimators in small samples \citep{efron98}.
In Section \ref{sec:asymptotic}, 
we have shown the asymptotic optimality of the MHLEs,
where both $m\to\infty$ and $n\to\infty$.
If either $m$ or $n$ is small 
as in Figure \ref{fig:consistency_nn},
asymptotic properties of MHLEs may not hold.
This section demonstrates that MHLEs would inherit small sample properties of MLEs
when either $m$ or $n$ is small, 
using the two examples developed for illustrating inferential difficulties
of likelihood-based procedures.

\citet{meng09} pointed out that predictors of missing data 
cannot be consistent even if $\widehat{\boldsymbol{\theta}}$ is consistent,
\begin{equation*}
\widehat{\boldsymbol{\theta}}-\boldsymbol{\theta}=o_{p}(1)
\quad \textrm{and} \quad 
\widehat{\mathbf{v}}-\mathbf{v}=O_{p}(1),
\end{equation*}
and $\widehat{\mathbf{v}}-\mathbf{v}$
would not be asymptotically normal,
as $\mathbf{v}$ can be a non-normal random variable. 
Note here that $m=0$ in missing data problems,
since there are no observations associated with the missing value.
In longitudinal studies, $m$ is often small while the
number of subjects $n$ tends to be large.
When $n\to\infty$ but $m$ is finite,
as we mentioned in Section \ref{sec:asymptotic},
consistency and asymptotic normality 
may not be ensured in the prediction of random unknowns.
In Section \ref{sec:approximate_cd}, we show how to make 
asymptotically correct interval prediction for $\mathbf{v}$
as long as $\widehat{\boldsymbol{\theta}} \overset{p}{\to} \boldsymbol{\theta}$.

The difficulties in the prediction of missing data arise from small $m$.
Asymptotic theories of the MLEs for fixed unknowns can encounter 
similar difficulties when $n$ is small.
We present an example that as $m\to\infty$,
\begin{equation*}
\widehat{\mathbf{v}}-\mathbf{v}=o_{p}(1)
\quad \textrm{and} \quad 
\widehat{\boldsymbol{\theta}}-\boldsymbol{\theta}=O_{p}(1).
\end{equation*}
In this case, even though $\widehat{\mathbf{v}} - \mathbf{v}$ is consistent and asymptotically normal,
the MLEs are neither consistent nor asymptotically normal.
However, we can still show that 
the MHLEs could remain asymptotically optimal in both examples, 
even if either $m$ or $n$ is small.
In Section \ref{sec:finite}, we further study
how to construct an exact interval estimation and prediction when both $m$ and $n$ are small.

\subsection{Missing data problems with $m=0$}
\label{sec:small_m}

Suppose that $\mathbf{y}=(y_{1},...,y_{n})^{\intercal}$ 
is a vector of independent samples from $f_{\theta}(y_{i})$ 
with a fixed unknown $\theta$ 
and $u=y_{n+1}$ is an unobserved future outcome or missing data. 
As \citet{meng09} noted, the h-likelihood $h(\theta,v)$ yields MHLEs with
\begin{equation*}
\widehat{\theta}-\theta=o_{p}(1)
\quad \textrm{and}\quad 
\widehat{v}-v=O_{p}(1). 
\end{equation*}
Let $\varepsilon =\widetilde{v}-v$, 
then $\widehat{v}-v$ can be expressed as 
\begin{equation*}
\widehat{v}-v 
=\widehat{v}-\widetilde{v}+\varepsilon,
\end{equation*}
where $\widetilde{v}=\widetilde{v}(\theta,\mathbf{y})=\argmax_{v}h(\theta,v)$.
The consistency of MLE $\widehat{\theta}$ leads to 
\begin{equation*}
\widehat{v}-\widetilde{v}
=\widetilde{v}(\widehat{\theta},\mathbf{y})-\widetilde{v}(\theta,\mathbf{y})
=o_{p}(1) 
\end{equation*}
with $\textrm{Var}(\widehat{v}-\widetilde{v})\to 0$ as $n\to\infty$, 
whereas $\varepsilon=O_{p}(1)$ since 
\begin{equation*}
\textrm{Var}(\varepsilon)
=\textrm{E}\{\textrm{Var}(v|\mathbf{y})\}
+\textrm{Var}\{\widetilde{v}-\textrm{E}(v|\mathbf{y})\}
\geq \textrm{E}\{\textrm{Var}(v|\mathbf{y})\} 
\end{equation*}
may not decrease as $n$ increases. 
However, for the prediction of future (or missing) data, 
the error term $\varepsilon $ would be predicted as null,
which implies that $\widehat{v}$ is used for predicting $v$ 
by consistently estimating $\widetilde{v}$. 
In consequence, the inconsistency of $\widehat{v}$ noted by \citet{meng09}
is due to $\textrm{Var}(\varepsilon)\not\to 0$,
because $m=0$ for the future observation or missing data $u=y_{n+1}$.

Suppose that $v=g(u)$ is a Bartlizable scale to give 
$\widetilde{u}=g^{-1}(\widetilde{v})=\textrm{E}(u|\mathbf{y})$, 
and denote the inverse matrices of observed and expected h-information by 
\begin{equation*}
I(\theta,v)^{-1}=
\begin{bmatrix}
I_{\theta\theta} & I_{\theta v} \\ 
I_{v\theta} & I_{vv}
\end{bmatrix}
^{-1}=
\begin{bmatrix}
I^{\theta\theta} & I^{\theta v} \\ 
I^{v\theta} & I^{vv}
\end{bmatrix}
\quad \textrm{and}\quad \mathcal{I}_{\theta}^{-1}=
\begin{bmatrix}
\mathcal{I}_{\theta\theta} & \mathcal{I}_{\theta v} \\ 
\mathcal{I}_{v\theta} & \mathcal{I}_{vv}
\end{bmatrix}
^{-1}=
\begin{bmatrix}
\mathcal{I}^{\theta\theta} & \mathcal{I}^{\theta v} \\ 
\mathcal{I}^{v\theta} & \mathcal{I}^{vv}
\end{bmatrix}
, 
\end{equation*}
respectively. 
Then, the MHLE $\widehat{u}$ becomes an asymptotically unbiased predictor of $u$. 
Furthermore, the MHLEs asymptotically achieve the GCRLB as $n\to \infty$, 
\begin{equation*}
\textrm{Var}(\widehat{\theta}-\theta)
\to \mathcal{I}^{\theta \theta}
\quad \textrm{and}\quad 
\textrm{Var}(\widehat{u}-u)
\to \mathcal{I}^{vv}\cdot \left[ \textrm{E}\{g^{\prime }(u)\}\right]^{-2}. 
\end{equation*}
Thus, even though $\widehat{u}-u=O_{p}(1)$, the MHLEs can still give asymptotically
optimal estimation and prediction for fixed and random parameters. 
For estimating the variance of fixed parameter estimator, 
the observed h-information gives
\begin{align*}
\widehat{\textrm{Var}}(\widehat{\theta}-\theta)
=\widehat{I}^{\theta \theta}
&= \left[ -\nabla_{\theta}^2 h(\theta,\widetilde{v})\right]
_{\theta=\widehat{\theta}}^{-1}
=\left[ -\nabla^{2}_{\theta}\ell(\theta)\right]
_{\theta=\widehat{\theta}}^{-1}
\to \textrm{Var}(\widehat{\theta}-\theta),
\end{align*}
by the property of MLE. Here the observed h-information gives 
\begin{equation*}
\widehat{\textrm{Var}}\left( \widehat{v}-v\right) 
=\widehat{I}^{vv}
=\widehat{I}_{vv}^{-1}\widehat{I}_{v\theta}\widehat{I}^{\theta \theta}
\widehat{I}_{\theta v}\widehat{I}_{vv}^{-1}+\widehat{I}_{vv}^{-1}, 
\end{equation*}
which may not be consistent in general. 
However, if $v=g(u)$ is a normalizing transformation 
such that $v|\mathbf{y}\sim N(\widetilde{v},I_{vv}^{-1})$, 
where $I_{vv}^{-1}$ is free of $v$, 
then the consistency of $\widehat{\theta}$ leads to 
$\widehat{I}_{vv}^{-1}\to \mathcal{I}_{vv}^{-1}$ and 
\begin{equation*}
\widehat{\textrm{Var}}\left( \widehat{v}-v\right) 
\to \textrm{Var}\left(\widehat{v}-v\right).
\end{equation*}
Thus, when $v|\mathbf{y}$ is approximately normal, 
we can still obtain a reasonable estimator of 
$\textrm{Var}(\widehat{v}-v)$ from the observed h-information.

\medskip
\noindent Example 3
\citet{meng09} demonstrated difficulties in the prediction of random unknowns,
by considering an example with observed data 
$\mathbf{y}=(y_1,...,y_n)^{\intercal}$ and missing data $u=y_{n+1}$
from independent $y_{i}\sim\textrm{Exp}(1/\theta)$
for $i=1,...,n+1$.
He claimed that the Hessian matrix cannot give a consistent variance estimator, because 
\begin{equation}
\textrm{Var}(\widehat{v}-v)
=\frac{1}{n}+\frac{\pi^{2}}{6}>\mathcal{I}^{vv}(\theta)
=\frac{1}{n}+1.  \label{eq:break}
\end{equation}
Note that $m=0$ and $\widetilde{v}$ is not a BUP of $v$. 
However, $\widetilde{u}=\textrm{E}(u|\mathbf{y})$ is the BUP of $u$, 
which attains the GCRLB in Theorem \ref{thm:hcrb}. 
Let $\zeta(\theta,v)=\exp(v)=u$, then 
\begin{equation*}
\mathcal{Z}_{\theta}
=\textrm{E}\left\{ \nabla_{\theta,v} \zeta(\theta,v) \right\}
=\textrm{E}(0, u)^{\intercal}
= (0, \theta)^{\intercal}.
\end{equation*}
From \eqref{eq:break}, $\widehat{u}$ achieves the GCRLB
$\mathcal{Z}_{\theta}\mathcal{I}_{\theta}^{-1}\mathcal{Z}_{\theta}$ 
in Theorem \ref{thm:hcrb}, 
\begin{equation*}
\textrm{Var}(\widehat{u}-u)
=\theta^{2}\left( \frac{1}{n}+1\right) 
=\theta^{2}\mathcal{I}^{vv}(\theta), 
\end{equation*}
and the inverse of observed h-information 
$I(\widehat{\theta},\widehat{u})^{-1}$ gives a consistent variance estimator, 
\begin{equation*}
\widehat{\textrm{Var}}(\widehat{u}-u)
=\widehat{I}^{uu}
=\frac{\widehat{\theta}^{2}}{n}+\widehat{\theta}^{2}
\to \textrm{Var}(\widehat{u}-u), 
\end{equation*}
since $\widehat{\theta}\to \theta$ as $n\to\infty$.
In this example, since $m=0$,
the GCRLB cannot be achieved for arbitrary scale $g(v)$. 
The delta method may not give a consistent variance estimator, 
since $\widehat{I}^{vv}\not\to\textrm{Var}(\widehat{v}-v)$. 
In such small samples, the GCRLB can be achieved 
only for a specific scale of random parameters, having the BUP property.
Thus, the scale of random parameters is important 
to obtain asymptotically optimal MHLEs and their consistent variance estimators.
\hfill \qed

\subsection{Random effects models with $n=1$}
\label{sec:small_n}
This section presents an example that $n=1$ and as $m\to\infty$,
\begin{equation*}
\widehat{\boldsymbol{\theta}}-\boldsymbol{\theta}=O_{p}(1)
\quad \textrm{and} \quad 
\widehat{\mathbf{v}}-\mathbf{v}=o_{p}(1).
\end{equation*}

\medskip
\noindent Example 4
We consider a model 
with $U\sim\textrm{Exp}(\theta)$ and $Y_{j}|u\sim\textrm{Exp}(u)$
for $j=1,...,m$ such that
\begin{equation*}
f_{\theta}(u)=\theta \exp \left( -\theta u\right) 
\quad \textrm{and}\quad
f(y_{j}|u)=u\exp (-uy_{j}).
\end{equation*}
In this example, 
$n=\textrm{rank}(\mathbf{Z})=1$. 
The marginal log-likelihood is 
\begin{equation*}
\ell (\theta)=\log \theta -(m+1)\log (\theta +m\bar{y})+\log \Gamma(m+1). 
\end{equation*}
Then, the MLE is $\widehat{\theta}=\bar{y}=(y_{1}+\cdots +y_{m})/m$, 
whose expectation and variance are 
\begin{equation*}
\textrm{E}(\widehat{\theta}) =\textrm{E}(\bar{y}) =\infty 
\quad \textrm{and}\quad 
\textrm{Var}(\widehat{\theta}) =\textrm{Var}(\bar{y}) =\infty.
\end{equation*}
While $\widehat{\theta}$ is not consistent,
the h-likelihood can still give 
an optimal predictor of random parameter. 
Since Lemma \ref{lem:bartlett} implies that 
$v=\log u\in \Omega_{v}=\mathbb{R}$ is Bartlizable, 
we can define the h-likelihood as 
\begin{equation*}
h(\theta,v)=\ell_{e}(\theta,v)+a(\theta;\mathbf{y})
= \log \theta -e^{v}(\theta +m\bar{y})+v(m+1) +a(\theta;\mathbf{y}), 
\end{equation*}
where $a(\theta;\mathbf{y})=(m+1)\{1-\log (m+1)\}+\log \Gamma(m+1)$. 
Joint maximization of $h(\theta,v)$ leads to the MHLEs 
$\widehat{\theta}=\bar{y}$ and $\widehat{v}= -\log \bar{y}$. 
Here, $\widehat{u}=\exp(\widehat{v})=1/\bar{y}$
is a consistent predictor of $u$
and the MHLE of $v$ has
\begin{equation*}
\textrm{E}(\widehat{v}-v)=\psi (m)-\log m\to 0
\quad \textrm{and} \quad 
\textrm{Var}(\widehat{v}-v)=\psi^{(1)}(m)\to 0, 
\end{equation*}
as $m\to \infty $, where $\psi(\cdot)$ and $\psi^{(1)}(\cdot)$
are digamma and trigamma functions, respectively. Since $\psi^{(1)}(m)\geq
1/m$ with $\psi^{(1)}(m)-1/m=O(m^{-2})$, the MHLE of $v$ is consistent and
asymptotically achieves the GCRLB,
\begin{equation*}
\textrm{Var}(\widehat{v}-v)
=\psi^{(1)}(m)
=\frac{1}{m}\left[ 1+O\left( \frac{1}{m}\right) \right] 
\geq (0,1) \ \mathcal{I}(\theta,v)^{-1} \ (0,1)^{\intercal}
=\frac{1}{m}. 
\end{equation*}
Though $\widehat{\theta}-\theta=O_{p}(1)$, 
the MHLE of $v$ can be asymptotically the BUP as $m\to\infty$.
In Section \ref{app:details} of the Appendix, we show that for given $\theta$,
the MHLE of $u$ is the BUP, $\widetilde{u} = \exp(\widetilde{v}) = \textrm{E}(u|\mathbf{y})$,
but the GCRLB is not tight.
The MHLE of $\xi=1/\theta=\textrm{E}(u)$ 
can still be an asymptotically
best unbiased estimator,
though it is not consistent as $m\to\infty$.
\hfill \qed

\subsection{Limitation of MHLEs}
\label{sec:small_both}
We have studied the cases where either $m$ or $n$ is small.
As \citet{efron98} remarked for ML procedures in small samples,
we show that MHL procedures also provide quite reliable estimators
for fixed and random parameters,
coming close to the ideal optimum even when either $m$ or $n$ is small.
However, when both $m$ and $n$ are small,
optimal properties of MHLEs may not hold.
For example, when $N=m=n=1$, Example 4 becomes 
\citeauthor{bayarri88}'s \citeyearpar{bayarri88} example
with $\widehat{\theta}=y$ and $\widehat{v}=-\log y$. 
Here, the h-likelihood is an extended likelihood,
exploiting the full data evidence \citep{bjornstad96},
but point estimation is still challenging
because there is a single observation with two parameters:
fixed $\theta$ and random $u$.
When $m=1$, 
$\textrm{E}(\widehat{\theta})$,  
$\textrm{Var}(\widehat{\theta})$, 
$\textrm{E}(\widehat{\xi})$,  
$\textrm{Var}(\widehat{\xi})$, 
$\textrm{Var}(\widehat{u}-\widetilde{u})$
and $\textrm{Var}(\widehat{u}-u)$ are all infinite.
Thus, the resulting MHLEs cannot provide meaningful point estimation and prediction,
and asymptotic normality cannot be applied to the MHLEs.

\section{Interval estimation and prediction}
\label{sec:finite} 

We have studied the asymptotic properties of MHLEs in Section \ref{sec:asymptotic}.
Likelihood procedures have been criticized 
for providing exact inferences only asymptotically. 
Examples 3 and 4 have been developed to show 
the difficulty in likelihood-based inference.
In \citeauthor{bayarri88}'s \citeyearpar{bayarri88} example with $N=m=n=1$,
the ML procedure fails to provide valid inference.
Bayesian procedures provide exact interval estimation in small samples,
but have been criticized for unverifiable prior assumptions.
Recently, there has been renewed interest 
with a modern definition of the fiducial probability \citep{fisher30}, 
called the confidence distribution
(CD; \citeauthor{schweder16}, \citeyear{schweder16}).
However, the current CD has been developed 
for statistical models with fixed unknowns only. 
In this section, we extend the CD
with the goal of providing intervals 
that maintain coverage probabilities
for both fixed and random unknowns in small samples.
Using the two examples, we introduce the h-confidence as a modified extended likelihood,
which can lead to the CD for fixed unknowns and the PD for random unknowns,
providing intervals that maintain coverage probabilities.

\subsection{Confidence interval procedures for fixed unknowns}
\label{sec:CI}

Let $\mathbf{y}=(y_{1},...,y_{n})^{\intercal}$ be a vector of observed values of
random variable $\textit{\textbf{Y}}=(Y_{1},...,Y_{n})^{\intercal}$ 
and $\theta_{0}$ represent the true value of a fixed unknown $\theta$. Suppose
that there exists a sufficient statistic $T=T(\textit{\textbf{Y}})$
for a statistical model $f_{\theta}(\mathbf{y})$\ with its observed value
denoted by $t=T(\mathbf{y})$. 
The confidence density for fixed unknowns 
can be derived from the right side p-value function: 
\begin{equation*}
c(\theta;\mathbf{y})
=\nabla_{\theta} P_{\theta}(T\geq t).
\end{equation*}
Here, the right side p-value 
$P_{\theta}(T\geq t)\sim \textrm{Uniform}(0,1)$ is a pivotal quantity,
whose distribution is independent of $\theta$.
Let $\textrm{CI}(\cdot)$ be a function that generates a confidence interval (CI)
for fixed unknown $\theta$. The current definition of CD is required to
satisfy the confidence feature \citep{schweder16,pawitan23}: under
appropriate conditions, for any true value of $\theta_{0}\in \Theta $, 
\begin{equation} 
P_{\theta_{0}}(\theta_{0}\in \textrm{CI}(\textit{\textbf{Y}}))
=C(\theta_{0}\in \textrm{CI}(\mathbf{y})),  
\label{eq:confidence_feature}
\end{equation}
where the LHS is the frequentist coverage probability 
of the CI procedure $\textrm{CI}(\textit{\textbf{Y}})$ 
and the RHS is the confidence of an observed interval $\textrm{CI}(\mathbf{y})$, 
defined by 
\begin{equation*}
C(\theta_{0}\in \textrm{CI}(\mathbf{y}))
=\int_{\textrm{CI}(\mathbf{y})}c(\theta;\mathbf{y})d\theta. 
\end{equation*}

The Bayesian credible interval $\textrm{BI}(\mathbf{y})$ 
from the posterior $\pi (\theta |\mathbf{y})$ 
is a fixed interval for the random parameter $\theta$ under a prior $\pi(\theta)$, 
allowing the probability statement,
\begin{equation*}
P(\theta \in \textrm{BI}(\mathbf{y})|\mathbf{y})
=\int_{\textrm{BI}(\mathbf{y})}\pi (\theta |\mathbf{y})d\theta
=1-\alpha,
\end{equation*}
for a predetermined level $1-\alpha$. 
The frequentist CI procedure $\textrm{CI}(\textit{\textbf{Y}})$ 
is a random interval for the fixed parameter $\theta=\theta_{0}$ with 
the coverage probability, allowing the probability statement,
\begin{equation*}
P_{\theta_{0}}(\theta_{0}\in \textrm{CI}(\textit{\textbf{Y}}))=1-\alpha, 
\end{equation*}
whereas the observed interval $\textrm{CI}(\mathbf{y})$ 
is a fixed interval for the fixed parameter $\theta=\theta_{0}$ 
with the confidence,
\begin{equation*}
C(\theta_{0}\in \textrm{CI}(\mathbf{y}))=\int_{\textrm{CI}(\mathbf{y})}c(\theta
;\mathbf{y})d\theta=1-\alpha. 
\end{equation*}
This confidence is clearly distinct from probability.
For an observed $\mathbf{y}$, the probability 
$P_{\theta_{0}}(\theta_{0}\in \textrm{CI}(\mathbf{y}))
=I(\theta_{0}\in \textrm{CI}(\mathbf{y}))$ is either 0 or 1, 
but remains unknown since $\theta_{0}$ is unknown. 
The Bayesian school allows the probability statement
for an observed interval,
whereas the frequentist school does not allow it
for an observed interval but for a CI procedure,
namely coverage probability \citep{pawitan23}.
Thus, this paper proposes the use of a confidence statement for observed intervals.

\citet{fisher35} showed that the fiducial argument 
leads to the t-interval for small samples,
but it encountered prolonged controversies
\citep{bartlett36, bartlett39, bartlett65, fisher41, fisher54, pedersen78},
as the marginalization yields incorrect confidence levels in multi-parameter cases.
\citet{pawitan21} showed that confidence is not a probability but an extended likelihood
(see Section \ref{app:confidence} of the Appendix),
which is key to avoiding probability-related paradoxes \citep{pawitan17,pawitan24}.
Confidence is neither the likelihood;
likelihood and extended likelihood are distinct concepts \citep{pawitan22}.
As \citet{schweder16} noted, CDs could be analogous to Bayesian posteriors, 
\begin{equation}
\label{eq:cd_fixed}
c(\theta;\mathbf{y})
= c_{0}(\theta;\mathbf{y})\ L(\theta;\mathbf{y})
\end{equation}
where $c_{0}(\theta;\mathbf{y})$ is an implied prior for $\theta$
\citep{pawitan23, lee24satellite}.
In single parameter cases,
a posterior under Jeffreys prior is asymptotically close to the CD \citep{lindley58, welch63}, 
and the same holds for the reference prior \citep{bernardo79}, 
as it is equivalent to Jeffreys prior in this case. 
However, in high-dimensional cases of \citeauthor{stein59}'s \citeyearpar{stein59} problem, 
\citet{lee24satellite} showed that credible intervals 
under the flat prior and the reference prior
cannot maintain the confidence feature \eqref{eq:confidence_feature},
whereas the CD can.
Thus, given that the implied prior could be data-dependent, 
we prefer to interpret the CD as a modified likelihood 
that allows the confidence statement for an observed CI.
\citet{efron98} remarked, `Maybe Fisher's biggest blunder will become a big hit in the 21st century,' 
highlighting the need for further investigation into the properties of confidence.
Following examples illustrate that the CD provides exact interval estimation for fixed parameters,
similar to Bayesian inferences, but in frequentist sense.

\medskip
\noindent Example 3 (continued)
Let $t$ be an observed value of 
a sufficient statistic $T=\sum_{i=1}^{n} Y_{i}$ for $\theta$
with respect to $f_{\theta}(\mathbf{y})$.
Then, for observed data $\mathbf{y}$,
$T\sim \textrm{Gamma}(n,1/\theta)$ leads
to the confidence density of $\theta$,
\begin{equation*}
c(\theta;\mathbf{y})
=\nabla_{\theta} P_{\theta}(T\geq t)
=\frac{t^{n}\exp (-t/\theta)}{\theta^{n+1}\Gamma(n+1)},
\label{eq:Ex3_cd}
\end{equation*}
which is a modified likelihood with a modification term
$
c_{0}(\theta;\mathbf{y})
= c(\theta;\mathbf{y})/L(\theta;\mathbf{y})
\propto \theta^{-1}.
$
This is equivalent to the \citet{jeffreys98} prior. 
This CD yields a $100(1-\alpha)\%$ observed CI for $\theta$, 
\begin{equation*}
\textrm{CI}_{\alpha}(\mathbf{y})
=\left(\theta_{\textrm{lower}},\theta_{\textrm{upper}}\right), 
\end{equation*}
such that 
$\Gamma(n,t/\theta_{\textrm{lower}})=\Gamma(n)\cdot\alpha/2$ 
and $\Gamma(n,t/\theta_{\textrm{upper}})=\Gamma(n)\cdot(1-\alpha/2)$, 
where $\Gamma(\cdot,\cdot)$ denotes the incomplete gamma function. 
This CI maintains the confidence feature \eqref{eq:confidence_feature} 
by satisfying the regularity conditions in \citet{pawitan23}. 
Figure \ref{fig:Ex3_theta_covp} shows 
the actual coverage probabilities of $90\%$ CI 
based on the CD (solid) and the Wald CI (dashed) 
from the simulation study with 10,000 iterations 
for each $\theta>0$ and $n=1,2,5,10$.
The Wald CI relies on asymptotic normality;
it maintains the coverage probability as $n$ grows.
Meanwhile, CI based on the CD is exact, 
maintaining the coverage probability exactly, 
even for $n=1$. 
\hfill \qed

\begin{figure}[tbp]
\centering
\includegraphics[width=\linewidth]{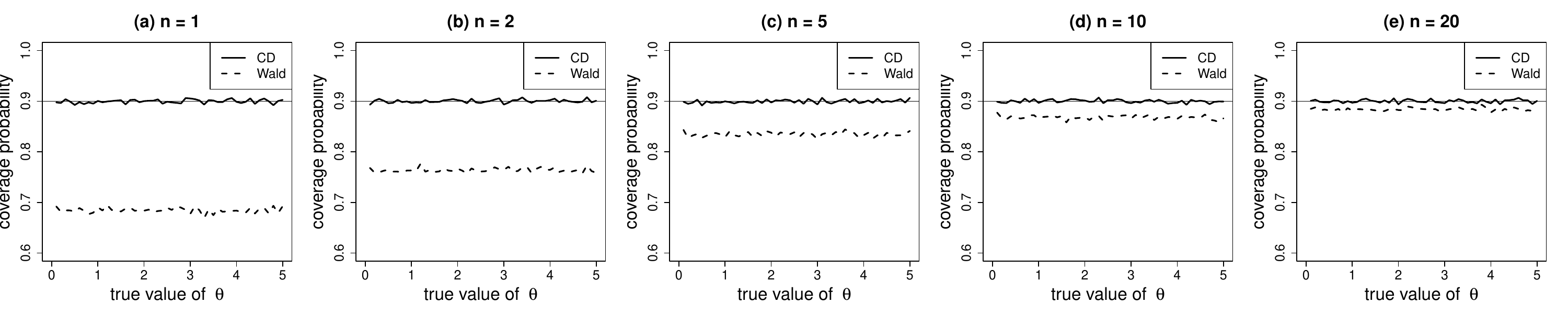}
\caption{Coverage probabilities of 90\% CI for $\theta $ in Example
3 where the true value of $\theta $ varies from 0.01 to 5.}
\label{fig:Ex3_theta_covp}
\includegraphics[width=\linewidth]{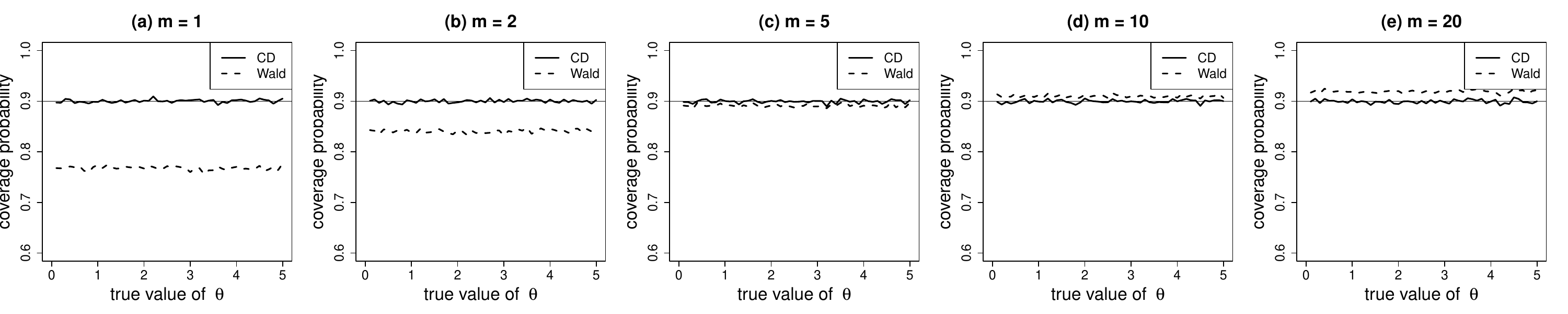}
\caption{Coverage probabilities of 90\% CI for $\theta $ in Example
4 where the true value of $\theta $ varies from 0.01 to 5.}
\label{fig:Ex4_theta_covp}
\end{figure}

\medskip
\noindent Example 4 (continued)
Let $t$ be an observed value of a sufficient statistic $T=\sum_{j=1}^{m}Y_{j}$
for $\theta$ with respect to $f_{\theta}(\mathbf{y})$, 
then the confidence density for $\theta$ is 
\begin{equation*}
c(\theta;\mathbf{y})
=\nabla_{\theta} P_{\theta}(T\geq t)
=\frac{mt^{m}}{(\theta +t)^{m+1}},  
\end{equation*}
leading to the Jeffreys prior as a modification term, 
$
c_{0}(\theta;\mathbf{y}) 
= c(\theta;\mathbf{y})/L(\theta;\mathbf{y})
\propto \theta^{-1}. 
$
This yields a CI for $\theta$,
\begin{equation*}
\textrm{CI}_{\alpha}(\mathbf{y})
=\left( t\cdot \{\left( 1-\alpha/2\right)^{-1/m}-1\},
\ t\cdot \{\left( \alpha/2\right)^{-1/m}-1\}\right).
\end{equation*}
This CI maintains the confidence feature \eqref{eq:confidence_feature} by
satisfying the regularity conditions in \citet{pawitan23}. 
Figure \ref{fig:Ex4_theta_covp} shows the actual coverage probabilities 
of $90\%$ CI based on the CD (solid) and the Wald CI (dashed) 
from the simulation study with 10,000 iterations for $\theta >0$ 
and $m=1,2,5,10$. Here the
coverage probability of the Wald CI does not converge to the
confidence level even though $m$ approaches infinity,
since the MLE of $\theta$ and variance estimator are not consistent. 
Thus, ML procedures for $\theta$ could not
give reasonable interval estimation, even with large $m$.
However, CI based on the CD is exact even for $m=1$. 
\hfill \qed

\subsection{Predictive interval procedures for random unknowns}
\label{sec:PI}

\citet{robinson91} pointed out that
the prediction of random unknown can be made for 
either the realized value (conditional prediction) 
or the future value (marginal prediction) of random unknown. 
For example, in animal breeding, the breeding value of an already-born animal 
is the realized value of a random unknown,
whereas the breeding value of a mating between two potential parents 
is the future value of a random unknown. 
To emphasize the distinction, 
we denote the realized value of a random unknown by $\mathbf{v}_{0}$ 
and the future value of a random unknown by $\textit{\textbf{V}}$,
while $\mathbf{v}$ may represent either $\mathbf{v}_{0}$ or $\textit{\textbf{V}}$. 
It is worth noting that the realized values $\mathbf{v}_{0}$ are treated 
as if they were fixed unknowns. 
For realized values, 
$\widehat{\mathbf{v}}_{0}-\mathbf{v}_{0}$ can be asymptotically normal, 
hence the Wald interval maintains the coverage probability as $m\to\infty$.
For future values with $m=0$, 
unless $\textit{\textbf{V}}$ is normal,
$\widehat{\textit{\textbf{V}}}-\textit{\textbf{V}}$ would not be normal,
regardless of $n\to\infty$.
Thus, the Wald interval cannot maintain the coverage probability.
\citet{pearson20} pointed out a limitation of Fisher's plug-in method, 
using $f_{\widehat{\boldsymbol{\theta}}}(\mathbf{v}|\mathbf{y})$
as the PD for future random variables $\mathbf{v}$,
proportional to the extended likelihood $L_{e}(\widehat{\boldsymbol{\theta}}, \mathbf{v})$
of \citet{bjornstad96} for prediction of $\mathbf{v}$.
It fails to account for the information loss 
from estimating the fixed parameters $\boldsymbol{\theta}$.
It tends to be misleadingly precise \citep{aitchison75},
especially pronounced with high dimensional $\boldsymbol{\theta}$ \citep{Kalbfleisch70}.

The marginal PI represents 
a fixed interval for a random (future) unknown $\textit{\textbf{V}}$, 
similar to the Bayesian credible interval, 
whereas the conditional PI is a CI that
represents a fixed interval for a fixed (realized) unknown $\mathbf{v}_{0}$.
Thus, different schemes are required for simulation studies; 
a new random effect is generated from replication to replication 
to compute the marginal coverage probability 
$P_{\theta_{0}}(\textit{\textbf{V}}\in 
\textrm{PI}(\textit{\textbf{Y}}))$ 
and a single random effect is generated throughout the replications 
to compute the conditional coverage probability 
$P_{\theta_{0}}(\mathbf{v}_{0}\in
\textrm{PI}(\textit{\textbf{Y}})|\textit{\textbf{V}}=\mathbf{v}_{0})$ 
for the realized value $\mathbf{v}_{0}$.
Examples 3 and 4 below show the existence of such PI procedures by using pivotal quantities.

\medskip
\noindent Example 3 (continued)
In this example, $S=U/T$ is a pivotal quantity
whose distribution is independent of $\theta$.
Analogous to pivotal method for fixed unknowns \citep{schweder16},
replacing the random variable $T$ 
with an observed value $t$ leads to a pseudo-density function for $u$,
\begin{equation*}
f^*(u;t) 
= f(s) \cdot \left|\frac{ds}{du}\right|
= f\left(\frac{u}{t}\right) \cdot \frac{1}{t} 
= n t^n (t+u)^{-(n+1)},
\end{equation*}
where $s=u/t$.
It produces the PI for future value $U$ with coverage probability $1-\alpha$,
\begin{equation*}
\textrm{PI}_{\alpha}(t)
=(u_{\textrm{lower}},u_{\textrm{upper}})
=\left( 
t\cdot \{(1-\alpha/2)^{-1/n}-1\}, \ 
t\cdot \{(\alpha/2)^{-1/n}-1\}
\right),
\end{equation*}
Then, for any $\theta_0 \in \Theta$, 
\begin{align*}
P_{\theta_{0}}(U\in \textrm{PI}_{\alpha}(T))
= \int_{0}^{\infty} \int_{u_{\textrm{lower}}}^{u_{\textrm{upper}}}
f_{\theta}(t,u) \ du \ dt
= 1-\alpha.
\end{align*}
Figure \ref{fig:Ex3_u_covp} illustrates 
the actual coverage probabilities of $90\%$ PI for future value $U$
based on the pivot (solid), plug-in method (dotted) and Wald PI (dashed)
from the simulation study with 10,000 iterations for each $\theta>0$. 
The Wald PI cannot maintain the coverage probability
even for large $n=20$, 
since $m=0$ causes inconsistency of $\widehat{U}=\exp(\widehat{V})$ 
and violation of the asymptotic normality. 
Plug-in method gives reasonable PIs when $n$ is large,
but cannot maintain the coverage probability when $n$ is small.
However, the PI based on the pivot is exact even for $n=1$.
\hfill \qed

\begin{figure}[tbp]
\centering
\includegraphics[width=\linewidth]{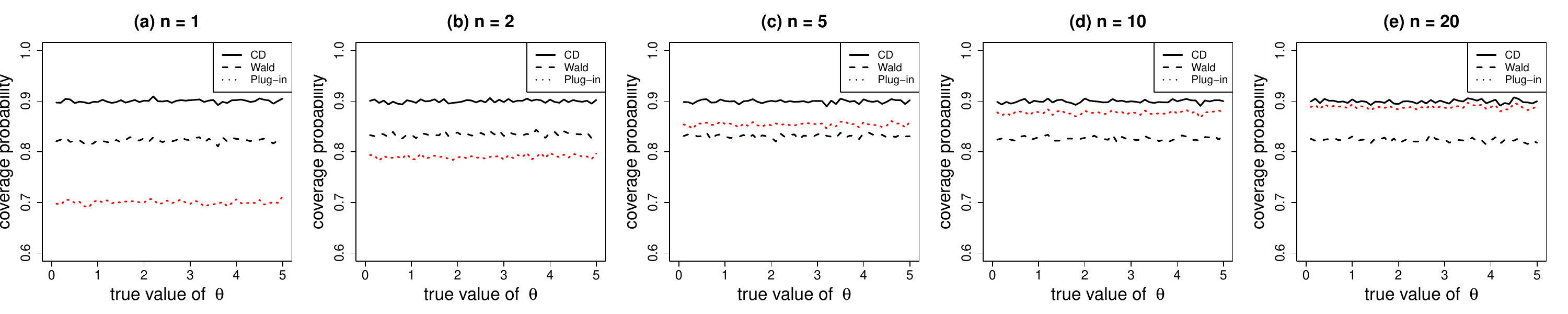} 
\caption{Marginal coverage probabilities of 90\% PI for future $u$ in Example 3,
based on pivotal quantities (solid), plug-in method (dotted) and Wald PI (dashed),
where the true value of $\theta $ varies.}
\label{fig:Ex3_u_covp}
\includegraphics[width=\linewidth]{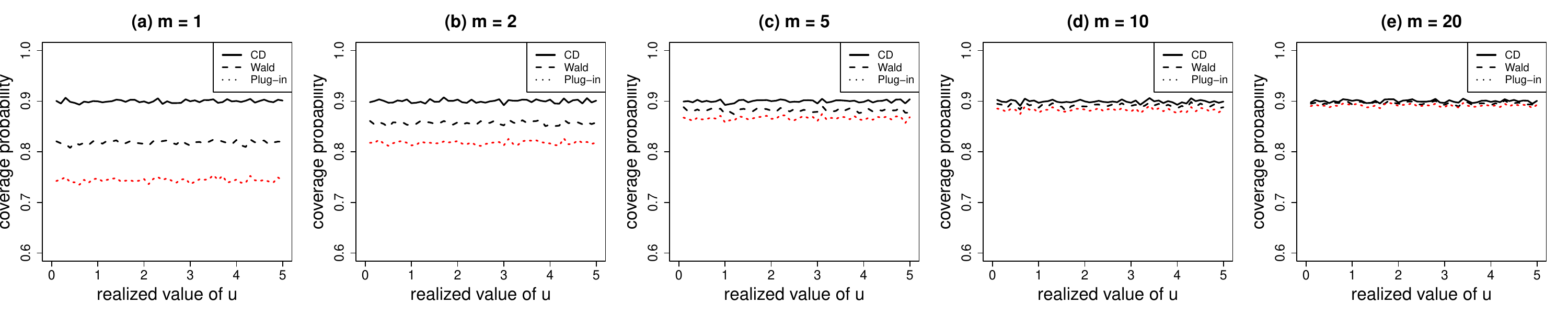}
\caption{Conditional coverage probabilities of 90\% PI for realized $u$ in Example 4, 
based on pivotal quantities (solid), plug-in method (dotted) and Wald PI (dashed),
where $\theta=1$ and realized value of $u$ varies.}
\label{fig:Ex4_u_covp}
\end{figure}

\medskip
\noindent Example 4 (continued)
In this example, $S=TU$ is a pivotal quantity
whose distribution is independent of $\theta$.
Replacing the random variable $T$ 
with an observed value $t$ leads to a pseudo-density function for $u$,
\begin{equation*}
f^*(u;t)
= f(s) \cdot \left|\frac{ds}{du}\right|
= f(tu) \cdot t
= \frac{t^{m}}{\Gamma(m)}u^{m-1}e^{-tu},
\end{equation*}
where $s=tu$.
It produces the PI for the realized value $u_{0}$ 
with coverage probability $1-\alpha$,
\begin{equation*}
\textrm{PI}_{\alpha}(t)=(u_{\textrm{lower}},u_{\textrm{upper}}), 
\end{equation*}
such that $\Gamma(m,t\cdot u_{\textrm{lower}})/\Gamma(m)=1-\alpha/2$
and $\Gamma(m,t\cdot u_{\textrm{upper}})/\Gamma(m)=\alpha/2$.
Let $(t_1,t_2)$ be the solutions of
$\Gamma(m,t_1 \cdot u_{0})/\Gamma(m)=1-\alpha/2$ and
$\Gamma(m,t_2 \cdot u_{0})/\Gamma(m)=\alpha/2$, respectively,
which leads to 
$u_{0} \in \textrm{PI}_{\alpha}(T) 
\Leftrightarrow T \in (t_1, t_2)$.
Then, for any $\theta_0 \in \Theta$, 
the coverage probability is
\begin{align*}
P_{\theta_{0}}(u_{0}\in \textrm{PI}_{\alpha}(T)|U=u_{0})
=\int_{t_1}^{t_2} f(t|u_{0}) \ dt
= 1 - \alpha.
\end{align*}
Figure \ref{fig:Ex4_u_covp} illustrates the actual coverage probabilities of 
$90\%$ PI for the realized value $u_{0}>0$ 
based on the pivot (solid), plug-in method (dotted) and Wald PI (dashed) 
from the simulation study with 10,000 iterations 
for each $u_{0}$, when $\theta=1$. 
The Wald PI and the PI based on plug-in method 
can maintain the coverage probability only for large $m$, 
whereas the PI based on the pivot is exact even for $m=1$.
\hfill \qed

\subsection{The h-confidence}
\label{sec:h-confidence}

This section proposes the h-confidence as a modified extended likelihood, 
whose marginal distributions become CD and PD.
LP and ELP imply that 
the PD should be based on the predictive likelihood 
$L_p(\mathbf{v};\boldsymbol{\theta},\mathbf{y}) = f_{\boldsymbol{\theta}}(\mathbf{v}|\mathbf{y})$.
The key is how to eliminate the nuisance parameters $\boldsymbol{\theta}$ 
from $f_{\boldsymbol{\theta}}(\mathbf{v}|\mathbf{y})$ to obtain a proper PD.
\citet{butler86} proposed a conditioning approach,
similar to \citet{hinkley77} to obtain a PD.
\citet{pawitan23} showed that the confidence with full data evidence
can be achieved by conditioning the maximal ancillary statistics,
but it could be intractable since the conditional distribution is often hard to derive.
\citet{pearson20} proposed the use of marginal posterior for $\mathbf{v}$, 
\begin{equation}
\label{eq:marginal_posterior}
\pi(\mathbf{v}|\mathbf{y})
=\int_{\boldsymbol{\Theta}} \pi (\boldsymbol{\theta}|\mathbf{y})
f_{\boldsymbol{\theta}}(\mathbf{v}|\mathbf{y}) d\boldsymbol{\theta}, 
\end{equation}
where $\pi (\boldsymbol{\theta}|\mathbf{y})$ denotes 
the marginal posterior of $\boldsymbol{\theta}$.
\citet{lee02,lee09} used the normalized profile h-likelihood as a PD,
proportional to $L_p(\mathbf{v}; \widetilde{\boldsymbol{\theta}}(\mathbf{v}), \mathbf{y})$
where $\widetilde{\boldsymbol{\theta}}(\mathbf{v})=\argmax_{\boldsymbol{\theta}}L_{e}(\boldsymbol{\theta},\mathbf{v})$.

This section defines the PD 
by extending the confidence feature \eqref{eq:confidence_feature}
for fixed unknowns to accommodate random unknowns.
Let $c(\mathbf{v};\mathbf{y})$ be the density function of the PD for $\mathbf{v}$
and define the confidence of an observed predictive interval (PI) as 
\begin{equation*}
C(\mathbf{v}\in \textrm{PI}(\mathbf{y}))
=\int_{\textrm{PI}(\mathbf{y})}c(\mathbf{v};\mathbf{y})d\mathbf{v},
\end{equation*}
so that the $100(1-\alpha)\%$ observed PI for $\mathbf{v}$ is defined to satisfy
$C(\mathbf{v}\in \textrm{PI}(\mathbf{y}))=1-\alpha$.
In this paper, confidence features for random unknowns indicate that
for any $\theta_{0}\in \Theta$, 
\begin{align*}
&P_{\theta_{0}}
(\textit{\textbf{V}}\in \textrm{PI}(\textit{\textbf{Y}}))
=C(\textit{\textbf{V}}\in \textrm{PI}(\mathbf{y}))
=1-\alpha, \\
&P_{\theta_{0}}
(\mathbf{v}_{0}\in \textrm{PI}(\textit{\textbf{Y}})
|\textit{\textbf{V}}=\mathbf{v}_{0})
=C(\mathbf{v}_{0}\in \textrm{PI}(\mathbf{y}))
=1-\alpha,
\end{align*}
for marginal PIs and conditional PIs, respectively.
This allows frequentist interpretation of PIs for random unknowns.
Performance of PIs are investigated in a frequentist perspective 
that the predetermined confidence level, 
either $C(\textit{\textbf{V}}\in \textrm{PI}(\mathbf{y}))$ 
or $C(\mathbf{v}_{0}\in \textrm{PI}(\mathbf{y}))$, 
becomes identical to the actual coverage probability, 
either $P_{\theta_{0}}(\textit{\textbf{V}}\in 
\textrm{PI}(\textit{\textbf{Y}}))$ 
or $P_{\theta_{0}}(\mathbf{v}_{0}\in 
\textrm{PI}(\textit{\textbf{Y}})|\textit{\textbf{V}}=\mathbf{v}_{0})$,
respectively. 

We now introduce an h-confidence,
extending the concept of CD \eqref{eq:cd_fixed}
as a modified extended likelihood
to accommodate additional random unknowns:
\begin{equation}
\label{eq:h-confidence}
c_{h}(\boldsymbol{\theta}, \mathbf{v}; \mathbf{y})
= c_{0}(\boldsymbol{\theta}, \mathbf{v}; \mathbf{y}) L_{e}(\boldsymbol{\theta}, \mathbf{v}; \mathbf{y}),
\end{equation}
where $L_{e}(\boldsymbol{\theta}, \mathbf{v}; \mathbf{y})$ is an extended likelihood
and $c_{0}(\boldsymbol{\theta}, \mathbf{v}; \mathbf{y})$ is a modification term,
such that marginalization of $c_h(\boldsymbol{\theta}, \mathbf{v}; \mathbf{y})$ leads to CD and PD,
$$
c(\boldsymbol{\theta};\mathbf{y}) = \int_{\Omega_{\mathbf{v}}} c_{h}(\boldsymbol{\theta}, \mathbf{v}; \mathbf{y}) d \mathbf{v}
\quad \textrm{and} \quad
c(\mathbf{v};\mathbf{y}) = \int_{\boldsymbol{\Theta}} c_{h}(\boldsymbol{\theta}, \mathbf{v}; \mathbf{y}) d \boldsymbol{\theta},
$$
respectively, satisfying the confidence features.
As the confidence from CD is interpreted as an extended likelihood \citep{pawitan21}, 
the confidence from PD can similarly be understood (see Section \ref{app:confidence} of the Appendix).

\subsubsection{Future random unknowns}

Let $t=T(\mathbf{y})\in\Omega$ be a sufficient statistic 
for the fixed unknown $\theta$
with respect to the statistical model $f_{\theta}(\mathbf{y})$
and $U=Y_{n+1}$ be a future random unknown from the density $f_{\theta}(u)$.
Suppose that there exists a monotone one-to-one transformation $\eta = \eta (\theta)$
from the parameter space $\Theta$ to $\Omega$,
such that for any $t \in \Omega$,
\begin{equation*}
F_{n} (t; \eta) = F_{n} (\eta; t),
\end{equation*}
where $F_{n}(\cdot \ ; \eta)$ is a cumulative distribution function of
$T$ with a parameter $\eta$ and sample size $n$.
Here, we can define the h-confidence \eqref{eq:h-confidence} as
\begin{equation}
\label{eq:hd_extendable}
c_{h}(\theta, u; \mathbf{y}) 
= \left\{\nabla_{\theta} P_{\theta}(T\geq t)\right\}
f_{\theta}(u|\mathbf{y})
= c(\theta; \mathbf{y}) f_{\theta}(u|\mathbf{y})
= c_{0}(\theta; \mathbf{y}) L_{e}(\theta, u; \mathbf{y}),
\end{equation}
where 
$c(\theta; \mathbf{y})=|\nabla_{\theta} P_{\theta}(T\geq t)|$
and the modification term $c_{0}(\theta;\mathbf{y})=c(\theta;\mathbf{y})/L(\theta;\mathbf{y})$.
Then, the marginal CD for $\theta$ becomes $c(\theta;\mathbf{y})$
and the marginal PD for $u$ becomes
$$
c(u; \mathbf{y}) = \int_{\Theta} c_{h}(\theta, u; \mathbf{y}) d\theta
= \int_{\Theta} c(\theta, \mathbf{y}) f_{\theta} (u|\mathbf{y}) d\theta.
$$
Let $\textrm{CI}_{\alpha}(t) \subseteq \Theta$ 
and $\textrm{PI}_{\alpha}(t) \subseteq \Omega_{u}$ 
be observed intervals of $\theta$ and $u$ such that
\begin{align*}
\int_{\textrm{CI}_{\alpha}(t)}
c(\theta;\mathbf{y}) d\theta
= 1-\alpha
\quad \textrm{and} \quad
\int_{\textrm{PI}_{\alpha}(t)}
c(u;\mathbf{y}) du 
= 1-\alpha,
\end{align*}
respectively.
The following theorem shows that
coverage probabilities of corresponding interval procedures
$\textrm{CI}_{\alpha}(T)$ and $\textrm{PI}_{\alpha}(T)$
are both exactly $1-\alpha$.

\begin{theorem}
\label{thm:future}
Let $q_{\alpha}(t)$ be the $\alpha$-quantile of $c(u;t)$.
If $F_{1}(q_{\alpha}(t); \eta) = F_{1}(t; q_{\alpha}(\eta))$,
where $F_{1}(\cdot \ ; \eta)$ is a cumulative distribution function
of sufficient statistic for future value $U=Y_{n+1}$.
Then for any true $\theta_0\in\Theta$,
\begin{align*}
P_{\theta_0}(\theta_0 \in \textrm{CI}_{\alpha}(T)) 
= 1-\alpha
\quad \text{and} \quad
P_{\theta_0}(U \in \textrm{PI}_{\alpha}(T)) 
= 1-\alpha.
\end{align*}
\end{theorem}

\medskip
\noindent Example 3 (continued)
Since $Y_i \sim \textrm{Exp}(1/\theta)$
and $T \sim \textrm{Gamma}(n,1/\theta)$,
$\eta = 1/\theta$ satisfies
$F_{n} (t; \eta) 
= 1 - \Gamma(n, t \eta)/\Gamma(n)
= F_{n} (\eta; t).
$
With the modification term
$c_{0}(\theta;\mathbf{y}) \propto 1/\theta$,
the h-confidence \eqref{eq:h-confidence} 
can be defined as \eqref{eq:hd_extendable},
$$
c(\theta,u;\mathbf{y}) 
= \frac{t^{n}\exp \{-(t+u)/\theta\}}{\theta^{n+2}\Gamma(n+1)}
\propto \theta^{-1} L_{e}(\theta,u;\mathbf{y}).
$$
This leads to the CD and PD,
$$
c(\theta;\mathbf{y}) = \frac{t^{n}\exp (-t/\theta)}{\theta^{n+1}\Gamma(n+1)}
\quad \textrm{and} \quad
c(u;\mathbf{y}) = \frac{nt^{n}}{(t+u)^{n+1}},
$$
which satisfy the confidence features,
as $c(\theta;\mathbf{y})$ in Section \ref{sec:CI} 
and $c(u;\mathbf{y})=f^{*}(u;t)$ in Section \ref{sec:PI}, respectively.
Since the quantile function $q_{\alpha}(t)=t\cdot (\alpha^{-1/n}-1)$ leads to
\begin{align*}
F_{1} (t\cdot (\alpha^{-1/n}-1); \eta) 
= 1 - \Gamma(1, t\eta \cdot (\alpha^{-1/n}-1))
= F_{1} (t; \eta\cdot (\alpha^{-1/n}-1)),
\end{align*}
Theorem \ref{thm:future} implies
that the h-confidence \eqref{eq:hd_extendable} gives 
intervals for fixed and random unknowns, maintaining confidence features.
\hfill \qed

\subsubsection{Realized random unknowns}
Let $t=T(\mathbf{y})$ be a sufficient statistic 
for a realized value of random unknown $u=u_0$
with respect to the statistical model $f(\mathbf{y}|u)$,
where the random unknown $u$ is generated from $f_{\theta}(u)$.
Analogous to the regularity condition in \citet{pawitan23},
suppose that for any $\alpha \in (0,1)$,
the $\alpha$-quantiles of $T|u$ and $U$
are strictly increasing function of $u$ and $\theta$, respectively.
Then, we can define
$$
c(u;\mathbf{y}) =\nabla_{u} P(T\geq t|u)
\quad \textrm{and} \quad
c'(\theta;u) = \nabla_{\theta} P_{\theta}(U \geq u),
$$
respectively. 
It is worth noting that $c(u;\mathbf{y})$ is a proper PD for $u$,
but $c'(\theta;u)$ is not a proper CD 
as $u$ is a realized value of an unobserved the random variable.
Let $c_{0}(u;\mathbf{y}) = {c(u;\mathbf{y})}/{f(\mathbf{y}|u)}$
and $c'_{0}(\theta;u) = {c'(\theta;u)}/{f_{\theta}(u)}$,
then we have an h-confidence \eqref{eq:h-confidence},
$$
c_h(\theta,u;\mathbf{y}) = c_{0}(\theta, u; \mathbf{y}) L_{e}(\theta, u; \mathbf{y}),
$$
where $c_{0}(\theta, u; \mathbf{y}) = c'_{0}(\theta;u) c_{0}(u;\mathbf{y})$.
The marginal PD for $u$ is $c(u;\mathbf{y})$ above
and the marginal CD for $\theta$ is
\begin{align*}
c(\theta;\mathbf{y}) = \int_{\Omega_u} c_h(\theta, u; \mathbf{y}) du.
\end{align*}
Let $\textrm{CI}_{\alpha}(t) \subseteq \Theta$ 
and $\textrm{PI}_{\alpha}(t) \subseteq \Omega_{u}$ 
be observed intervals for $\theta$ and $u$ such that
\begin{align*}
\int_{\textrm{CI}_{\alpha}(t)} c(\theta;\mathbf{y}) d\theta = 1-\alpha
\quad \textrm{and} \quad
\int_{\textrm{PI}_{\alpha}(t)} c(u;\mathbf{y}) du = 1-\alpha,
\end{align*}
respectively.
The following theorem shows that
coverage probabilities of the corresponding interval procedures,
$\textrm{CI}_{\alpha}(T)$ and $\textrm{PI}_{\alpha}(T)$,
are both exactly $1-\alpha$.

\begin{theorem}
\label{thm:realized}
For any true value $\theta_0 \in \Theta$ 
and for any realized value $u_0 \in \Omega_{u}$,
\begin{align*}
P_{\theta_0}(\theta_0 \in \textrm{CI}_{\alpha}(T)) 
= 1-\alpha
\quad \text{and} \quad
P_{\theta_0}(u_0 \in \textrm{PI}_{\alpha}(T)|u_0) 
= 1-\alpha.
\end{align*}
\end{theorem}

\medskip
\noindent Example 4 (continued)
Based on $T|u\sim\textrm{Gamma}(m,u)$ and $U\sim\textrm{Exp}(\theta)$,
we can define
$$
c(u;\mathbf{y}) = \nabla_{u} P(T\geq t|u)
= \frac{t^{m}}{\Gamma(m)}u^{m-1}e^{-tu}
\quad \textrm{and} \quad
c'(\theta;u) = \nabla_{\theta} P_{\theta}(U \geq u)
= u e^{-u\theta},
$$
leading to $c_{0}(u;\mathbf{y})=t^{m}/\{u\Gamma(m)\}$
and $c'_{0}(\theta;u)=u/\theta$, respectively.
It gives the modification term for h-confidence,
$$
c_{0}(\theta, u; \mathbf{y}) 
= c'_{0}(\theta;u) c_{0}(u;\mathbf{y})
=\frac{t^{m}}{\theta \Gamma(m)}
= c_{0}(\theta; \mathbf{y})
\propto \theta^{-1},
$$
so that the h-confidence \eqref{eq:h-confidence} 
becomes \eqref{eq:hd_extendable} again,
$$
c_{h}(\theta, u; \mathbf{y})
= \frac{t^{m}}{\Gamma(m)} u^m \exp\{ -u(\theta+t) \}
= c_{0}(\theta; \mathbf{y}) L_{e}(\theta,u;\mathbf{y})
= c(\theta;\mathbf{y}) f_{\theta}(u|\mathbf{y}).
$$
This leads to the CD for fixed unknown $\theta$,
$$
c(\theta;\mathbf{y})
=\int_{0}^{\infty} c_{h}(\theta, u; \mathbf{y}) du
=\frac{mt^{m}}{(\theta +t)^{m+1}},
$$
and the PD for random unknown $u$,
\begin{equation*}
c(u;\mathbf{y})
=\int_{0}^{\infty} c_{h}(\theta, u; \mathbf{y}) d\theta
=\int_{0}^{\infty}f_{\theta}(u|\mathbf{y})c(\theta;\mathbf{y})d\theta
=\frac{t^{m}}{\Gamma(m)}u^{m-1}e^{-tu}, 
\end{equation*}
which satisfy the confidence features,
as $c(\theta;\mathbf{y})$ in Section \ref{sec:CI} 
and $c(u;\mathbf{y})=f^{*}(u;t)$ in Section \ref{sec:PI}, respectively.
Theorem \ref{thm:realized} implies 
that the h-confidence \eqref{eq:hd_extendable}
gives intervals for fixed and random unknowns,
maintaining confidence features.
\hfill \qed

\section{Approximation methods}
\label{sec:approximate}

To facilitate extended likelihood inferences,
it is crucial to find appropriate modification terms,
such as $\exp\{a(\boldsymbol{\theta};\mathbf{y})\}$ for the h-likelihood \eqref{eq:h-lik}
and $c_{0}(\boldsymbol{\theta};\mathbf{y})$ for the h-confidence \eqref{eq:hd_extendable}.
However, in general class of models with additional random unknowns,
such modification terms could often be hard to find or do not have explicit forms.
This section studies approximate methods for the h-confidence and the h-likelihood.

\subsection{Approximate h-confidence}
\label{sec:approximate_cd}
As the current CD has mainly been developed for single parameter cases,
extending h-confidence to multi-parameter cases could be challenging.
This section proposes approximate h-confidences 
for multi-dimensional $\boldsymbol{\theta}$ and $\mathbf{v}$,
which asymptotically achieve the confidence feature 
as long as $\widehat{\boldsymbol{\theta}}\overset{p}{\to}\boldsymbol{\theta}$.
If we have a proper CD for $\boldsymbol{\theta}$,
the definition of h-confidence \eqref{eq:hd_extendable}
can be naturally extended to multi-parameter cases as
$$
c_{h}(\boldsymbol{\theta},\mathbf{v};\mathbf{y})
= c_0(\boldsymbol{\theta};\mathbf{y}) f_{\boldsymbol{\theta}}(\mathbf{y},\mathbf{v})
= c(\boldsymbol{\theta};\mathbf{y}) f_{\boldsymbol{\theta}}(\mathbf{v}|\mathbf{y}),
$$
leading to the PD for $\mathbf{v}$,
$$
c(\mathbf{v};\mathbf{y}) = \int_{\boldsymbol{\Theta}} c(\boldsymbol{\theta}; \mathbf{y}) 
f_{\boldsymbol{\theta}}(\mathbf{v}|\mathbf{y}) d \boldsymbol{\theta},
$$
analogous to the marginal posterior of $\mathbf{v}$ \eqref{eq:marginal_posterior}
by replacing the marginal posterior $\pi(\boldsymbol{\theta}|\mathbf{y})$ for $\boldsymbol{\theta}$ 
with the CD $c(\boldsymbol{\theta};\mathbf{y})$.
\citet{fraser84} showed that the normalized marginal likelihood
and the asymptotic distribution of MLEs
are asymptotically equivalent to the CD.
The Jeffreys prior leads to the normalized likelihood 
for a specific scale of $\boldsymbol{\theta}$ \citep{lee24induction}, 
which is also asymptotically equivalent to the CD.
Thus, we consider the following three approximate h-confidences:

\begin{itemize}
\item[] (C1)
Where $\pi_{\textrm{J}}(\boldsymbol{\theta}|\mathbf{y})$ is the marginal posterior
density of $\boldsymbol{\theta}$ under the Jeffreys prior,
\begin{equation*}
c_{1}(\boldsymbol{\theta},\mathbf{v};\mathbf{y})
=\pi_{\textrm{J}}(\boldsymbol{\theta}|\mathbf{y}) \
f_{\boldsymbol{\theta}}(\mathbf{v}|\mathbf{y}).
\end{equation*}

\item[] (C2)
Where $L_{n}(\boldsymbol{\theta};\mathbf{y})
=L(\boldsymbol{\theta};\mathbf{y})
/\int_{\Theta}L(\boldsymbol{\theta};\mathbf{y})d\boldsymbol{\theta}$ 
is the normalized likelihood,
\begin{equation*}
c_{2}(\boldsymbol{\theta},\mathbf{v};\mathbf{y})
=L_{n}(\boldsymbol{\theta};\mathbf{y}) \
f_{\boldsymbol{\theta}}(\mathbf{v}|\mathbf{y}).
\end{equation*}

\item[] (C3)
Where $\phi(\widehat{\boldsymbol{\theta}},\widehat{I}^{\boldsymbol{\theta}\boldsymbol{\theta}})$
is the probability density of the asymptotic normal distribution $N(\widehat{\boldsymbol{\theta}},\widehat{I}^{\boldsymbol{\theta}\boldsymbol{\theta}})$,
\begin{equation*}
c_{3}(\boldsymbol{\theta},\mathbf{v};\mathbf{y})
=\phi (\widehat{\boldsymbol{\theta}},\widehat{I}^{\boldsymbol{\theta}\boldsymbol{\theta}}) \
f_{\boldsymbol{\theta}}(\mathbf{v}|\mathbf{y}).
\end{equation*}
\end{itemize}
Results in Section \ref{sec:finite} give justifications of C1-C3 
as approximate h-confidences.
Approximate h-confidences C1-C3 give the PDs considered by \citet{lee16}.
As long as $\widehat{\boldsymbol{\theta}}\overset{p}{\to}\boldsymbol{\theta}_0$,
it can be generally shown that the approximate PDs are asymptotically equivalent:
\begin{equation*}
c_{k}(\mathbf{v};\mathbf{y})
= \int_{\boldsymbol{\Theta}} c_{k}(\boldsymbol{\theta},\mathbf{v};\mathbf{y}) d\boldsymbol{\theta}
\overset{p}{\to} f_{\boldsymbol{\theta}_{0}}(\mathbf{v}|\mathbf{y}).  
\end{equation*}
Even though the MHLE of $\mathbf{v}$ can be neither consistent nor asymptotically normal,
we can still obtain asymptotically correct PD $f_{\theta_{0}}(\mathbf{v}|\mathbf{y})$, 
as long as the MLE is consistent, i.e., $\widehat{\boldsymbol{\theta}}\overset{p}{\to}\boldsymbol{\theta}_0$.
The resulting PDs from C1-C3 asymptotically satisfy the confidence feature 
for both marginal and conditional predictions,
whereas the Wald PI satisfies it only for conditional predictions.
Thus, C1-C3 can overcome the difficulty
when $\widehat{\mathbf{v}}-\mathbf{v}=O_p(1)$
as in the missing data problem, prediction of future events, longitudinal studies, etc.
\citet{lee16} conducted simulation studies for various models with independent random effects,
demonstrating that C1-C3 and the normalized profile h-likelihood perform similarly,
improving Wald PIs of \citet{paik15} and the plug-in method 
for both marginal and conditional PIs in small samples.
C1 and C2 often involve intractable integration
as the dimension of $\boldsymbol{\theta}$ increases.
The use of the normalized profile h-likelihood is also computationally demanding.
\citet{cao18} showed that C3 can be straightforwardly computed
by using the bootstrap method,
which provides almost exact PIs 
for heavy-tailed process functional regression models in finite samples.

\subsubsection{Approximate PD for marginal prediction}

For marginal prediction of future random unknowns in the exponential model,
\citet{lee16} showed that
the normalized profile h-likelihood (the PD of \citet{lee09})
is equivalent to the marginal posterior $\pi(\mathbf{v}|\mathbf{y})$ 
under the Jeffreys prior,
which can be shown to be the same as our PD, $c_{1}(\mathbf{v};\mathbf{y})$,
from the h-confidence. 
Via simulation studies for various models with $n$ random effects,
they showed that C1-C3 and the normalized profile h-likelihood 
maintain the confidence feature well even when $m=1$: 
for a $100(1-\alpha)\%$ PI for $\textit{\textbf{V}}$ 
based on $c(\mathbf{v};\mathbf{y})$, 
\begin{equation*} \label{eq:approximate_cd_marginal}
P_{\theta_{0}}(\textit{\textbf{V}}\in \textrm{PI}(\textit{\textbf{Y}}))
\approx C(\textit{\textbf{V}}\in \textrm{PI}(\mathbf{y}))=1-\alpha, 
\end{equation*}
i.e., the coverage probability of a PI procedure (LHS) is very close to the
confidence level of an observed PI (RHS). 
Thus, the approximate PDs are satisfactory 
for marginal prediction of $\textit{\textbf{V}}$.

\subsubsection{Approximate PD for conditional prediction}

\citet{lee16} showed that the PI for an individual realized value 
maintains the coverage probability well. 
However, care is necessary when PIs are made simultaneously for all $n$ realized values. 
For $100(1-\alpha)\%$ simultaneous PI procedure of all the realized values, note that
\begin{equation*}
C(v_{0i}\in \textrm{PI}_{i}(\mathbf{y}))=1-\alpha \ 
\approx \ \frac{1}{n}\sum_{i=1}^{n}
P_{\theta_{0}}(v_{0i}\in \textrm{PI}_{i}(\textit{\textbf{Y}})|V_{i}=v_{0i})
\overset{p}{\to} P_{\theta_{0}}(V_{i}\in \textrm{PI}_{i}(\textit{\textbf{Y}})), 
\end{equation*}
i.e., the confidence level of an observed PI for a certain realized value (LHS) 
is close to the average coverage probability $1-\alpha$ of the PI procedure (RHS).
As \citet{dawid86} noted, the future values are exchangeable, while the
realized values are not: for marginal PIs, 
\begin{equation*}
P_{\theta_{0}}(V_{i}\in \textrm{PI}_{i}(\textit{\textbf{Y}}))=P_{\theta
_{0}}(V_{j}\in \textrm{PI}_{j}(\textit{\textbf{Y}}))\quad \textrm{for all }i,j, 
\end{equation*}
whereas for conditional PIs, 
\begin{equation*}
P_{\theta_{0}}(v_{0i}\in \textrm{PI}_{i}(\textit{\textbf{Y}})|V_{i}=v_{0i})
\neq P_{\theta_{0}}(v_{0j}\in \textrm{PI}_{j}(\textit{\textbf{Y}})|V_{j}=v_{0j})
\quad \textrm{for }v_{0i}\neq v_{0j}. 
\end{equation*}
Via simulation studies,
\citet{lee16} noted that the coverage probability 
can be liberal at extreme value of $v_{0i}$, 
$P_{\theta_{0}}(v_{0i} \in \textrm{PI}_{i}(\textit{\textbf{Y}})|V_{i}=v_{0i})<1-\alpha $, 
and conservative at moderate value of $v_{0i}$, 
$P_{\theta_{0}}(v_{0i} \in \textrm{PI}_{i}(\textit{\textbf{Y}})|V_{i}=v_{0i})>1-\alpha$,
because the average coverage probability of simultaneous PIs is $1-\alpha$.
To have uniform coverage probabilities for a simultaneous PI procedure,
$m$ needs to be large \citep{lee16}. 
Further studies are required for the simultaneous
conditional PIs to maintain the coverage probability uniformly 
for all the realized values.
A real data example of simultaneous PIs is provided in Section \ref{app:ela} of the Appendix.

\subsection{Approximate h-likelihood}

The construction of the h-likelihood \eqref{eq:h-lik}
necessitates identifying 
a Bartlizable scale $\mathbf{v}$ and its modification term $a(\boldsymbol{\theta};\mathbf{y})$.
While a Bartlizable scale can be easily identified 
using Lemma \ref{lem:bartlett}, 
the modification term may not have an explicit form in general.
In such cases, the h-likelihood can be approximated
by adapting approximation methods for the marginal likelihood
$L(\boldsymbol{\theta};\mathbf{y})$.
Recently, \citet{han22b} proposed the use of enhanced Laplace approximation (ELA)
to obtain the MLEs and their variance estimators.
They employed the Monte Carlo sampling method 
to approximate the marginal likelihood $L(\boldsymbol{\theta};\mathbf{y})$ by 
\begin{equation*}
L_{B}(\boldsymbol{\theta};\mathbf{y})
=\frac{1}{B}\sum_{b=1}^{B}
\frac{L_{e}(\boldsymbol{\theta},\mathbf{v}_{b})}{q(\mathbf{v}_{b})}, 
\end{equation*}
where $\mathbf{v}_{b}$ are independent samples from an arbitrary
probability density function $q(\cdot)$, having the same support
with $\mathbf{v}$. They used a multivariate normal distribution for $q(\mathbf{v})$, 
\begin{equation*}
N\left( \widetilde{\mathbf{v}},
\left[ \nabla_{\mathbf{v}}^{2}h(\boldsymbol{\theta},\mathbf{v})
\right]_{\mathbf{v}=\widetilde{\mathbf{v}}}^{-1}\right).
\end{equation*}
When $B=1$ and $\mathbf{v}_{b}=\widetilde{\mathbf{v}}$, 
the ELA becomes equivalent to the LA, 
which is exact when the predictive likelihood is normal.
Numerical studies of \citet{han22b} demonstrate that the ELA performs well
even in binary data with temporally and spatially correlated random effects.
They further extended the ELA to restricted MLEs (REMLEs).
The restricted likelihood is a modified marginal likelihood,
which has been proposed for inferences 
when the dimension of the fixed effects $\boldsymbol{\beta}$ increases with the sample size $N$.
From \eqref{eq:htilde}, it can be viewed as another modified extended likelihood.

We adapt the ELA for approximation of the h-likelihood.
When $a(\boldsymbol{\theta};\mathbf{y})
=-\ell_{p}(\widetilde{\mathbf{v}};\mathbf{y})
=\ell (\boldsymbol{\theta};\mathbf{y})-\ell_{e}(\boldsymbol{\theta};\widetilde{\mathbf{v}})$ 
is not explicitly known, it can be approximated by 
$\log L_{B}(\boldsymbol{\theta};\mathbf{y})
-\ell_{e}(\boldsymbol{\theta},\widetilde{\mathbf{v}})$ 
to give an approximated h-likelihood, 
\begin{equation}
h_{B}(\boldsymbol{\theta},\mathbf{v})
=\ell_{e}(\boldsymbol{\theta},\mathbf{v})
-\ell_{e}(\boldsymbol{\theta},\widetilde{\mathbf{v}})
+\log L_{B}(\boldsymbol{\theta};\mathbf{y}).
\label{eq:h-ela}
\end{equation}
Then, we can obtain the approximate h-information, 
\begin{equation*}
I_{B}(\boldsymbol{\theta},\mathbf{v})
=-\nabla^{2}_{\boldsymbol{\theta},\mathbf{v}}\ell_{e}(\boldsymbol{\theta},\mathbf{v})
+\left[\nabla^{2}_{\boldsymbol{\theta},\mathbf{v}}
\ell_{e}(\boldsymbol{\theta},\mathbf{v})\right]_{\mathbf{v}=\widetilde{\mathbf{v}}}
- \textrm{block-diag}\{ I_{11}(\boldsymbol{\theta}), \mathbf{0} \}, 
\end{equation*}
where
$L_e^{(b)} = L_{e}\left(\boldsymbol{\theta},\mathbf{v}^{(b)}\right)$,
$\ell_{e}^{(b)}=\log L_{e}^{(b)}$,
$w_{b} = \{L_{e}^{(b)}/q(\mathbf{v}^{(b)})\}
\{ \sum_{b=1}^{B}
L_{e}^{(b)}/q(\mathbf{v}^{(b)})
\}^{-1}$
and
\begin{align*}
I_{11}(\boldsymbol{\theta})
=& \left[ \sum_{b=1}^{B}
w_{b}
\nabla_{\boldsymbol{\theta}} \ell_{e}^{(b)}
\right] 
\left[ \sum_{b=1}^{B}
w_{b}
\nabla_{\boldsymbol{\theta}} \ell_{e}^{(b)}
\right]^{\intercal}
- \sum_{b=1}^{B}w_{b}\left[
\left(\nabla_{\boldsymbol{\theta}} \ell_{e}^{(b)}\right)
\left(\nabla_{\boldsymbol{\theta}} \ell_{e}^{(b)}\right)^\intercal
+\nabla_{\boldsymbol{\theta}}^2 \ell_{e}^{(b)}
\right].
\end{align*}

\begin{theorem}
\label{thm:h-ela} Suppose that $(\widehat{\boldsymbol{\theta}},\widehat{\mathbf{v}})$ 
are MHLEs from the h-likelihood and 
$(\widehat{\boldsymbol{\theta}}_{B},\widehat{\mathbf{v}}_{B})$ are approximate MHLEs
from the approximate h-likelihood \eqref{eq:h-ela}. Then, as $B\to\infty$,
\begin{equation*}
(\widehat{\boldsymbol{\theta}}_{B},\widehat{\mathbf{v}}_{B})
\overset{p}{\to }(\widehat{\boldsymbol{\theta}},\widehat{\mathbf{v}})
\quad \textrm{and} \quad 
I_{B}(\widehat{\boldsymbol{\theta}}_{B},\widehat{\mathbf{v}}_{B})
\overset{p}{\to }I(\widehat{\boldsymbol{\theta}},\widehat{\mathbf{v}}). 
\end{equation*}
\end{theorem}

\begin{corollary}
\label{cor:ela} Under the regularity conditions R1-R6
in Section \ref{app:proof} of the Appendix, 
as $B\to\infty$,
\begin{equation*}
(\widehat{\boldsymbol{\theta}}_{B}-\boldsymbol{\theta},
\ \widehat{\mathbf{v}}_{B}-\mathbf{v})^\intercal
\overset{d}{\to }
N (\mathbf{0},I(\widehat{\boldsymbol{\theta}},\widehat{\mathbf{v}})).
\end{equation*}
\end{corollary}

\noindent Thus, using the approximate h-likelihood, we can obtain
MHLEs and their consistent variance estimators.
As $B$ increases, we can obtain more accurate MHLEs, 
but the computational cost also increases. 
It is worth noting that Theorem \ref{thm:h-ela} holds 
for any choice of sampling distribution $q(\mathbf{v})$. 
However, the convergence rate would depend on the choice of $q(\mathbf{v})$,
since it could require large $B$ when the predictive likelihood is too far from $q(\mathbf{v})$. 
In Section \ref{app:ela} of the Appendix, 
we demonstrate that the use of the conjugate distribution for $q(\mathbf{v})$
can improve the computational efficiency.

\section{Conclusion}
\label{sec:conclusion}

For fixed unknowns, 
\citet{fisher22} introduced the MLEs for point estimation,
and \citet{fisher35} proposed the fiducial probability for interval estimation.
The extended likelihood has been proposed for
inferences of statistical models with additional random unknowns,
but it fails to give optimal estimation and prediction.
This paper introduces the new h-likelihood and h-confidence 
by modifying the extended likelihood.
The new h-likelihood properly extends the ML procedure.
The MHLEs provide asymptotically optimal estimation and prediction 
for fixed and random parameters,
and the Hessian matrix provides consistent variance estimators of the MHLEs.
Furthermore, the h-likelihood offers advantages
in scalability for large datasets and complex models,
compared with the marginal likelihood.
For interval estimation, h-confidence extends CD for fixed unknowns
to provide exact interval estimation and prediction for 
fixed and random unknowns, even in small samples.
Thus, our modified extended likelihoods 
can provide not only asymptotically optimal point estimators and predictors
but also exact interval estimators and predictors for small samples.
For the general class of models,
approximate h-likelihood gives asymptotically optimal MHLEs
and approximate h-confidence gives 
asymptotically correct interval prediction for random unknowns 
as long as the MLEs are consistent.
The main elements of Fisherian likelihood theory
are consistency, sufficiency, Fisher information, efficiency, 
asymptotic optimality of MLEs, and ancillarity \citep{efron98}.
This paper extends them
to accommodate both fixed and random unknowns.
Building upon the new theory for prediction of unobserved random unknowns, 
we anticipate the necessity for more comprehensive 
and diverse avenues of future research.

\bibliographystyle{apalike}
\bibliography{reference}

\newpage
\begin{appendix}

\section{Numerical study for Poisson-gamma HGLM}
\label{app:PG-HGLM}

This section presents a numerical study for Poisson-gamma HGLM in Example 2 of the main paper.
Similar to the numerical study in Section 4.3 of the main paper,
we investigate the properties of MHLEs for the Poisson-gamma HGLM
with the linear predictor
$$
\eta_{ij} = \log \mu_{ij} = \beta_0 + \beta_1 x_{ij} + \log u_{i},
$$
where $x_{ij}$ is generated from Uniform$[-0.5,0.5]$
and all the true parameters are set to be 0.5.
To see the convergence, 
we let $m$ or $n$ increase in $\{5, 20, 80\}$.
Figure \ref{fig:consistency_pg} shows the box-plots 
for the estimation errors of each parameter from 500 replications.
The MHLE of $\beta_{1}$ for the within-cluster covariate $x_{ij}$
is consistent when either $m$ or $n$ increases.
However, the MHLE of $\beta_{0}$ and 
the MHLE of the between-cluster variance $\lambda$
are consistent only when $n$ increases.
Furthermore, the MHLE of a random parameter $u_1$
is consistent when both $m$ and $n$ increase:
$\widetilde{u}_{1}-u_{1}\overset{p}{\to}0$ when $m\to\infty$,
whereas $\widehat{u}_{1}-\widetilde{u}_{1}=\widetilde{u}_{1}(\widehat{\boldsymbol{\theta}})-\widetilde{u}_{1}(\boldsymbol{\theta})\overset{p}{\to}0$
when $\widehat{\boldsymbol{\theta}}\overset{p}{\to}\boldsymbol{\theta}$ as $n\to\infty$.
Thus, similar to the result for LMMs in the main paper,
$\widehat{u}_{1}-u_{1}\to 0$ when both $m$ and $n$ increase.

\begin{figure}[tbp]
\centering
\includegraphics[width=\linewidth]{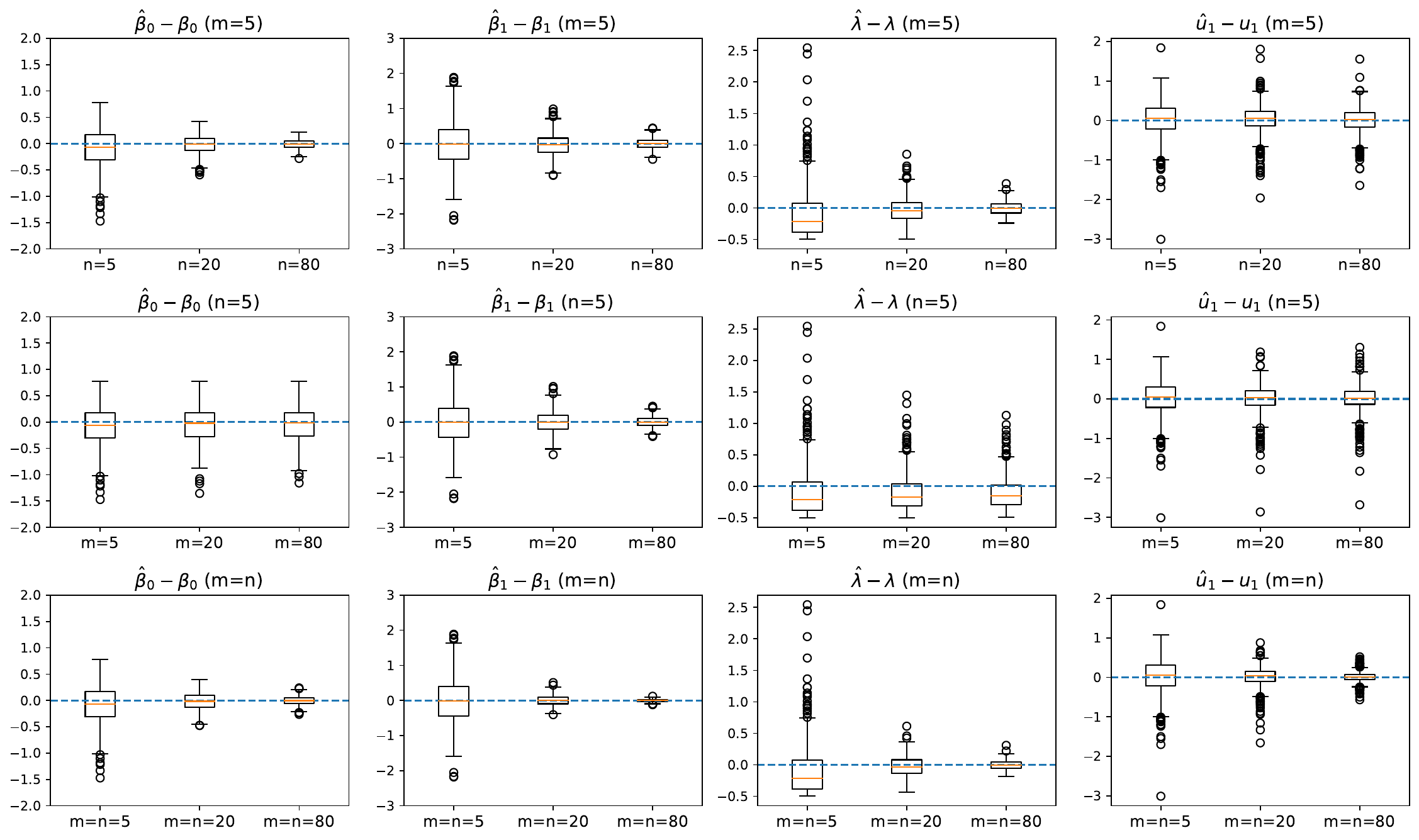}
\caption{Box-plots for the estimation errors of fixed and random unknowns in Poisson-gamma HGLM.\label{fig:consistency_pg}}
\end{figure}

\section{BUP for the conditional mean}
\label{app:BUP}

Suppose that $\widetilde{\mathbf{v}}=\textrm{E}(\mathbf{v}|\mathbf{y})$ 
is the BUP of a random unknown $\mathbf{v}$
and consider a plug-in predictor $g(\widetilde{\mathbf{v}})$
for an arbitrary non-linear transformation $g(\mathbf{v})$.
Then, asymptotically the predictor $g(\widetilde{\mathbf{v}})$ 
achieves the BUP property for any scale $g(\mathbf{v})$ 
such that $\textrm{E}\{|g(\mathbf{v})||\mathbf{y}\}<\infty$.
However, it is no longer the exact BUP in general:
$$
g(\widetilde{\mathbf{v}})=g(\textrm{E}(\mathbf{v}|\mathbf{y}))
\neq \textrm{E}(g(\mathbf{v})|\mathbf{y}).
$$
As the conditional mean $\boldsymbol{\mu}=\textrm{E}(\mathbf{y}|\mathbf{v})$ 
is also a function of $\boldsymbol{\theta}$ and $\mathbf{v}$, 
it could often be preferred to use a specific scale of $\mathbf{v}$,
leading to the BUP for the conditional mean.
Consider an HGLM with the linear predictor, 
\begin{equation*}
\eta_{ij}=g(\mu_{ij})=\mathbf{x}_{ij}^{\intercal}\boldsymbol{\beta}+v_{i},
\end{equation*}
and suppose that we want to find a predictor $\widetilde{v}_{i}$ 
that gives the exact BUP of each conditional mean $\mu_{ij}$ in small samples, 
\begin{equation}
\label{eq:BUP-mu}
\widetilde{\mu}_{ij}
=g^{-1}(\mathbf{x}_{ij}^{\intercal}\boldsymbol{\beta}+\widetilde{v}_{i})
=\textrm{E}(\mu_{ij}|\mathbf{y}).
\end{equation}
We have shown that such scales $\widetilde{v}_{i}$ exist 
in LMMs and Poisson-gamma HGLMs.
However, the following theorem
shows that such scale $\widetilde{v}_{i}$ may not always exist,
for example, in binomial HGLMs with the logit link.

\begin{theorem}
\label{thm:bup_mean} 
Suppose that $\boldsymbol{\beta}\neq \mathbf{0}$ and $\mathbf{x}_{ij}^{\intercal}\boldsymbol{\beta}$ can freely move on $\mathbb{R}$.
Given $\boldsymbol{\theta}$, 
there exists $\widetilde{v}_{i}$ satisfying \eqref{eq:BUP-mu} 
if and only if there exists a constant $c_{i}\in \Omega_{v}$ such that 
\begin{equation}
\mathrm{E}\left( g_{(k)}^{-1}(v_{i})\Big|\mathbf{y}\right)
=g_{(k)}^{-1}(c_{i})\quad \textrm{for all }k=0,1,2,...  \label{eq:iff_mu_bup}
\end{equation}
where $g_{(k)}^{-1}(\cdot)$ is the $k$-th order derivative of inverse link $g^{-1}(\cdot)$.
\end{theorem}
\begin{proof}
$(\Leftarrow)$ By the Taylor expansion and equation \eqref{eq:iff_mu_bup},
\begin{align*}
\textrm{E} \left[ g^{-1}(\mathbf{x}_{ij}^\intercal\boldsymbol{\beta}  + v_i) |\mathbf{y} \right] 
&= \textrm{E} \left[ \sum_{k=1}^{\infty} \frac{g^{-1}_{(k)}(v_i)}{k!}
\left(\mathbf{x}_{ij}^\intercal\boldsymbol{\beta} \right)^k \Bigg|\mathbf{y}  \right]
\\
&= \sum_{k=1}^{\infty} \frac{\textrm{E} \left(g^{-1}_{(k)}(v_i) |\mathbf{y} 
\right)}{k!} \left(\mathbf{x}_{ij}^\intercal\boldsymbol{\beta} \right)^k 
\\
&= \sum_{k=1}^{\infty} \frac{g^{-1}_{(k)}(c_i)}{k!} \left(\mathbf{x}_{ij}^\intercal\boldsymbol{\beta} \right)^k 
= g^{-1}\left(\mathbf{x}_{ij}^\intercal\boldsymbol{\beta} + c_i\right).
\end{align*}
Thus, taking $\widehat{v}_i = c_i$ proves the existence of $\widehat{\mathbf{v}}$ 
which gives the BUP for each $\mu_{ij}$.

\noindent $(\Rightarrow)$ Since the equation \eqref{eq:BUP-mu} hold 
for any value of $\mathbf{x}_{ij}^\intercal\boldsymbol{\beta}  \in \mathbb{R}$, 
substituting $\mathbf{x}_{ij}^\intercal \boldsymbol{\beta} = t$ and $c=\widehat{v}_i$ implies that 
$$
\textrm{E} \left(g^{-1}(v_i+t)|\mathbf{y} \right) 
= g^{-1}(c+t) \quad \forall t \in \mathbb{R}. 
$$
If the equation 
$\textrm{E} \left( g^{-1}_{(k-1)}(v_i+t)|\mathbf{y} \right) 
= g^{-1}_{(k-1)} (c+t)$ holds for all $t \in \mathbb{R}$, then 
\begin{align*}
\textrm{E} \left( g^{-1}_{(k)}(v_i+t)|\mathbf{y}  \right) 
& = \textrm{E} \left[ \lim_{\epsilon \to 0} \frac{g^{-1}_{(k-1)}(v_i+t+\epsilon) 
- g^{-1}_{(k-1)}(v_i+t)}{\epsilon} \Bigg| \mathbf{y}  \right] \\
&= \lim_{\epsilon \to 0} \frac{1}{\epsilon} 
\left[ \textrm{E} \left(g^{-1}_{(k-1)}(v_i+t+\epsilon) | \mathbf{y}  \right) 
- \textrm{E} \left(g^{-1}_{(k-1)}(v_i+t) | \mathbf{y}  \right) \right] \\
&= \lim_{\epsilon \to 0} \frac{1}{\epsilon} 
\left[ g^{-1}_{(k-1)}(c+t+\epsilon) - g^{-1}_{(k-1)}(c+t) \right] \\
&= g^{-1}_{(k)}(c+t).
\end{align*}
By induction, it is proved that for any $k=0, 1, 2, ...,$ 
$$
\textrm{E} \left(g^{-1}_{(k)}(v_i+t)|\mathbf{y}  \right) 
= g^{-1}_{(k)} (c+t)
\quad \forall t \in \mathbb{R}. 
$$
Substituting $t=0$ ends the proof.
\end{proof}

When $g(\cdot)$ is the identity link, 
$g_{(k)}^{-1}(v_{i})=0$ for all $k\geq 1$
and $c_{i}=\textrm{E}(v_{i}|\mathbf{y})$ satisfies the condition. 
When $g(\cdot)$ is the logarithm, 
$g_{(k)}^{-1}(v_{i})=\exp (v_{i})$ for all $k$
and $c_{i}=\log \{\textrm{E}(\exp (v_{i})|\mathbf{y})\}$ satisfies the condition.
Thus, in LMMs and Poisson-gamma HGLMs with the canonical link,
we can have a specific scale of $\mathbf{v}$ to obtain the BUP of $\boldsymbol{\mu}$.
However, in binomial HGLMs with the canonical link, 
we cannot have such a scale:
Consider a binomial HGLM,
\begin{equation*}
\eta_{ij}=\log \left( \frac{p_{ij}}{1-p_{ij}}\right) 
=\mathbf{x}_{ij}^{\intercal}\boldsymbol{\beta}+v_{i},
\end{equation*}
where $p_{ij}=P(Y_{ij}=1|\mathbf{x}_{ij},v_{i})$.
If there exists a predictor satisfying \eqref{eq:BUP-mu},
then there exists a constant $c_{i}$ satisfying \eqref{eq:iff_mu_bup}
by Theorem \ref{thm:bup_mean}.
Note here that $g_{(0)}^{-1}(v_{i})=\{1+\exp(-v_{i})\}^{-1}$ and 
\begin{equation*}
g_{(1)}^{-1}(v_{i})
=\exp (-v_{i})\{1+\exp(-v_{i})\}^{-2}=g_{(0)}^{-1}(v_{i})-\{g_{(0)}^{-1}(v_{i})\}^{2}. 
\end{equation*}
Let $T=g_{(0)}^{-1}(v_{i})$ and $t=g_{(0)}^{-1}(c_{i})$. 
If $\textrm{E}(T)=t$ holds, then $\textrm{E}(T-T^{2})<\textrm{E}(T)-\textrm{E}(T)^{2}=t-t^{2}$ 
by the Jensen's inequality. 
This implies that the equation \eqref{eq:iff_mu_bup} cannot be satisfied 
simultaneously for both $k=0$ and $k=1$. 
Thus, with the logit link,
BUP property for the conditional mean is only available asymptotically,
which can be provided by the MHLEs.

\section{Details for exponential-exponential HGLM (Example 3)}
\label{app:details}

\subsection{Tightness of GCRLB}
Consider the exponential-exponential HGLM with $n=1$ in Example 3.
For given $\theta$, the MHLE of $u$ is the BUP,
\begin{equation*}
\widetilde{u}=\exp (\widetilde{v})
=\exp \left[ \argmax_{v}h(\theta,v)\right] 
=\frac{m+1}{\theta +m\bar{y}}=\textrm{E}(u|\mathbf{y}), 
\end{equation*}
which leads to 
\begin{equation*}
\textrm{E}(u-\widetilde{u}|\mathbf{y})=0
\quad \textrm{and}\quad 
\textrm{Var}(u-\widetilde{u}|\mathbf{y})
=\frac{m+1}{(\theta +m\bar{y})^{2}}. 
\end{equation*}
Here, the marginal variance $\textrm{Var}(u-\widetilde{u})$ is
\begin{align*}
\textrm{Var}(u-\widetilde{u})
&=\textrm{E}\{\textrm{Var}(u|\mathbf{y})\}
=\textrm{E}\left[ \frac{m+1}{(\theta +m\bar{y})^{2}}\right] 
=\frac{2\theta^{-2}}{m+2}
=O(1/m),
\end{align*}
which implies that for given $\theta$,
$\widetilde{u}$ is a consistent predictor of a random unknown $u$. 
However, the GCRLB becomes the Bayesian CRLB, 
which may not be tight in small samples \citep{van07}: 
\begin{equation*}
\textrm{Var}(u-\widetilde{u})
=\frac{2\theta^{-2}}{m+2}
\geq \textrm{E}\left[ \frac{\partial u}{\partial v}\right]^{2}
\textrm{E}\left[ -\frac{\partial^{2}h(\theta,v)}{\partial v^{2}}\right]^{-1}
=\frac{\theta^{-2}}{m+1}. 
\end{equation*}
Here, for any $\theta>0$,
the Bayesian CRLB cannot be achieved
by the BUP $\widetilde{u}=\textrm{E}(u|\mathbf{y})$ in finite samples.

Joint maximization of $h(\theta,v)$ leads to the MHLEs 
$\widehat{\theta}=\bar{y}$ and $\widehat{v}= -\log \bar{y}$. 
It is worth noting that 
$\widehat{u}=\exp(\widehat{v})=1/\bar{y}$
is a consistent predictor of $u$. 
Here, $\textrm{Var}(u-\widehat{u})$ can be decomposed as 
\begin{equation*}
\textrm{Var}(u-\widehat{u})
=\textrm{Var}(u-\widetilde{u}(\theta))
+\textrm{Var}(\widetilde{u}(\theta)-\widetilde{u}(\widehat{\theta})),
\end{equation*}
and the GCRLB can be decomposed by block
matrix inversion, 
\begin{align*}
\mathcal{Z}\mathcal{I}(\theta)^{-1}\mathcal{Z}^{\intercal}
&=\frac{\theta^{-2}}{m+1}+\frac{\theta^{-2}}{m(m+1)}=\frac{\theta^{-2}}{m}.
\end{align*}
The first term $\theta^{-2}/(m+1)$ is the Bayesian CRLB,
which cannot be achieved by the BUP as derived above. 
If $\widehat{\theta}-\theta=o_{p}(1)$ and asymptotically normal, 
the delta method would lead to
\begin{equation*}
\textrm{Var}(\widetilde{u}(\theta)-\widetilde{u}(\widehat{\theta}))
\approx 
\left( \frac{\partial \widetilde{u}(\theta)}{\partial \theta}\right)^{2}
\textrm{Var}(\widehat{\theta}-\theta)
=\frac{\theta^{-2}}{m(m+1)},
\end{equation*}
which achieves the bound $\theta^{-2}/\{m(m+1)\}$.
However, in this example, $\widehat{\theta}-\theta=O_{p}(1)$ and 
$\textrm{Var}(\widehat{\theta}-\theta)=\infty $ lead to 
\begin{equation*}
\textrm{Var}(\widetilde{u}(\theta)-\widetilde{u}(\widehat{\theta}))
=\frac{\theta^{-2}(13m^{2}-10m)}{(m-2)(m+2)(m-1)^{2}}
\geq \frac{\theta^{-2}}{m(m+1)}. 
\end{equation*}
Thus, as $m$ grows, the MHLE 
$\widehat{u}=\widetilde{u}(\widehat{\theta})$ 
is asymptotically the BUP
$\widetilde{u}=\widetilde{u}(\theta)$.
However, the GCRLB is not tight: 
\begin{equation*}
\textrm{Var}(u-\widehat{u})
=\textrm{Var}(u-\widetilde{u}(\theta))
+\textrm{Var}(\widetilde{u}(\theta)
-\widetilde{u}(\widehat{\theta}))
\geq \frac{\theta^{-2}}{m}
=\mathcal{Z}\mathcal{I}(\theta)^{-1}\mathcal{Z}^{\intercal},
\end{equation*}
and the LHS becomes infinity for $m\leq 2$.
Further study to find a tighter lower bound for such cases would be
interesting, as it was in the Bayesian literature 
\citep{ziv69, bobrovsky76, weiss85}.

\subsection{Reparameterization of fixed parameters}
Here, we consider a reparameterization $\xi=1/\theta$, 
then the MLE $\widehat{\xi}=1/\bar{y}=\widehat{u}$ becomes 
asymptotically the best unbiased estimator of $\xi$. 
For inference, the scale of fixed parameter is important 
when the MLE is not consistent. 
Reparameterization $\xi=1/\theta$ gives 
\begin{equation*}
f_{\xi}(u)=\frac{1}{\xi}\exp \left( -\frac{u}{\xi}\right), 
\end{equation*}
leading to the classical log-likelihood 
\begin{equation*}
\ell (\xi )=-\log \xi -(m+1)\log (\xi^{-1}+m\bar{y})+\log \Gamma(m+1), 
\end{equation*}
where $\Gamma(\cdot)$ is the gamma function. Then, we have the
h-likelihood 
\begin{equation*}
h(\xi,v)=-\log \xi -e^{v}(\xi^{-1}+m\bar{y})+v(m+1)-(m+1)\log
(m+1)+\log \Gamma(m+1)+(m+1), 
\end{equation*}
with the expected h-information matrix 
\begin{equation*}
\mathcal{I}(\xi )=\textrm{E}\{I(\xi,v)\}
=\textrm{E}\left[ \frac{-\partial^{2}h(\xi,v)}{\partial (\xi,v)\partial (\xi,v)^{\intercal}}\right] 
=\begin{bmatrix}
\xi^{-2} & -\xi^{-1} \\ 
-\xi^{-1} & m+1
\end{bmatrix}. 
\end{equation*}
The observed h-information yields variance estimators, 
\begin{equation*}
\widehat{\textrm{Var}}
\begin{bmatrix}
\xi -\widehat{\xi} \\ 
v-\widehat{v}
\end{bmatrix}
={I}(\widehat{\xi},\widehat{v})^{-1}=
\begin{bmatrix}
\bar{y}^{2} & -\bar{y} \\ 
-\bar{y} & m+1
\end{bmatrix}
^{-1}=\frac{1}{m}
\begin{bmatrix}
(m+1)\bar{y}^{-2} & \bar{y}^{-1} \\ 
\bar{y}^{-1} & 1
\end{bmatrix}. 
\end{equation*}
The h-likelihood and the classical likelihood give the same estimation for $\xi$, 
and their information matrices give the same variance estimator. 
The MHLE of $\xi$ is the MLE $\widehat{\xi}=1/\bar{y}$, 
which is an asymptotically unbiased estimator of $\xi$ with 
\begin{align*}
& \textrm{E}(\widehat{\xi})
=\textrm{E}(1/\bar{y})
=\textrm{E}(\textrm{E}(1/\bar{y}|u))
=\frac{m\xi }{m-1}\to \xi, \\
& \textrm{Var}(\widehat{\xi})
=\textrm{Var}(1/\bar{y})
=\textrm{E}(\textrm{Var}(1/\bar{y}|u))+\textrm{Var}(\textrm{E}(1/\bar{y}|u))
=\frac{m^{3}\xi^{2}}{(m-1)^{2}(m-2)}
\to \xi^{2} < \infty,
\end{align*}
as $m\to \infty$. Even though the MHLE (MLE) of $\xi$ is not consistent,
it is still asymptotically the best unbiased estimator that achieves the GCRLB, 
\begin{align*}
\textrm{Var}(\xi -\widehat{\xi})=\frac{\xi^{2}m^{3}}{(m-1)^{2}(m-2)}
& =\frac{\xi^{2}(m+1)}{m}\left[ 1+O\left( \frac{1}{m}\right) \right] \\
& \geq (1,0)\ \mathcal{I}(\xi,v)^{-1}\ (1,0)^{\intercal}
=\frac{\xi^{2}(m+1)}{m}.
\end{align*}
However, $\textrm{Var}(\widehat{\theta})$ cannot be obtained 
by using the delta method on $\textrm{Var}(\widehat{\xi})$, 
as $\textrm{Var}(\widehat{v}-v)$ cannot be obtained 
from $\textrm{Var}(\widehat{u}-u)$ in missing data problem.
In consequence, when $\widehat{\theta}-\theta=O_{p}(1)$, 
validity of statistical inference depends upon the parameterization of fixed parameters.
Thus, the scale is important for using the Hessian matrix in small samples.

\subsection{Reparameterization of random effects}
Here, we consider another scale of the random parameter 
for constructing the h-likelihood. 
Though the $u$-scale is not weak canonical, 
it is Bartlizable when $m\geq 2$ from Lemma 3.2. 
Thus, the h-likelihood with the $u$-scale can be defined as 
\begin{align*}
h(\xi,u)& =\ell_{e}(\xi,u)+c(\xi ;y) \\
&=-\log \xi -e^{v}(\xi^{-1}+m\bar{y})
-\log (\xi^{-1}+m\bar{y})+m(v+1-\log m)+\log \Gamma(m+1),
\end{align*}
which gives the MLE $\widehat{\xi}=\bar{y}$ 
and the MHLE $\widehat{v}'=-\log \bar{y}+\log \{m/(m+1)\}$. 
Here the observed h-information leads to 
\begin{equation*}
\widehat{\textrm{Var}}
\begin{bmatrix}
\xi -\widehat{\xi} \\ 
v-\widehat{v}'
\end{bmatrix}
=I_{2}(\widehat{\xi},\widehat{v}')^{-1}=\frac{1}{m}
\begin{bmatrix}
(m+1)\bar{y}^{-2} & \bar{y}^{-1} \\ 
\bar{y}^{-1} & \frac{m+2}{m+1}
\end{bmatrix}
=\widehat{\textrm{Var}}
\begin{bmatrix}
\xi -\widehat{\xi} \\ 
v-\widehat{v}
\end{bmatrix}
\cdot \left[ 1+O\left( \frac{1}{m}\right) \right]. 
\end{equation*}
Thus, both $h(\xi,v)$ and $h(\xi,u)$ yield asymptotically correct
inferences. To compare the $v$-scale and $u$-scale, we investigate their
predictions for the conditional mean $\mu =\textrm{E}(y_{j}|u)=1/u$. 
For given $\xi$, the predictor $\widetilde{v}=\log (m+1)-\log (\xi^{-1}+m\bar{y})$ leads to 
\begin{equation*}
\widetilde{\mu}=\widetilde{u}^{-1}=\frac{\xi^{-1}+m\bar{y}}{m+1}, 
\end{equation*}
whereas the predictor $\widetilde{v}'=\log m-\log (\xi^{-1}+m\bar{y})$ leads to 
\begin{equation*}
\widetilde{\mu}'=\widetilde{u}'^{-1}
=\frac{\xi^{-1}+m\bar{y}}{m}
=\textrm{E}(\mu |y). 
\end{equation*}
Note here that the h-likelihood with $u$-scale yields the BUP of $\mu$ in small samples, 
\begin{equation*}
\textrm{E}(\mu |\mathbf{y})=\frac{\xi^{-1}+m\bar{y}}{m}
=\widetilde{\mu}'=\exp (-\widetilde{v}'), 
\end{equation*}
and 
\begin{equation*}
\textrm{Var}(\mu |\mathbf{y})
=\textrm{Var}(\mu -\widetilde{\mu}'|\mathbf{y})
=\frac{(\xi^{-1}+m\bar{y})^{2}}{m^{2}(m-1)}
=\left[ \frac{-\partial^{2}h(\xi,u)}{\partial \mu^{2}}\right]
_{\mu =\widetilde{\mu}'}^{-1}\left[ 1+O\left( \frac{1}{m}\right) \right].
\end{equation*}
Thus, for given $\xi$, the Hessian matrix provides a consistent estimator of 
$\textrm{Var}(\mu|\mathbf{y})$ as $m\to\infty$.
Furthermore, it can be shown that 
\begin{equation}
\textrm{E}[(\mu -\widetilde{\mu})^{2}-(\mu -\widetilde{\mu}')^{2}|\mathbf{y}]
=\frac{(\xi^{-1}+m\bar{y})^{2}}{m^{2}(m+1)^{2}}>0.
\label{ineq:bayarri}
\end{equation}
Thus, $h(\xi,u)$ would provide a better prediction for $\mu $ than $h(\xi,v)$. 

\subsection{When $m=n=1$}
As described in Section 5.3 of the main paper,
when $N=m=n=1$, Example 3 becomes 
\citeauthor{bayarri88}'s \citeyearpar{bayarri88} example.
Here, $u$-scale is no longer Bartlizable, as shown by 
\begin{equation*}
\textrm{E}\left[ \textrm{E}\left[ \frac{\partial^{2}h(\theta,u)}{\partial u^{2}}
+\left( \frac{\partial h(\theta,u)}{\partial u}\right)^{2}\Bigg|y\right] \right]
=\textrm{E}\left[ (\theta^{-1}+y)^{2}-2\right] \neq 0,
\end{equation*}
but $v$-scale is still Bartlizable 
to give $\widehat{\xi}=1/y$ and $\widehat{v}=-\log y$. 
However, we have $\textrm{E}(\widehat{\theta})=\infty$, 
$\textrm{Var}(\widehat{\theta})=\infty$,
$\textrm{E}(\widehat{\xi})=\infty$, 
$\textrm{Var}(\widehat{\xi})=\infty$,
$\textrm{Var}(\widehat{u}-\widetilde{u})=\infty$,
and the asymptotic normality cannot be applied to MHLEs 
of both fixed and random parameters,
which makes it difficult to provide meaningful inferences.

\section{Confidence as an extended likelihood}
\label{app:confidence}

There has been a debate about whether confidence is a likelihood. 
Confidence and p-value are not likelihoods 
because they depend on the experiment scheme \citep{lavine22}. 
\citet{pawitan21, pawitan22} showed that 
the confidence of CIs for fixed unknowns is an extended likelihood.
In this section, we show that
the confidence of PIs for random unknowns is also an extended likelihood.

First, consider the confidence for $\boldsymbol{\theta}$. 
Let $\textrm{CI}(\textit{\textbf{Y}})$ be a CI procedure 
for the fixed parameter of interest $\xi=\xi(\boldsymbol{\theta})$, 
which could be a subset of $\boldsymbol{\theta}$ or a transformation of the parameters in $\boldsymbol{\theta}$. 
\citet{bjornstad96} established the EPL for the parameter of interest, 
which is an aim of statistical inference. 
Now, let us define a binary unobserved random variable,
$$
W = W(\xi, \textit{\textbf{Y}}) = I(\xi(\boldsymbol{\theta}) \in \textrm{CI}(\textit{\textbf{Y}})),
$$
which an indicator function whether the CI contains $\xi$ or not.
For $w=W(\xi, \mathbf{y})$ with an observed $\mathbf{y}$,
\citet{lavine22} showed that
\begin{align*}
\begin{cases}
L_{e}(\boldsymbol{\theta}, w=0; \mathbf{y}) = L(\boldsymbol{\theta}; \mathbf{y}) \{ 1 - I(\xi \in \textrm{CI}(\mathbf{y})) \} \\
L_{e}(\boldsymbol{\theta}, w=1; \mathbf{y}) = L(\boldsymbol{\theta}; \mathbf{y}) I(\xi \in \textrm{CI}(\mathbf{y}))
\end{cases}
\end{align*}
is an extended likelihood of \citet{bjornstad96} 
that exploits the full data evidence for $(\boldsymbol{\theta}, w)$.
However, using $L_{e}(\boldsymbol{\theta}, w=1; \mathbf{y})$ to construct a CI procedure would not be feasible, 
as it cannot be determined whether the observed interval $\textrm{CI}(\mathbf{y})$ 
contains the unknown true value $\xi_{0}$. 
Notably, the extended likelihood can be factorized as follows:
for $i = 0, 1$,
\begin{align*}
L_{e}(\boldsymbol{\theta}, w=i; \mathbf{y}) 
= L_{e}(w=i; \mathbf{y}) f_{\boldsymbol{\theta}}(\mathbf{y}|w=1)
= L_{e}(w=i; \mathbf{y}) \cdot \frac{f_{\boldsymbol{\theta}}(\mathbf{y}) I(w=i)}{P_{\boldsymbol{\theta}}(W=i)}
\end{align*}
\citet{pawitan21} showed that
$$
L_{e}(w=1; \mathbf{y}) = P_{\boldsymbol{\theta}}(\xi \in \textrm{CI}(\textit{\textbf{Y}})),
$$
where the LHS is a marginal extended likelihood of $w$
and the RHS is the coverage probability of the corresponding CI procedure
$\textrm{CI}(\textit{\textbf{Y}})$.
The confidence feature
$C(\xi_0 \in \textrm{CI}(\mathbf{y})) = P_{\boldsymbol{\theta}_0}(\xi_0 \in \textrm{CI}(\textit{\textbf{Y}}))$
implies that the confidence of CIs for fixed unknown can be interpreted as an extended likelihood,
$$
L_{e}(w_0=1; \mathbf{y}) = C(\xi_0 \in \textrm{CI}(\mathbf{y})) 
= \int_{\textrm{CI}(\mathbf{y})} c(\xi; \mathbf{y}) d\xi,
$$
where $w_{0}=W(\xi_{0};\mathbf{y})$.
However, the extended likelihood $L_{e}(w=i; \mathbf{y})$ may not contain the full data evidence
required to satisfy the ELP \citep{lavine22},
since the remaining term $f_{\boldsymbol{\theta}}(\mathbf{y}) I(w=i)$ would carry partial evidence for $w$.

Similar arguments apply to random parameters. 
Let $\textrm{PI}(\textit{\textbf{Y}})$ be a PI procedure 
for the random unknown of inferential interest, $\xi = \xi(\mathbf{v})$, 
which could be a subset of $\mathbf{v}$ or a transformation of the random parameters in $\mathbf{v}$. 
Then, as before, we can define a binary unobserved random variable,
$$
W = W(\xi, \textit{\textbf{Y}}) = I(\xi(\mathbf{v}) \in \textrm{CI}(\textit{\textbf{Y}})).
$$
For a realized value $w=W(\xi, \mathbf{y})$,
the extended likelihood with full data evidence for $(\boldsymbol{\theta},\mathbf{v},w)$
can be factorized as follows: for $i = 0, 1$,
\begin{align*}
L_{e}(\boldsymbol{\theta}, \mathbf{v}, w=i; \mathbf{y})
= L_{e}(w=i; \mathbf{y}) f_{\boldsymbol{\theta}}(\mathbf{y}, \mathbf{v}|w=i) 
= L_{e}(w=i; \mathbf{y}) \cdot \frac{f_{\boldsymbol{\theta}}(\mathbf{y},\mathbf{v}) I(w=i)} {P_{\boldsymbol{\theta}}(W=i)}.
\end{align*}
This leads to
$$
L_{e}(w=1) = P(\xi \in \textrm{PI}(\textit{\textbf{Y}})),
$$
where the LHS is a marginal extended likelihood of $w$
and the RHS is the coverage probability of the corresponding PI procedure.
Thus, the confidence feature for random unknowns implies that 
$$
L_{e}(w_0=1; \mathbf{y}) = C(\xi_0 \in \textrm{PI}(\mathbf{y})) 
= \int_{\textrm{PI}(\mathbf{y})} c(\xi; \mathbf{y}) d\xi,
$$
indicating that the confidence of PIs for random unknowns
can also be interpreted as an extended likelihood.
Using two illustrative examples, 
Section 6 of the main paper demonstrates the existence of such $c(\xi; \mathbf{y})$. 
Section 7 further introduces approximate h-confidences, 
which provide asymptotically correct PIs for random unknowns in multi-parameter cases.

There remain interesting issues that call for deeper examination. 
For instance, the parameter of interest could be $\xi(\boldsymbol{\theta}, \mathbf{v})$, 
which depends on both the fixed parameters $\boldsymbol{\theta}$ and the random parameters $\mathbf{v}$. 
However, integrating the CD or PD for dimension reduction is generally discouraged, 
as marginalization yields inappropriate confidence levels in multi-parameter cases 
\citep{bartlett36, bartlett39, bartlett65, pedersen78}, necessitating further studies.

\section{Improving the ELA for h-likelihood}
\label{app:ela}

In this section, we study the choice of the conjugate distribution $q(\mathbf{v})$, 
which improves the computational efficiency of the approximation. 
In Poisson-gamma HGLM, the use of the conjugate gamma distribution, 
$q(\mathbf{v})=f_{\boldsymbol{\theta}}(\mathbf{v}|\mathbf{y})$, 
results in the h-likelihood 
$h_{B}(\boldsymbol{\theta},\mathbf{v})=h(\boldsymbol{\theta},\mathbf{v})$ 
even when $B=1$. 
In contrast, the use of normal distribution requires large $B$ 
to obtain the accurate h-likelihood.
Now, as an example, we consider a Poisson-gamma-gamma HGLM: 
$y_{ij}|u_{1i},u_{2ij}\sim \textrm{Poisson}(\mu_{ij})$, 
$$
\mu_{ij}=\textrm{E}(y_{ij}|u_{1i},u_{2ij})
=\exp \left( \mathbf{x}_{ij}^{\intercal}\boldsymbol{\beta}\right) \cdot u_{1i}u_{2ij}, 
$$
where $u_{1i}\sim \textrm{Gamma}(\lambda_{1}^{-1},\lambda_{1}^{-1})$ 
and $u_{2ij}\sim \textrm{Gamma}(\lambda_{2}^{-1},\lambda_{2}^{-1})$. 
This can be regarded as a NB(negative binomial)-gamma HGLM: 
$y_{ij}|u_{1i}\sim \textrm{NB}(\mu_{ij},\lambda_{2})$,
\begin{equation*}
\mu_{ij}=\textrm{E}(y_{ij}|u_{1i})
=\exp \left( \mathbf{x}_{ij}^{\intercal}\boldsymbol{\beta}\right) \cdot u_{1i}, 
\end{equation*}
where $u_{1i}\sim \textrm{Gamma}(\lambda_{1}^{-1},\lambda_{1}^{-1}).$ 
Let $v_{1i}=\log u_{1i}$ and $v_{2ij}=\log u_{2ij}$, 
then Lemma 3.2 implies that 
$\mathbf{v}_{1}$ and $\mathbf{v}_{2}$ are Bartlizable. 
If we have an explicit form of 
$f_{\boldsymbol{\theta}}(\mathbf{v}_{1},\mathbf{v}_{2}|\mathbf{y})$, then 
\begin{equation*}
h(\boldsymbol{\theta},\mathbf{v}_{1},\mathbf{v}_{2})
\equiv \log f_{\boldsymbol{\theta}}(\mathbf{y}|\mathbf{v}_{1},\mathbf{v}_{2})
+\log f_{\boldsymbol{\theta}}(\mathbf{v}_{1},\mathbf{v}_{2})
-\log f_{\boldsymbol{\theta}}(\widetilde{\mathbf{v}}_{1},\widetilde{\mathbf{v}}_{2}|\mathbf{y}), 
\end{equation*}
where $(\widetilde{\mathbf{v}}_{1},\widetilde{\mathbf{v}}_{2})$ are the modes of 
$f_{\boldsymbol{\theta}}(\mathbf{v}_{1},\mathbf{v}_{2}|\mathbf{y})$. 
However, conditional density 
$f_{\boldsymbol{\theta}}(\mathbf{v}_{1},\mathbf{v}_{2}|\mathbf{y})$ 
may not have an explicit form. Instead of using multivariate normal
distributions for Monte Carlo sampling \citep{han22b}, we may use a gamma
distribution. Let 
\begin{equation*}
h_{B}(\boldsymbol{\theta},\mathbf{v}_{1},\mathbf{v}_{2})
=\ell_{e}(\boldsymbol{\theta},\mathbf{v}_{1},\mathbf{v}_{2})
-\ell_{e}(\boldsymbol{\theta},\widetilde{\mathbf{v}}_{1},\widetilde{\mathbf{v}}_{2})
+\log \left[ \frac{1}{B}\sum_{b=1}^{B}
\frac{L_{e}(\boldsymbol{\theta},\mathbf{v}_{1}^{(b)},\mathbf{v}_{2}^{(b)})}
{q_{1}(\mathbf{v}_{1}^{(b)})q_{2}(\mathbf{v}_{2}^{(b)})}\right] . 
\end{equation*}
Note that both $\mathbf{v}_{1}|\mathbf{y},\mathbf{v}_{2}$ 
and $\mathbf{v}_{2}|\mathbf{y},\mathbf{v}_{1}$ follow explicit log-gamma distributions,
while $\mathbf{v}_{1},\mathbf{v}_{2}|\mathbf{y}$ may not have an explicit form. 
Thus, we may use gamma distributions 
$q_{1}\left( \mathbf{v}_{1}\right) 
=f_{\widehat{\boldsymbol{\theta}}}(\mathbf{v}_{1}|\mathbf{y},\widehat{\mathbf{v}}_{2})$ 
and $q_{2}\left( \mathbf{v}_{2}\right) 
=f_{\widehat{\boldsymbol{\theta}}}(\mathbf{v}_{2}|\mathbf{y},\widehat{\mathbf{v}}_{1})$ 
to approximate $f_{\widehat{\boldsymbol{\theta}}}(\mathbf{v}_{1},\mathbf{v}_{2}|\mathbf{y})$ 
by $q_{1}(\mathbf{v}_{1})q_{2}(\mathbf{v}_{2})$ such that 
\begin{align*}
q_{1}\left( \mathbf{v}_{1}\right) 
& =\prod_{i}\frac{\exp \left[ \left(
y_{i+}+\widehat{\lambda}_{1}^{-1}\right) 
\left\{ v_{1i}+\log \left( 
\widehat{\mu}_{i+}^{(1)}+\widehat{\lambda}_{1}^{-1}\right) \right\}
-\left( \widehat{\mu}_{i+}^{(1)}+\widehat{\lambda}_{1}^{-1}\right) \exp
(v_{1i})\right] }{\Gamma \left( y_{i+}+\widehat{\lambda}_{1}^{-1}\right) }
\\
q_{2}\left( \mathbf{v}_{2}\right) & =\prod_{i}\frac{\exp \left[ \left(
y_{ij}+\widehat{\lambda}_{2}^{-1}\right) \left\{ v_{2ij}+\log \left( 
\widehat{\mu}_{ij}^{(2)}+\widehat{\lambda}_{2}^{-1}\right) \right\}
-\left( \widehat{\mu}_{ij}^{(2)}+\widehat{\lambda}_{2}^{-1}\right) \exp
(v_{2ij})\right] }{\Gamma \left( y_{ij}+\widehat{\lambda}_{2}^{-1}\right) }
\end{align*}
where $\widehat{\mu}_{i+}^{(1)}
=\sum_{j}\exp (\mathbf{x}_{ij}^{\intercal}\widehat{\boldsymbol{\beta}}
+\widehat{v}_{2ij})$ and $\widehat{\mu}_{ij}^{(2)}
=\exp (\mathbf{x}_{ij}^{\intercal}\widehat{\boldsymbol{\beta}}+\widehat{v}_{1i})$. 
From this Poisson-gamma-gamma sampling, the convergence of approximate MHLEs 
is expected to be more rapid than that from Poisson-normal-normal sampling of \citet{han22b}, 
because $q_{1}(\mathbf{v}_{1})q_{2}(\mathbf{v}_{2})$ above 
would be closer to the true predictive likelihood than the multivariate normal distribution. 
However, the Poisson-gamma-gamma sampling still requires sampling of $n(m+1)$ random effects. 
Using the NB-gamma HGLM, we can simplify the sampling process, 
\begin{equation*}
\frac{1}{B}\sum_{b=1}^{B}
\frac{f_{\boldsymbol{\theta}}(\mathbf{y}|\mathbf{v}_{1}^{(b)})
f_{\boldsymbol{\theta}}(\mathbf{v}_{1}^{(b)})}
{q_{1}(\mathbf{v}_{1}^{(b)})}, 
\end{equation*}
where $f_{\boldsymbol{\theta}}(\mathbf{y}|\mathbf{v}_{1}^{(b)})$ 
is the probability function of the NB$(\mu_{ij},\lambda_{2})$ distribution. 
This method only requires the sampling of $n$ number of random effects.

\subsection*{Example S1 (Epilepsy data)}
Epilepsy data \citep{thall90}
is from a clinical trial involving $n=59$ patients with epilepsy. The data
contain $N=236$ observations of seizure counts with $m=4$ repeated measures
from each patient.
We consider a Poisson-gamma HGLM: 
for $i=1,...,n$ and $j=1,...,m$, 
\begin{equation*}
\log \mu_{ij}=\mathbf{x}_{ij}^{\intercal}\boldsymbol{\beta}+\log u_{i}, 
\end{equation*}
where $y_{ij}|u_{i}\sim \textrm{Poisson}(\mu_{ij})$ 
and $u_{i}\sim \textrm{Gamma}(\lambda^{-1},\lambda^{-1})$. 
We construct a simultaneous 95\% PI procedure for all realized values $u_{0i}$, 
based on C3 using the bootstrap method with $B=10^{6}$. 
Figure \ref{fig:pi} shows the
resulting PIs in increasing order. The effects of patients 10, 25, 35, 49,
56 are significantly large ($u_{0i}>1$) and those of patients 15, 16, 17,
38, 41, 52, 57, 58 are significantly small ($u_{0i}<1$), indicating the
presence of heterogeneity among the patients. 
\citet{lee16} identified the same patients (10, 25, 35, 49, 56) to have
significantly large effects ($u_{0i}>1$), which were identified as outliers
in previous studies \citep{thall90, breslow93, ma07}. 
For a certain realized value,
the coverage probability of the PI is very close to the confidence level even for $m=1$.
However, the coverage probabilities of 
the simultaneous PI procedure for all realized values 
can be liberal at extreme values of $\widehat{u}_{0i}$ in both ends,
because $m=4$ is small.

\begin{figure}[tbp]
\centering
\includegraphics[width=\linewidth]{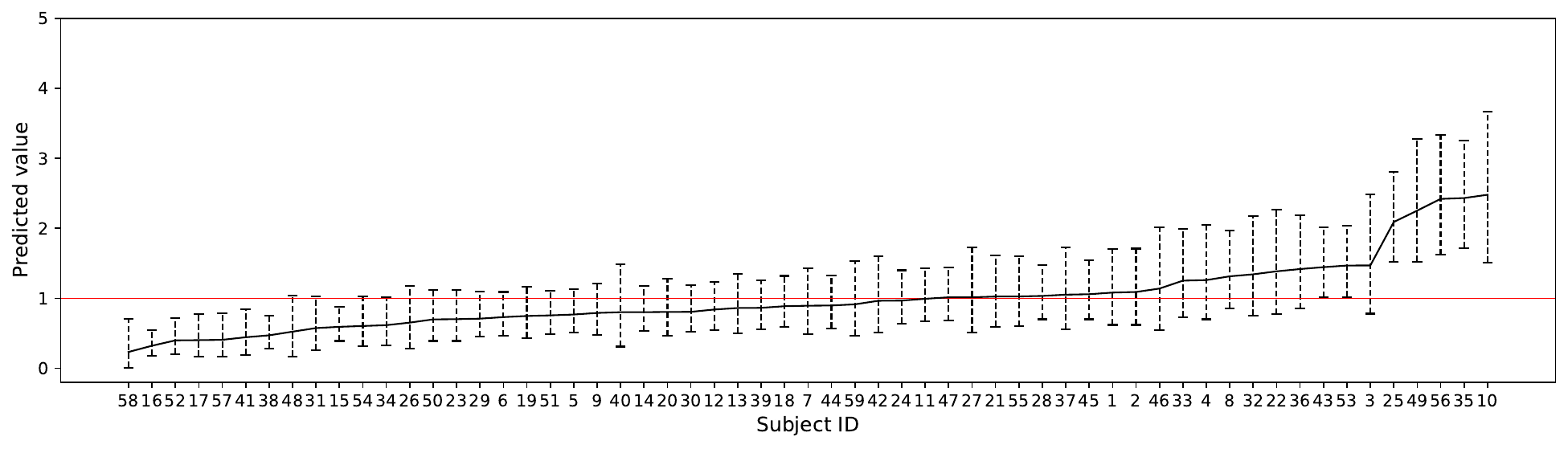}
\caption{95\% PIs of individual random effects of 59
patients, sorted by the predicted values.}
\label{fig:pi}
\end{figure}

We further consider the Poisson-gamma-gamma HGLM 
for the analysis of epilepsy data in Example 5.
We compare the convergence rate of 
Poisson-normal-normal, Poisson-gamma-gamma and NB-gamma Monte Carlo
samplings, using the epilepsy data. Figure \ref{fig:epilepsy} illustrates
the absolute errors between the exact MLEs and approximate MLEs. Exact MLEs
are obtained from NB-gamma sampling with $B=10^{5}$. 
To obtain the box-plots, we perform 100 repetitions for each sampling procedure.
Since the distribution of random effects is not normal,
the LA gives biased estimates for fixed parameters. As $B$ grows, all the
other approximate MLEs converge to the exact MLEs but in different rates.  
MLEs from Poisson-normal-normal sampling exhibit the slowest
convergence, requiring the largest $B$ for accurate approximations.
Advantage of gamma samplings is obvious, especially in estimation of
variance components. MLEs from NB-gamma sampling converge most rapidly and
consistently outperform the LA, even when $B=1$. 
After we obtain MLEs for $\boldsymbol{\theta}$, we compute MHLEs for $\mathbf{v}$.
It is an interesting future research 
to develop more efficient computational method 
when the h-likelihood is not explicit.

\begin{figure}[tbp]
\centering
\includegraphics[width=0.9\linewidth]{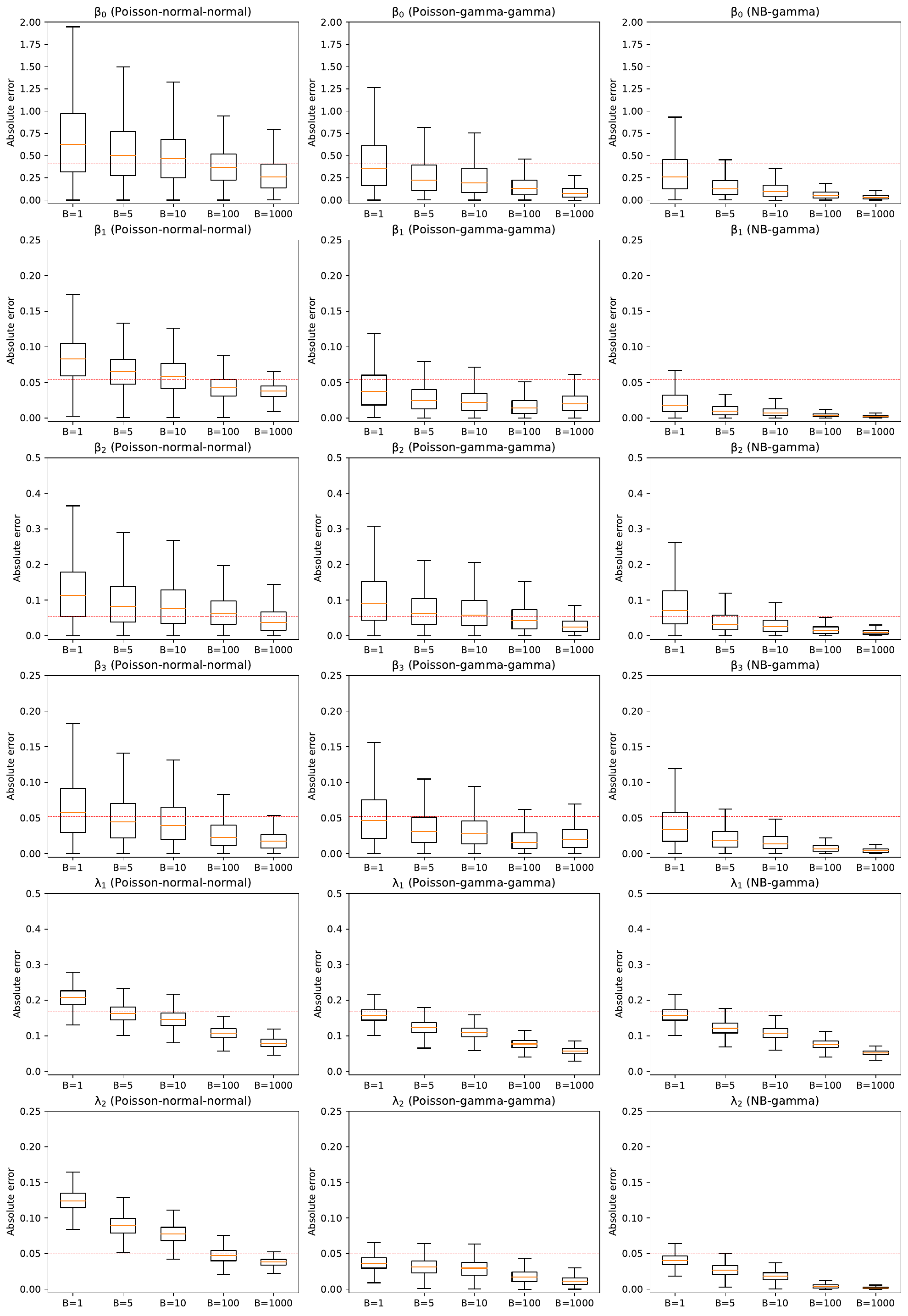}
\caption{Dotted line indicates the absolute error between true MLE and
approximate MLE from Laplace approximation. Each box-plot illustrates the
absolute error between true MLE and approximate MLE from Monte Carlo method.
Poisson-normal-normal sampling, Poisson-gamma-gamma sampling and NB-gamma
sampling are compared.}
\label{fig:epilepsy}
\end{figure}

\section{Proofs}
\label{app:proof}

We assume the following regularity conditions for 
the h-likelihood $h(\boldsymbol{\theta},\mathbf{v};\mathbf{y})$
and all the density functions of $\mathbf{y}$ and $\mathbf{v}$,
$f_{\boldsymbol{\theta}}(\mathbf{y})$, 
$f_{\boldsymbol{\theta}}(\mathbf{v})$, 
$f_{\boldsymbol{\theta}}(\mathbf{y}|\mathbf{v})$, 
and $f_{\boldsymbol{\theta}}(\mathbf{v}|\mathbf{y})$.
\begin{enumerate}
\item[] (R1) There exists an open set $\boldsymbol{\Theta}^{\ast}\subseteq \boldsymbol{\Theta}$ of which the true value is an interior point.
\item[] (R2) Supports are invariant to parameter $\boldsymbol{\theta}$.
\item[] (R3) Continuously differentiable up to third order.
\item[] (R4) Integration and differentiation can be interchanged up to second order.
\item[] (R5) The third derivatives are bounded by some integrable functions.
\item[] (R6) The information matrices are positive definite.
\end{enumerate}
In the absence of random unknowns $\mathbf{v}$,
R1-R6 are required only for the likelihood 
$L(\boldsymbol{\theta};\mathbf{y})=f_{\boldsymbol{\theta}}(\mathbf{y})$,
as commonly required by asymptotic optimality of MLEs for fixed unknowns $\boldsymbol{\theta}$.

\subsection{Proof of Lemma 3.2}
\label{proof:bartlett}

(a) Consider 
$v(u) = \textrm{logit} F_{\boldsymbol{\theta}}(u) = \log F_{\boldsymbol{\theta}}(u) - \log (1-F_{\boldsymbol{\theta}}(u))$ 
where $F_{\boldsymbol{\theta}}(u)$ denotes the cumulative distribution function of $u$, 
then $v(\cdot)$ is an increasing function from $\Omega_u$ to $(-\infty,\infty)$. 
The density function of $v$ is given by 
\begin{align*}
f_{\boldsymbol{\theta}}(v) 
&= f_{\boldsymbol{\theta}}(u) \left| \frac{du}{dv}\right| 
= f_{\boldsymbol{\theta}}(u) \left|\frac{f_{\boldsymbol{\theta}}(u)}{F_{\boldsymbol{\theta}}(u)}
+ \frac{f_{\boldsymbol{\theta}}(u)}{1-F_{\boldsymbol{\theta}}(u)}\right|^{-1} 
= F_{\boldsymbol{\theta}}(u)(1-F_{\boldsymbol{\theta}}(u))
\end{align*}
and its derivative is 
\begin{align*}
f^{\prime}_{\boldsymbol{\theta}}(v) 
&= \frac{du}{dv} \frac{d}{du} \left\{ F_{\boldsymbol{\theta}}(u)(1-F_{\boldsymbol{\theta}}(u)) \right\} 
= F_{\boldsymbol{\theta}}(u) (1-F_{\boldsymbol{\theta}}(u)) (1-2F_{\boldsymbol{\theta}}(u)).
\end{align*}
Since $F_{\boldsymbol{\theta}}(u)$ is cumulative distribution function that
increases from 0 to 1 and $v(u)$ is an increasing function of $u$, 
we have $f_{\boldsymbol{\theta}}(v)=f^{\prime}_{\boldsymbol{\theta}}(v)=0$ at the boundary $v = \pm \infty$. 
This satisfies a sufficient condition for the first and second Bartlett identities \citep{meng09}, 
hence we can always find at least one Bartlizable transformation.

\noindent (b) It is enough to investigate the derivatives with respect to $\mathbf{v}$. 
Since $f_{\boldsymbol{\theta}}(\mathbf{v} |\mathbf{y}) = 0$ at the boundary, we have 
$$
\textrm{E} \left\{ \nabla_{\mathbf{v}} \log f_{\boldsymbol{\theta}}(\mathbf{v} |\mathbf{y})
| \mathbf{y}  \right\} 
= \int_{\Omega_{\mathbf{v}}} \nabla_{\mathbf{v}} f_{\boldsymbol{\theta}}(\mathbf{v} |\mathbf{y}) d\mathbf{v}  
= \mathbf{0}, 
$$
which leads to the first Bartlett identity 
$$
\textrm{E} \left\{ \nabla_{\mathbf{v}} \ell_e(\boldsymbol{\theta}, \mathbf{v}) \right\}
= \textrm{E} \left\{ \nabla_{\mathbf{v}} \log f_{\boldsymbol{\theta}}(\mathbf{v} |\mathbf{y}) \right\}
= \textrm{E} \left[ \textrm{E} \left\{ 
\nabla_{\mathbf{v}} \log f_{\boldsymbol{\theta}}(\mathbf{v} |\mathbf{y}) | \mathbf{y}  \right\} \right] 
= \mathbf{0}. 
$$
Using the similar argument, we have 
$$
\textrm{E} \left\{ 
\frac{\partial^2 \log f_{\boldsymbol{\theta}}(\mathbf{v}|\mathbf{y})}{\partial v_i \partial v_j} \Bigg| \mathbf{y}  
\right\} 
= \int_{\Omega_{\mathbf{v}}} \frac{\partial}{\partial v_i} 
\left( \frac{\partial f_{\boldsymbol{\theta}}(\mathbf{v} |\mathbf{y})}{\partial v_j} \right) d\mathbf{v}  = 0, 
$$
for any $i, j = 1, ..., q$. This leads to the second Bartlett identity.

\noindent (c) Now suppose that the probability density function 
$f_{\boldsymbol{\theta}}(v)$ is differentiable with respect to $u$ for any $u$ 
in the support $\Omega_v = (-\infty, \infty)$. 
Since $f_{\boldsymbol{\theta}}(v)>0$ and $\int_{-\infty}^\infty f(v) dv = 1$, 
we have $f(v)=0$ at the boundary $v=\pm \infty$. 
Then we also have $f^{\prime}(v)=0$ at the boundary $v=\pm \infty$ 
since $\int_{-\infty}^t f^{\prime}(v) dv < \infty$ for any $t\in \mathbb{R}$. 
Thus, $v$ achieves the sufficient condition for the Bartlett identities. 
\hfill \qed

\subsection{Proof of Theorem 4.1}
\label{proof:hcrb}

In Section 4 of the main paper, 
we suppose that $\widehat{\boldsymbol{\zeta}}(\mathbf{y})$ is an unbiased estimator of 
$\boldsymbol{\zeta}(\boldsymbol{\theta}, \mathbf{v})$ such that 
\begin{equation}  \label{eq:unbiased1}
\textrm{E} \left[ \widehat{\boldsymbol{\zeta}}(\mathbf{y}) 
- \boldsymbol{\zeta}(\boldsymbol{\theta}, \mathbf{v}) \right] = 0
\end{equation}
with 
\begin{equation}  \label{eq:unbiased2}
\lim_{\mathbf{v}  \to \partial \Omega_{\mathbf{v}}} \int_{\Omega_{\mathbf{y}}} 
\left[ \widehat{\boldsymbol{\zeta}}(\mathbf{y}) 
- \boldsymbol{\zeta}(\boldsymbol{\theta}, \mathbf{v}) \right] f_{\boldsymbol{\theta}}(\mathbf{y}, \mathbf{v}) d\mathbf{y}  = 0.
\end{equation}
We further assume an additional condition 
either $a(\boldsymbol{\theta}, \mathbf{y})$ does not depend on $\mathbf{y}$ 
or its second derivative
$\nabla^2_{\boldsymbol{\theta}} a(\boldsymbol{\theta}, \mathbf{y})$
is positive semi-definite.
Then, we have
\begin{align*}
& \nabla_{\boldsymbol{\theta},\mathbf{v}} 
\int_{\Omega_{\mathbf{y}}} \left\{ \widehat{\boldsymbol{\zeta}}(\mathbf{y}) 
-\boldsymbol{\zeta}(\boldsymbol{\theta}, \mathbf{v}) \right\} f_{\boldsymbol{\theta}}(\mathbf{y}, \mathbf{v}) d\mathbf{y} 
\\
&= - \int_{\Omega_{\mathbf{y}}} 
\{ \nabla_{\boldsymbol{\theta},\mathbf{v}} \boldsymbol{\zeta}(\boldsymbol{\theta}, \mathbf{v}) \} 
f_{\boldsymbol{\theta}}(\mathbf{y}, \mathbf{v}) d\mathbf{y}  
+ \int_{\Omega_{\mathbf{y}}} \{\nabla_{\boldsymbol{\theta},\mathbf{v}} f_{\boldsymbol{\theta}}(\mathbf{y}, \mathbf{v})\}
\left\{ \widehat{\boldsymbol{\zeta}}(\mathbf{y}) - \boldsymbol{\zeta}(\boldsymbol{\theta}, \mathbf{v}) \right\} d\mathbf{y} 
\end{align*}
Integrating with respect to $\mathbf{v}$ leads to 
\begin{align*}
\textrm{LHS} = \int_{\Omega_{\mathbf{v}}} 
\left[ \nabla_{\boldsymbol{\theta},\mathbf{v}}
\int_{\Omega_{\mathbf{y}}} \left\{ \widehat{\boldsymbol{\zeta}}(\mathbf{y}) 
- \boldsymbol{\zeta}(\boldsymbol{\theta}, \mathbf{v}) \right\} f_{\boldsymbol{\theta}}(\mathbf{y},\mathbf{v}) d\mathbf{y} \right] d\mathbf{v}  = 0
\end{align*}
from the assumption \eqref{eq:unbiased1} and \eqref{eq:unbiased2}. 
Thus, we have 
\begin{align*}
\textrm{E} \left\{ \nabla \boldsymbol{\zeta}(\boldsymbol{\theta},\mathbf{v}) \right\}
&= \iint \{ \nabla \boldsymbol{\zeta}(\boldsymbol{\theta},\mathbf{v}) \}
f_{\boldsymbol{\theta}}(\mathbf{y},\mathbf{v}) d\mathbf{y}  d\mathbf{v}  
\\
&= \iint \nabla f_{\boldsymbol{\theta}}(\mathbf{y}, \mathbf{v})
\left\{ \widehat{\boldsymbol{\zeta}}(\mathbf{y}) - \boldsymbol{\zeta}(\boldsymbol{\theta}, \mathbf{v}) \right\} d\mathbf{y}  d\mathbf{v}  
\\
&= \iint \nabla \ell_e(\boldsymbol{\theta}, \mathbf{v})
\left\{ \widehat{\boldsymbol{\zeta}}(\mathbf{y}) 
- \boldsymbol{\zeta}(\boldsymbol{\theta}, \mathbf{v}) \right\} 
f_{\boldsymbol{\theta}}(\mathbf{y}, \mathbf{v}) d\mathbf{y}  d\mathbf{v}.
\end{align*}
Note here that the variance and covariance are given by
\begin{align*}
\textrm{Var} \left\{ \nabla \ell_e(\boldsymbol{\theta}, \mathbf{v}) \right\}
&= \textrm{E} \left[ \left\{ \nabla \ell_e(\boldsymbol{\theta}, \mathbf{v}) \right\}^\intercal 
\left\{ \nabla \ell_e(\boldsymbol{\theta}, \mathbf{v}) \right\} \right]
= \textrm{E} \left\{ - \nabla^2 \ell_e(\boldsymbol{\theta}, \mathbf{v}) \right\}
\geq \textrm{E} \left\{ - \nabla^2 h(\boldsymbol{\theta}, \mathbf{v}) \right\}
\end{align*}
and
\begin{align*}
\textrm{Cov} \left\{ \nabla \ell_e(\boldsymbol{\theta}, \mathbf{v}), 
\widehat{\boldsymbol{\zeta}}(\mathbf{y}) - \boldsymbol{\zeta}(\boldsymbol{\theta}, \mathbf{v}) \right\}
&= \textrm{E} \left[ \left\{ \nabla \ell_e(\boldsymbol{\theta}, \mathbf{v}) \right\}
\left\{ \widehat{\boldsymbol{\zeta}}(\mathbf{y}) 
- \boldsymbol{\zeta}(\boldsymbol{\theta}, \mathbf{v}) \right\} \right] 
\\
&= \iint \left\{ \nabla \ell_e(\boldsymbol{\theta}, \mathbf{v}) \right\}
\left\{ \widehat{\boldsymbol{\zeta}}(\mathbf{y}) 
- \boldsymbol{\zeta}(\boldsymbol{\theta}, \mathbf{v}) \right\}
f_{\boldsymbol{\theta}}(\mathbf{y}, \mathbf{v}) d\mathbf{y}  d\mathbf{v}  
\\
&=\textrm{E} \left\{ \nabla \zeta(\boldsymbol{\theta}, \mathbf{v}) \right\},
\end{align*}
respectively.
By the multivariate Cauchy-Schwartz inequality, we have 
\begin{align*}
\textrm{Var} \left\{ \widehat{\boldsymbol{\zeta}}(\mathbf{y}) 
- \boldsymbol{\zeta}(\boldsymbol{\theta}, \mathbf{v}) \right\}
&\geq \textrm{E}\left\{ \nabla \zeta(\boldsymbol{\theta}, \mathbf{v}) \right\}
\textrm{E}\left\{ -\nabla^2 h(\boldsymbol{\theta}, \mathbf{v}) \right\}^{-1} 
\textrm{E} \left\{ \nabla \zeta(\boldsymbol{\theta}, \mathbf{v}) \right\}^\intercal 
= \mathcal{Z}_{\boldsymbol{\theta}} \mathcal{I}_{\boldsymbol{\theta}}^{-1} \mathcal{Z}_{\boldsymbol{\theta}}^\intercal.    
\end{align*}
\hfill \qed

\subsection{Proof of Theorem 4.2}
\label{proof:h-asymptotic}

We denote the inverse matrices of observed and expected h-information by
$$
I(\boldsymbol{\theta},\mathbf{v})^{-1} 
= \begin{bmatrix} 
I_{\boldsymbol{\theta}\boldsymbol{\theta}} & I_{\boldsymbol{\theta}\mathbf{v}} 
\\ I_{\mathbf{v}\boldsymbol{\theta}} & I_{\mathbf{v}\mathbf{v}} 
\end{bmatrix}^{-1}
= \begin{bmatrix} 
I^{\boldsymbol{\theta}\boldsymbol{\theta}} & I^{\boldsymbol{\theta}\mathbf{v}} 
\\ I^{\mathbf{v}\boldsymbol{\theta}} & I^{\mathbf{v}\mathbf{v}} 
\end{bmatrix}
\quad \textrm{and} \quad
\mathcal{I}_{\boldsymbol{\theta}}^{-1} 
= \begin{bmatrix} 
\mathcal{I}_{\boldsymbol{\theta}\boldsymbol{\theta}} & \mathcal{I}_{\boldsymbol{\theta}\mathbf{v}} 
\\ \mathcal{I}_{\mathbf{v}\boldsymbol{\theta}} & \mathcal{I}_{\mathbf{v}\mathbf{v}} 
\end{bmatrix}^{-1}
= \begin{bmatrix} 
\mathcal{I}^{\boldsymbol{\theta}\boldsymbol{\theta}} & \mathcal{I}^{\boldsymbol{\theta}\mathbf{v}} 
\\ \mathcal{I}^{\mathbf{v}\boldsymbol{\theta}} & \mathcal{I}^{\mathbf{v}\mathbf{v}} 
\end{bmatrix},
$$
respectively.
Asymptotic normality can be proved by the central limit theorem and
the Bartlett identities. Note that the MHLE
$\widehat{\boldsymbol{\theta}} = \argmax_{\boldsymbol{\theta}}h(\boldsymbol{\theta},\mathbf{v})$ 
adapts the asymptotic normality of the MLE, 
$$
(\widehat{\boldsymbol{\theta}} - \boldsymbol{\theta}) \to N \left(\mathbf{0}, \ \mathcal{I}_m(\boldsymbol{\theta})^{-1}\right), 
$$
where
$\mathcal{I}_m(\boldsymbol{\theta}) 
= \textrm{E} \left\{ - \nabla^2 \ell(\boldsymbol{\theta}; \mathbf{y}) \right\}$ 
is the expected Fisher information. 
Here, the variance of $\widehat{\boldsymbol{\theta}}$
can be estimated by using the observed Fisher information, 
$$
\widehat{\textrm{Var}} (\widehat{\boldsymbol{\theta}}) 
= \left[ -\nabla^2 \ell(\boldsymbol{\theta}; \mathbf{y}) \right]^{-1}_{\boldsymbol{\theta}=\widehat{\boldsymbol{\theta}}}
= \left[ -\nabla^2 h(\boldsymbol{\theta}, \widetilde{\mathbf{v}}(\boldsymbol{\theta})) \right]^{-1}_{\boldsymbol{\theta}=\widehat{\boldsymbol{\theta}}}. 
$$
Since $\widetilde{\mathbf{v}}$ is the solution of $\nabla_{\mathbf{v}} h(\boldsymbol{\theta}, \mathbf{v}) = 0$,
the chain rule leads to 
\begin{align*}
\nabla \ell(\boldsymbol{\theta};\mathbf{y})
&= \nabla_{\boldsymbol{\theta}} h(\boldsymbol{\theta}, \widetilde{\mathbf{v}})
= \left[ \nabla_{\boldsymbol{\theta}} h(\boldsymbol{\theta}, \mathbf{v}) \right]_{\mathbf{v}=\widetilde{\mathbf{v}}}
+ \left\{ \nabla_{\boldsymbol{\theta}}\widetilde{\mathbf{v}}(\boldsymbol{\theta}) \right\}
\left[ \nabla_{\mathbf{v}} h(\boldsymbol{\theta},\mathbf{v})\right]_{\mathbf{v} =\widetilde{\mathbf{v}}}
= \left[ \nabla_{\boldsymbol{\theta}} h(\boldsymbol{\theta}, \mathbf{v}) \right]_{\mathbf{v}=\widetilde{\mathbf{v}}}
\end{align*}
and 
\begin{align*}
\nabla^2 \ell(\boldsymbol{\theta}; \mathbf{y})
= \nabla^2_{\boldsymbol{\theta}} h(\boldsymbol{\theta}, \widetilde{\mathbf{v}})
= \left[ \nabla^2_{\boldsymbol{\theta}} h(\boldsymbol{\theta}, \mathbf{v}) \right]_{\mathbf{v} =\widetilde{\mathbf{v}}}
+ \left\{ \nabla_{\boldsymbol{\theta}}\widetilde{\mathbf{v}}(\boldsymbol{\theta}) \right\}
\left[ \nabla_{\boldsymbol{\theta}} \nabla_{\mathbf{v}}^{\intercal} h(\boldsymbol{\theta}, \mathbf{v}) 
\right]_{\mathbf{v} =\widetilde{\mathbf{v}}}.
\end{align*}
From the fact that 
$$
\nabla_{\boldsymbol{\theta}} 
\left[ \nabla_{\mathbf{v}}^\intercal h(\boldsymbol{\theta},\mathbf{v})\right]_{\mathbf{v}=\widetilde{\mathbf{v}}}
= \left\{ \nabla_{\boldsymbol{\theta}} \widetilde{\mathbf{v}}(\boldsymbol{\theta}) \right\}
\left[ \nabla^2_{\mathbf{v}} h(\boldsymbol{\theta}, \mathbf{v}) \right]_{\mathbf{v} =\widetilde{\mathbf{v}}}
+ \left[ \nabla_{\boldsymbol{\theta}} \nabla_{\mathbf{v}}^{\intercal} h(\boldsymbol{\theta}, \mathbf{v})
\right]_{\mathbf{v} =\widetilde{\mathbf{v}}} = 0, 
$$
we have
\begin{align*}
\nabla^2 \ell(\boldsymbol{\theta}; \mathbf{y})
&= \left[ \nabla^2_{\boldsymbol{\theta}} h(\boldsymbol{\theta}, \mathbf{v}) 
-
\left\{ \nabla_{\boldsymbol{\theta}} \nabla_{\mathbf{v}}^\intercal h(\boldsymbol{\theta}, \mathbf{v})\right\} 
\left\{ \nabla^2_{\mathbf{v}} h(\boldsymbol{\theta}, \mathbf{v})\right\}^{-1} 
\left\{ \nabla_{\mathbf{v}}\nabla_{\boldsymbol{\theta}}^\intercal h(\boldsymbol{\theta}, \mathbf{v})\right\}
\right]_{\mathbf{v} =\widetilde{\mathbf{v}}},
\end{align*}
which implies that 
$\left\{ - \nabla^2 \ell(\boldsymbol{\theta}; \mathbf{y}) \right\}^{-1}
=\widetilde{I}^{\boldsymbol{\theta}\boldsymbol{\theta}}$
where $\widetilde{I}=I(\boldsymbol{\theta}, \widetilde{\mathbf{v}})$.
Thus, the variance of $\widehat{\boldsymbol{\theta}}$ is
$$
\textrm{Var}(\widehat{\boldsymbol{\theta}} - \boldsymbol{\theta})
\to \mathcal{I}^{\boldsymbol{\theta}\boldsymbol{\theta}},
$$
which can be estimated by 
$
\widehat{\textrm{Var}} (\widehat{\boldsymbol{\theta}}-\boldsymbol{\theta}) 
= \widehat{I}^{\boldsymbol{\theta}\boldsymbol{\theta}}.
$
Note here that the convergence of MLE
$\widehat{\boldsymbol{\theta}} \to \boldsymbol{\theta} $
leads to
$\widehat{\mathbf{v}} = \widetilde{\mathbf{v}}(\widehat{\boldsymbol{\theta}}; \mathbf{y}) 
\to \widetilde{\mathbf{v}}(\boldsymbol{\theta};\mathbf{y})$.
Furthermore, $\widetilde{\mathbf{v}}(\boldsymbol{\theta}; \mathbf{y}) \to \textrm{E}(\mathbf{v} |\mathbf{y})$
since $\textrm{Var}(\mathbf{v} |\mathbf{y}) \to 0$. 
Thus, $\widehat{\mathbf{v}}$ converges to the best unbiased predictor, 
$$
\widehat{\mathbf{v}}
\to \textrm{E}_{\boldsymbol{\theta}}(\mathbf{v} |\mathbf{y}) . 
$$
The variance of the MHL predictor $\widehat{\mathbf{v}}$ can be expressed as 
$$
\textrm{Var}(\widehat{\mathbf{v}}-\mathbf{v}) 
= \textrm{Var} \left(\widetilde{\mathbf{v}}(\widehat{\boldsymbol{\theta}}) 
- \widetilde{\mathbf{v}}(\boldsymbol{\theta}) + \widetilde{\mathbf{v}}(\boldsymbol{\theta}) 
-\mathbf{v} \right). 
$$
By the delta method, $\widetilde{\mathbf{v}}(\widehat{\boldsymbol{\theta}})
\approx \widetilde{\mathbf{v}}(\boldsymbol{\theta}) 
+ \left\{ \nabla_{\boldsymbol{\theta}} \widetilde{\mathbf{v}}(\boldsymbol{\theta})\right\}^\intercal
(\widehat{\boldsymbol{\theta}}-\boldsymbol{\theta}) 
$ implies that 
$$
\textrm{Var} \left\{ \widetilde{\mathbf{v}}(\widehat{\boldsymbol{\theta}}) - \widetilde{\mathbf{v}}(\boldsymbol{\theta}) \right\} 
\approx \left\{ \nabla_{\boldsymbol{\theta}} \widetilde{\mathbf{v}}(\boldsymbol{\theta})\right\}^\intercal
\textrm{Var}(\widehat{\boldsymbol{\theta}})
\left\{ \nabla_{\boldsymbol{\theta}} \widetilde{\mathbf{v}}(\boldsymbol{\theta})\right\}
= \widetilde{I}_{\mathbf{v}\mathbf{v}}^{-1} \widetilde{I}_{\mathbf{v} \boldsymbol{\theta}} 
\widetilde{I}^{\boldsymbol{\theta} \boldsymbol{\theta}}
\widetilde{I}_{\boldsymbol{\theta} \mathbf{v}} \widetilde{I}_{\mathbf{v} \mathbf{v}}^{-1}, 
$$
and the Bartlett identities imply that
$
\textrm{Var} \left\{\widetilde{\mathbf{v}}(\boldsymbol{\theta}) - \mathbf{v} \right\}
\approx \mathcal{I}_{\mathbf{v} \mathbf{v}}^{-1}.
$
Thus, the variance of $(\widehat{\mathbf{v}}-\mathbf{v})$ is
$$
\textrm{Var}(\widehat{\mathbf{v}}-\mathbf{v})
\to \mathcal{I}^{\mathbf{v}\mathbf{v}},
$$
which can be estimated by
$
\widehat{\textrm{Var}} (\widehat{\mathbf{v}}-\mathbf{v}) = \widehat{I}^{\mathbf{v} \mathbf{v}}.
$
Similarly, the covariance component becomes 
$$
\textrm{Cov}\left[ (\widehat{\boldsymbol{\theta}}-\boldsymbol{\theta}), (\widehat{\mathbf{v}}-\mathbf{v}) \right] 
\to \mathcal{I}^{\boldsymbol{\theta} \mathbf{v}}
$$
which can be estimated by 
$
\widehat{\textrm{Cov}} [ (\widehat{\boldsymbol{\theta}}-\boldsymbol{\theta}), (\widehat{\mathbf{v}}-\mathbf{v})  ] 
= \widehat{I}^{\boldsymbol{\theta} \mathbf{v}}.
$
Consequently, we have
$$
\textrm{Var} \begin{bmatrix}
\widehat{\boldsymbol{\theta}} - \boldsymbol{\theta} \\
\widehat{\mathbf{v}} - \mathbf{v}
\end{bmatrix}
\to \mathcal{I}^{-1},
$$
which can be estimated by ${I}(\widehat{\boldsymbol{\theta}}, \widehat{\mathbf{v}})^{-1} \to \mathcal{I}_{\boldsymbol{\theta}}^{-1}$.
\hfill \qed
\subsubsection{Theorem 4.2 in LMMs}
In the linear mixed model of Example 1, we have two MHLEs, 
\begin{equation*}
\widetilde{v}_{i}
=\frac{\lambda\sum_{j=1}^{m} 
(y_{ij}-\mathbf{x}_{ij}^{\intercal} \boldsymbol{\beta})}{\sigma^{2}+m\lambda}
\quad \textrm{and}\quad 
\widetilde{v}_{i}'
=\frac{\lambda\sum_{j=1}^{m}(y_{ij}-\mathbf{x}_{ij}^{\intercal} \boldsymbol{\beta})
-\sigma^{2}\lambda}{\sigma^{2}+m\lambda},
\end{equation*}
from the $\mathbf{v}$-scale $h(\boldsymbol{\theta},\mathbf{v})$
and $\mathbf{u}$-scale $h(\boldsymbol{\theta},\mathbf{u})$, respectively. 
Here, the $\mathbf{v}$-scale leads to BUP for the conditional mean $\mu_{ij}$, 
\begin{equation*}
\widetilde{\mu}_{ij}
=\mathbf{x}_{ij}^{\intercal} \boldsymbol{\beta}+\widetilde{v}_{i}
=\textrm{E}(\mathbf{x}_{ij}^{\intercal} \boldsymbol{\beta} + v_{i}|\mathbf{y})
=\textrm{E}(\mu_{ij}|\mathbf{y}),
\end{equation*}
whereas the $\mathbf{u}$-scale leads to asymptotic BUP 
for $\mu_{ij}$ as $m\to\infty$,
\begin{equation*}
\widetilde{\mu}_{ij}'
=\mathbf{x}_{ij}^{\intercal} \boldsymbol{\beta}+\widetilde{v}_{i}'
=\textrm{E}(\mathbf{x}_{ij}^{\intercal} \boldsymbol{\beta}+v_{i}|\mathbf{y})
-\frac{\sigma^{2}\lambda}{\sigma^{2}+m\lambda}
=\textrm{E}(\mu_{ij}|\mathbf{y})+O(1/m). 
\end{equation*}

Now we illustrate Theorem 4.2 in this example.
For simplicity of arguments, consider a one-way LMM
with $\mathbf{x}_{ij}^{\intercal}\boldsymbol{\beta} = \mu_{0}$
and $\sigma^{2}=\lambda=1$, where $n=2$. 
The use of $v$-scale leads to the MLE 
$\widehat{\mu}_{0}=\bar{y}=(\bar{y}_{1}+\bar{y}_{2})/2$ 
and the BUP 
$\widehat{v}_{1}
=m(\bar{y}_{1}-\widehat{\mu}_{0})/(m+1)
=m(\bar{y}_{1}-\bar{y}_{2})/(2m+2)$ 
and $\widehat{v}_{2}
=m(\bar{y}_{2}-\widehat{\mu}_{0})/(m+1)
=m(\bar{y}_{2}-\bar{y}_{1})/(2m+2)$. 
Note that $\textrm{E}(\mu_{0}-\widehat{\mu}_{0})
=\textrm{E}(v_{1}-\widehat{v}_{1})=\textrm{E}(v_{2}-\widehat{v}_{2})=0$. 
Here the variance and covariance are given by 
\begin{align*}
& \textrm{Var}(\mu_{0}-\widehat{\mu}_{0})
=\textrm{Var}(\bar{y})
=\textrm{Var}(\textrm{E}(\bar{y}|v_{1},v_{2}))+\textrm{E}(\textrm{Var}(\bar{y}|v_{1},v_{2}))
=\frac{m+1}{2m}, \\
& \textrm{Var}(v_{i}-\widehat{v}_{i})
=\textrm{Var}(\textrm{E}(v_{i}-\widehat{v}_{i}|\mathbf{y}))
+\textrm{E}(\textrm{Var}(v_{i}-\widehat{v}_{i}|\mathbf{y}))
=\frac{m}{2m+2}+\frac{1}{m+1}=\frac{m+2}{2m+2}, \\
& \textrm{Cov}(\mu_{0}-\widehat{\mu}_{0},v_{i}-\widehat{v}_{i})
=\textrm{E}\left[ (\mu_{0}-\widehat{\mu}_{0})(v_{i}-\widehat{v}_{i})\right] 
=-\textrm{E}\left[ \bar{y}\{\textrm{E}(v_{i}|\mathbf{y})-\widehat{v}_{i}\}\right] 
=-\frac{1}{2}, \\
& \textrm{Cov}(v_{1}-\widehat{v}_{1},v_{2}-\widehat{v}_{2})
=\frac{m\textrm{E}\{(\bar{y}_{1}-\bar{y}_{2})(v_{1}-v_{2})\}}{2m+2}
-\frac{m^{2}\textrm{E}\{(\bar{y}_{1}-\bar{y}_{2})^{2}\}}{(2m+2)^{2}}
=\frac{m}{2m+2},
\end{align*}
for $i=1,2$. Since the expected h-information gives 
\begin{equation*}
\left[ -\textrm{E}\left\{ 
\nabla^{2} h(\mu_{0},v_{1},v_{2})
\right\} \right]
^{-1}= 
\begin{bmatrix}
2m & m & m \\ 
m & m+1 & 0 \\ 
m & 0 & m+1
\end{bmatrix}
^{-1}= 
\begin{bmatrix}
\frac{m+1}{2m} & -\frac{1}{2} & -\frac{1}{2} \\ 
-\frac{1}{2} & \frac{m+2}{2m+2} & \frac{m}{2m+2} \\ 
-\frac{1}{2} & \frac{m}{2m+2} & \frac{m+2}{2m+2}
\end{bmatrix}
, 
\end{equation*}
we can see that Theorem 4.2 holds exactly for this
example, even in small samples.

\subsection{Proof of Theorem 6.1}
\label{proof:future}

Since the h-confidence gives the CD for $\theta$,
which satisfy the confidence feature,
it is enough to show that the PD for $u$ satisfies the theorem.
For any $t \in \Omega$, we have
\begin{align*}
C(q_{\alpha}(t); t)
&= \int_{\Omega} F_{1}(q_{\alpha}(t);\eta)  
\left\{ \nabla_{\eta} F_{n}(t; \eta) \right\} d \eta \\
&= \int_{\Omega} F_{1}(q_{\alpha}(t);\eta)  
\left\{ \nabla_{\eta} F_{n}(\eta;t) \right\} d \eta \\
&= \int_{\Omega} F_{1}(q_{\alpha}(t);\eta)  
f_{n}(\eta;t) d \eta \\
&= 1-\alpha,
\end{align*}
where $f_{n}(t;\eta)$ denotes the density function of $t$ with parameter $\eta$
and $f_{n}(\eta;t)$ is the function obtained by switching $t$ and $\eta$.
Since the support of $T$ and parameter space of $\eta$ is identical,
for any $\eta \in \Omega$, we also have
\begin{align*}
\int_{\Omega} F_{1}(q_{\alpha}(\eta);t)  
f_{n}(t;\eta) d t
= 1-\alpha.
\end{align*}
Thus, for any $\eta \in \Omega$,
\begin{align*}
P_{\theta}(U \leq q_{\alpha}(t)) 
&= \int_{\Omega} \int_{u \leq q_{\alpha}(t)}
f(u;\eta) f_{n}(t;\eta) d u d t
\\
&= \int_{\Omega} F_{1}(q_{\alpha}(t);\eta) f_{n}(t;\eta) d t
= 1-\alpha,
\end{align*}
which implies that $\textrm{PI}_{\alpha}(t)$ satisfies the confidence feature.
\hfill \qed

\subsection{Proof of Theorem 6.2}
\label{proof:realized}

Here, the h-confidence gives the PD for $u$,
$$
c(u;t) 
= \int_{\Theta} c(\theta,u;t) d\theta
= \int_{\Theta} [\nabla_{u} \{1-F(t|u)\}]
[\nabla_{\theta} \{1-F_{\theta}(u)\}] d\theta
= \nabla_{u} P(T\geq t| u),
$$
which gives the confidence feature for realized value,
analogous to the current definition of the CD for fixed parameter.
Since $F(u;\theta)$ is a cumulative distribution function,
$\nabla_{\theta} F(u;\theta) = 0$ at the boundary of $u \in \Omega$.
Then, the h-confidence leads to the marginal CD for $\theta$,
\begin{align*}
c(\theta;t)
= \int_{\Omega} c(\theta,u;t) du
&= \int_{\Omega} \{\nabla_{u} F(t|u)\}
\{\nabla_{\theta} F(u;\theta)\} du \\
&= - \int_{\Omega} F(t|u) \{\nabla_{\theta} f_{\theta}(u)\} du \\
&= - \int_{t' \leq t} \int_{\Omega} 
\{\nabla_{\theta} f_{\theta}(t', u)\} dt' du 
=\nabla_{\theta} P_{\theta}(T\geq t).
\end{align*}
Thus, the CIs and PIs from the h-confidence maintain the coverage probability.
\hfill \qed

\subsection{Proof of Theorem 7.1} 
\label{proof:h-ela}

\begin{lemma} \label{lem:ela}
Suppose that $(\widehat{\boldsymbol{\theta}},\widehat{\mathbf{v}})$
is a MHLE from the h-likelihood 
and $(\widehat{\boldsymbol{\theta}}_{B},\widehat{\mathbf{v}}_{B})$ 
is an approximate MHLE from the approximate h-likelihood,
\begin{equation}
h_{B}(\boldsymbol{\theta},\mathbf{v})
=\ell_{e}(\boldsymbol{\theta},\mathbf{v})
-\ell_{e}(\boldsymbol{\theta},\widetilde{\mathbf{v}})
+\log L_{B}(\boldsymbol{\theta};\mathbf{y}).
\end{equation}
Then, as $B\to \infty$, 
\begin{equation*}
(\widehat{\boldsymbol{\theta}}_{B},\widehat{\mathbf{v}}_{B})
\overset{p}{\to}
(\widehat{\boldsymbol{\theta}},\widehat{\mathbf{v}})
\quad \textrm{and} \quad
I_{B}(\widehat{\boldsymbol{\theta}}_{B},\widehat{\mathbf{v}}_{B})
\overset{p}{\to}
I(\widehat{\boldsymbol{\theta}},\widehat{\mathbf{v}}). 
\end{equation*}
\end{lemma}
\begin{proof}
\citet{han22b} showed that
$$
\frac{1}{B} \sum_{b=1}^{B}
\frac{L_e(\boldsymbol{\theta}, \mathbf{v}^{(b)})}
{q(\mathbf{v}^{(b)})}
\overset{p}{\to}
L(\boldsymbol{\theta};\mathbf{y})
\quad \textrm{and} \quad
I_{11}(\widehat{\boldsymbol{\theta}}_B)
\overset{p}{\to}
- \left[ \nabla_{\boldsymbol{\theta}}^2 \log f_{\boldsymbol{\theta}}(\mathbf{y}) \right]
_{\boldsymbol{\theta}=\widehat{\boldsymbol{\theta}}}.
$$
Since $
\log f_{\boldsymbol{\theta}}(\widetilde{\mathbf{v}}|\mathbf{y}) 
= \log f_{\boldsymbol{\theta}}(\widetilde{\mathbf{v}}|\mathbf{y}) 
+ \log f_{\boldsymbol{\theta}}(\mathbf{y})
- \log f_{\boldsymbol{\theta}}(\mathbf{y})
= \ell_e(\boldsymbol{\theta}, \widetilde{\mathbf{v}}) + \ell_m(\boldsymbol{\theta})$,
we have
$$
h_B(\boldsymbol{\theta}, \mathbf{v})
\overset{p}{\to}
\ell_e(\boldsymbol{\theta}, \mathbf{v})
- \ell_e(\boldsymbol{\theta}, \widetilde{\mathbf{v}})
+ \ell_m(\boldsymbol{\theta})
= \log f_{\boldsymbol{\theta}}(\mathbf{v}|\mathbf{y}) 
- \log f_{\boldsymbol{\theta}}(\widetilde{\mathbf{v}}|\mathbf{y})
+ \log f_{\boldsymbol{\theta}}(\mathbf{y})
= h(\boldsymbol{\theta}, \mathbf{v}).
$$
This leads to
$$
\widehat{\boldsymbol{\theta}}_{B}
= \argmax_{\boldsymbol{\theta}} h_B(\boldsymbol{\theta}, \widetilde{\mathbf{v}})
= \argmax_{\boldsymbol{\theta}} \left\{ 
\frac{1}{B} \sum_{b=1}^{B}
\frac{L_e(\boldsymbol{\theta}, \mathbf{v}^{(b)})}
{q(\mathbf{v}^{(b)})}
\right\}
\overset{p}{\to}
\argmax_{\boldsymbol{\theta}} L_m(\boldsymbol{\theta})
= \widehat{\boldsymbol{\theta}}
$$
and $
\widehat{\mathbf{v}}_B
= \widetilde{\mathbf{v}}(\widehat{\boldsymbol{\theta}}_B)
\overset{p}{\to}
\widetilde{\mathbf{v}}(\widehat{\boldsymbol{\theta}})
= \widehat{\mathbf{v}}.
$
Furthermore, convergence of 
$I_{11}(\widehat{\boldsymbol{\theta}}_B)$ leads to
\begin{align*}
\left[ - \nabla^2_{\boldsymbol{\theta}} \ell_e(\boldsymbol{\theta}, \mathbf{v})
+ \nabla^2_{\boldsymbol{\theta}} \ell_e(\boldsymbol{\theta}, \widetilde{\mathbf{v}})
- I_{11}(\boldsymbol{\theta}) \right]
_{\widehat{\boldsymbol{\theta}}_B, \widehat{\mathbf{v}}_B}
&\overset{p}{\to}
\left[ - \nabla^2_{\boldsymbol{\theta}}
\{ \ell_e(\boldsymbol{\theta}, \mathbf{v}) 
- \ell_e(\boldsymbol{\theta}, \widetilde{\mathbf{v}}) 
+ \ell_m(\boldsymbol{\theta})
\}
\right]_{\widehat{\boldsymbol{\theta}},\widehat{\mathbf{v}}} 
\\
&= - \left[ \nabla^2_{\boldsymbol{\theta}} h(\boldsymbol{\theta}, \mathbf{v})
\right]_{\widehat{\boldsymbol{\theta}}, \widehat{\mathbf{v}}}.
\end{align*}
Since $\nabla_{\mathbf{v}} h(\boldsymbol{\theta}, \mathbf{v})
= \nabla_{\mathbf{v}} \ell_e(\boldsymbol{\theta}, \mathbf{v})$,
we have
$
I_{B}(\widehat{\boldsymbol{\theta}}_{B},\widehat{\mathbf{v}}_{B})
\overset{p}{\to}
I(\widehat{\boldsymbol{\theta}},\widehat{\mathbf{v}}). 
$
\hfill
\end{proof}
To show Theorem 7.1,
it is enough to consider the derivatives with respect to $\boldsymbol{\theta} $ only,
because
$$
\nabla_{\mathbf{v}} h(\boldsymbol{\theta}, \mathbf{v}) 
=\nabla_{\mathbf{v}} \ell_e(\boldsymbol{\theta}, \mathbf{v})
=\nabla_{\mathbf{v}} h_B(\boldsymbol{\theta}, \mathbf{v}). 
$$
\citet{han22b} showed that 
$\widetilde{L}_B(\boldsymbol{\theta}) \stackrel{p}{\to} \argmax_{\boldsymbol{\theta}} L(\boldsymbol{\theta}; \mathbf{y}) $ 
and substituting $\mathbf{v}  = \widetilde{\mathbf{v}}$ leads to 
$h_B(\boldsymbol{\theta}, \widetilde{\mathbf{v}}) = \log \widetilde{L}_B(\boldsymbol{\theta})$. 
Thus, we have 
$$
\widehat{\boldsymbol{\theta}}_B 
= \argmax_{\boldsymbol{\theta}} h_B(\boldsymbol{\theta}, \widetilde{\mathbf{v}}) 
= \argmax_{\boldsymbol{\theta}} \widetilde{L}_B(\boldsymbol{\theta}) 
\stackrel{p}{\to} \argmax_{\boldsymbol{\theta}} L(\boldsymbol{\theta}; \mathbf{y}) 
= \widehat{\boldsymbol{\theta}}. 
$$
For the second derivatives, by using the theorem in \citet{han22b}, 
we can show that 
$$
I_{11}( \widehat{\boldsymbol{\theta}}_B) \stackrel{p}{\to} 
\left[ \nabla^2_{\boldsymbol{\theta}}  \ell_e(\boldsymbol{\theta}, \widetilde{\mathbf{v}})
- \nabla^2_{\boldsymbol{\theta}} \ell(\boldsymbol{\theta}; \mathbf{y}) \right]
_{\boldsymbol{\theta}  = \widehat{\boldsymbol{\theta}}}. 
$$
This leads to 
$$
\widetilde{I}_B (\widehat{\boldsymbol{\theta}}_B, \widehat{\mathbf{v}}_B) 
\stackrel{p}{\to} I(\widehat{\boldsymbol{\theta}}, \widehat{\mathbf{v}}), 
$$
and $\widehat{I}$ is a consistent estimator of $\mathcal{I}_{\boldsymbol{\theta}}$ 
from the Theorem 4.2. 
\hfill \qed

\end{appendix}

\end{document}